\documentclass[11pt]{article}
\usepackage{amsmath,amssymb,amsthm,amscd}

\newtheorem{theorem}{Theorem}[section]
\newtheorem{lemma}[theorem]{Lemma}
\newtheorem{proposition}[theorem]{Proposition}
\newtheorem{corollary}[theorem]{Corollary}
\newtheorem{example}[theorem]{Example}

\theoremstyle{plain}

\theoremstyle{definition}
\newtheorem{definition}[theorem]{Definition}
\newtheorem{remark}[theorem]{Remark}

\numberwithin{equation}{section}

\newcommand{\Ext}{\operatorname{Ext}}
\newcommand{\End}{\operatorname{End}}

\newcommand{\id}{\operatorname{id}}
\newcommand{\Ker}{\operatorname{Ker}}

\newcommand{\Ad}{\operatorname{Ad}}

\newcommand{\pr}{\operatorname{pr}}
\newcommand{\st}{\operatorname{st}}

\newcommand{\K}{\mathcal{K}}

\newcommand{\G}{\mathcal{G}}

\newcommand{\N}{\mathbb{N}}

\newcommand{\Z}{\mathbb{Z}}
\newcommand{\Zp}{{\mathbb{Z}}_+}

\def\OL{{\mathcal{O}}_{\frak L}}

\def\M{\mathcal{M}}
\def\SN{\mathcal{N}}
\def\A{\mathcal{A}}

\def\P{\mathcal{P}}
\def\Q{\mathcal{Q}}
\def\X{\mathcal{X}}
\def\Y{\mathcal{Y}}
\def\H{\mathcal{H}}
\def\K{\mathcal{K}}

\def\E{{\mathcal{E}}}

\def\OA{{{\mathcal{O}}_A}}

\def\DA{{{\mathcal{D}}_A}}

\def\Ext{{{\operatorname{Ext}}}}

\def\Ad{{{\operatorname{Ad}}}}

\def\det{{{\operatorname{det}}}}

\def\LGBS{({\frak L}^-, {\frak L}^+)}
\def\LLGBS{({\frak L}^-_\Lambda, {\frak L}^+_\Lambda)}

\def\OALMP{{{\mathcal{O}}_{{\frak L}^-}^+}}
\def\OALPM{{{\mathcal{O}}_{{\frak L}^+}^-}}

\def\ALM{{{\mathcal{A}}_{{\frak L}^-}}}
\def\ALP{{{\mathcal{A}}_{{\frak L}^+}}}

\def\DLMP{{{\mathcal{D}}_{{\frak L}^-}^+}}
\def\DLPM{{{\mathcal{D}}_{{\frak L}^+}^-}}
\def\Pev{{{\operatorname{P_{ev}}}}}

\title{ Subshifts, $\lambda$-graph bisystems and  $C^*$-algebras
\\
}
\author{Kengo Matsumoto \\
Department of Mathematics \\
Joetsu University of Education \\
Joetsu, 943-8512, Japan
}

\begin{document}
\maketitle

\date{}

\def\det{{{\operatorname{det}}}}

\begin{abstract} 
We introduce a notion of $\lambda$-graph bisystem
that consists of a pair $\LGBS$ of two labeled Bratteli diagrams 
${\frak L}^-, {\frak L}^+$
satisfying  certain compatibility condition for labeling their edges.
It is a two-sided extension of $\lambda$-graph system, 
that has been previously introduced by the author.
 Its matrix presentation is called a symbolic matrix bisystem.
We first show that  any $\lambda$-graph bisystem presents subshifts 
and conversely any subshift is presented by a  $\lambda$-graph bisystem,
called the canonical $\lambda$-graph bisystem for the subshift.
We introduce an algebraically defined relation on symbolic matrix bisystems
called properly strong shift equivalence 
and show that two subshifts are topologically conjugate 
if and only if their canonical  symbolic matrix bisystems 
are properly strong shift equivalent.
A $\lambda$-graph bisystem $({\frak L}^-, {\frak L}^+)$ yields 
 a pair of $C^*$-algebras 
written $\OALMP, \OALPM$ that are first defined as the $C^*$-algebras 
of certain \'etale groupoids constructed from 
$({\frak L}^-, {\frak L}^+).$ 
 We  study structure of the $C^*$-algebras, and show that 
they are universal unital unique $C^*$-algebras subject to
certain operator relations 
among canonical generators of partial isometries and projections
encoded by the structure of the  $\lambda$-graph bisystem $({\frak L}^-, {\frak L}^+).$ 
If a $\lambda$-graph bisystem comes from a $\lambda$-graph system 
of a finite directed graph, 
then the associated subshift is the two-sided topological Markov shift
$(\Lambda_A, \sigma_A)$
 by its transition  matrix $A$ of the graph,
 and the associated $C^*$-algebra $\OALMP$ is isomorphic to 
 the Cuntz--Krieger algebra ${\mathcal{O}}_A,$
whereas the other  $C^*$-algebra $\OALPM$ is isomorphic to 
the crossed product $C^*$-algebra 
$C(\Lambda_A)\rtimes_{\sigma_A^*}\mathbb{Z}$
of the commutative $C^*$-algebra $C(\Lambda_A)$ 
of continuous functions on the shift space $\Lambda_A$
of the two-sided topological Markov shift 
by the automorphism $\sigma_A^*$ induced by the homeomorphism of the shift 
$\sigma_A.$
This phenomena shows  a duality 
between Cuntz--Krieger algebra ${\mathcal{O}}_A$ 
and the crossed product $C^*$-algebra 
$C(\Lambda_A)\rtimes_{\sigma_A^*}\mathbb{Z}$.  
\end{abstract}



\newpage

Contents:

\begin{enumerate}
\renewcommand{\theenumi}{\arabic{enumi}}
\renewcommand{\labelenumi}{\textup{\theenumi}}
\item Introduction
\item Subshifts, $\lambda$-graph systems and its $C^*$-algebras
\item $\lambda$-graph bisystems 
\item Symbolic matrix bisystems
\item Subshifts and $\lambda$-graph bisystems
\item Strong shift equivalence  
\item \'Etale groupoids for $\lambda$-graph bisystems
and its $C^*$-algebras
\item Structure of the $C^*$-algebra $\OALMP$
\item K-groups for $\OALMP$
\item A duality: $\lambda$-graph systems as $\lambda$-graph bisystems
\end{enumerate}



\section{Introduction}
Cuntz--Krieger in \cite{CK} initiated an interplay between symbolic dynamics and $C^*$-algebras.
They constructed a purely infinite simple $C^*$-algebra $\OA$ 
called Cuntz--Krieger algebra from a topological Markov shift defined by a square matrix 
$A$ with entries in 
$\{0,1\}.$ 
They showed that the algebra $\OA$ is a universal unique $C^*$-algebra generated by 
a finite family of partial isometries subject to 
certain operator relations defined by the matrix $A$.
They also proved not only that the stable isomorphism classes of the resulting $C^*$-algebras are invariant under topological conjugacy of the underlying topological Markov shifts, 
but also that their K-groups and  $\Ext$-groups are realized as flow equivalence invariants of the Markov shifts.
Their pioneer work has given  big influence to both classification theory of $C^*$-algebras and interplay between symbolic dynamical systems and $C^*$-algebras.
 After their work, many generalizations of Cuntz--Krieger algebras
have come up from several view points (cf. 
\cite{Deaconu}, \cite{ExcelLaca}, \cite{Katsura}, \cite{KPRR},  \cite{Pim},  \dots).
  In \cite{MaIJM1997} (cf. \cite{CaMa}), 
  the author attempted to generalize Cuntz--Krieger algebras 
  defined from topological Markov shifts
  to $C^*$-algebras defined  from general subshifts.
After \cite{MaIJM1997}, 
he introduced a notion of $\lambda$-graph system written 
${\frak L}=(V, E,\lambda,\iota).$
It consists of a labeled Bratteli diagram $(V, E, \lambda)$
with its vertex set $V = \bigcup_{l=0}^\infty V_l$, 
edge set $E = \bigcup_{l=0}^\infty E_{l,l+1}$
and a labeling map
 $\lambda:E\longrightarrow\Sigma,$ 
together with a surjective map  
$\iota: V_{l+1}\longrightarrow V_l, l \in \Zp=\{0,1,\dots, \}.$
We require certain compatibility condition between the labeled Bratteli diagram
$(V, E, \lambda)$ and the map $\iota:V\longrightarrow V$,
called local property of $\lambda$-graph system.
He showed that any $\lambda$-graph system presents a subshift 
and conversely
any subshift can be presented by a $\lambda$-graph system, 
called the canonical $\lambda$-graph system.
He also introduced a notion of symbolic matrix system that is a matrix presentation of 
$\lambda$-graph system,
and defined some algebraic relations called (properly)  
strong shift equivalence in symbolic matrix systems.
He proved that
if two symbolic matrix systems are (properly) strong shift equivalent, 
then their presenting subshifts are  topologically conjugate.
Conversely if two subshifts are topologically conjugate, 
then  their canonically 
constructed symbolic matrix systems from the subshifts   
are (properly) strong shift equivalent (cf. \cite{MaETDS2003}).
 This result  generalizes a fundamental classification theorem
of topological Markov shifts proved by R. Williams \cite{Williams}.
A construction of $C^*$-algebra from a $\lambda$-graph system   
was presented in \cite{MaDocMath2002}.
The class of such $C^*$-algebras are generalization of Cuntz--Krieger algebras.
The resulting $C^*$-algebra was written  
${\mathcal{O}}_{\frak L}$
and whose K-theoretic groups were proved to be invariant 
under (properly) strong shift equivalence of  underlying symbolic dynamical systems,
and hence yield topological conjugacy invariants of general subshifts.
Especially, the K-groups $K_*({\mathcal{O}}_{{\frak L}_\Lambda})$ and the Ext-group $\Ext_*({\mathcal{O}}_{{\frak L}_\Lambda})$
 for the $C^*$-algebra ${\mathcal{O}}_{{\frak L}_\Lambda}$ 
of the canonical $\lambda$-graph system 
${{\frak L}_\Lambda}$ of a subshift $\Lambda$
was the first found computable invariant under flow equivalence 
of general subshifts (\cite{MaKTheory}).

As seen  in the construction of $\lambda$-graph system from subshifts in \cite{MaDocMath1999},
it is essentially due to its (right) one-sided structure of the subshifts.
Hence the resulting $C^*$-algebra $\mathcal{O}_{\frak L}$
do not exactly reflect two-sided dynamics.
 In this paper, we will attempt to construct two-sided extension of $\lambda$-graph systems, construct associated $C^*$-algebras
and study their structure.
We will introduce a notion of $\lambda$-graph bisystem over a finite alphabet.
It is a pair of two labeled Bratteli diagrams ${\frak L}^-, {\frak L}^+$
over alphabets $\Sigma^-, \Sigma^+$, respectively,
and satisfy certain compatible condition of their edge labeling, 
called local property of $\lambda$-graph bisystem, 
where two alphabet sets  $\Sigma^-, \Sigma^+$ are not related in general.
The two labeled Bratteli diagrams are of the form
${\frak L}^- =(V, E^-,\lambda^-),
 {\frak L}^+ = (V,E^+, \lambda^+).$
They have common vertex sets
$V = \bigcup_{l=0}^\infty V_l$
together with 
edge sets
$E^- = \bigcup_{l=0}^\infty E^-_{l+1,l}$
and
$E^+ = \bigcup_{l=0}^\infty E^+_{l,l+1}$
 and labeling maps
$
\lambda^-: E^-\longrightarrow \Sigma^-,
\lambda^+: E^+\longrightarrow \Sigma^+,
$
respectively.
Its matrix presentation is called a symbolic matrix bisystem written
$(\M^-, \M^+).$
It is a pair $(\M_{l,l+1}^-, \M_{l,l+1}^+)_{l \in \Zp}$
of sequences of rectangular matrices 
such that 
$\M_{l,l+1}^-, \M_{l,l+1}^+$
are symbolic matrices over $\Sigma^-, \Sigma^+,$
respectively
satisfying the following commutation relations 
corresponding to the local property of $\lambda$-graph bisystem :
\begin{equation}
\M_{l,l+1}^- \M_{l+1,l+2}^+ \overset{\kappa}{\simeq}\M_{l,l+1}^+ \M_{l+1,l+2}^-,
\qquad l \in \Zp, \label{eq:MMlocal}
\end{equation}
where
$\overset{\kappa}{\simeq}$
denotes the equality through exchanging
symbols 
$
\kappa: \beta\cdot \alpha \in \Sigma^-\cdot \Sigma^+ 
\longrightarrow \alpha\cdot \beta \in \Sigma^+\cdot \Sigma^-.
$
The notions of $\lambda$-graph bisystem and symbolic matrix bisystem are not only two-sided extensions of 
the preceding $\lambda$-graph system and symbolic matrix system,
respectively,
but also generalization of them, respectively.
We will first show that  any $\lambda$-graph bisystem presents two subshifts.
One is the subshift presented by the labeled Bratteli diagram ${\frak L}^-,$
and the other one is the one presented  the labeled Bratteli diagram ${\frak L}^+.$
If a $\lambda$-graph bisystem satisfies a particular condition on edge labeling called 
FPCC (Follower and Predecessor Compatibility Condition),
then the two presented subshifts coincide.
Conversely any subshift is presented by a  $\lambda$-graph bisystem satisfying FPCC,
called the canonical $\lambda$-graph bisystem for the subshift (Proposition \ref{prop:5.4}).
We will introduce a notion of properly strong shift equivalence in 
symbolic matrix bisystems satisfying FPCC, and prove the following theorem.

\begin{theorem}[{Theorem \ref{thm:DM4.2}}]
Two  subshifts are topologically conjugate
if and only if  
their canonical symbolic matrix bisystems
are properly strong shift equivalent.
\end{theorem}
The proof of the only if part of the theorem is harder than that of the if part.
To prove the only if part, 
we basically follow the idea given in the proof of \cite[Theorem 4.2]{MaDocMath1999},
and provide a notion of bipartite $\lambda$-graph bisystem as well as bipartite symbolic matrix bisystem.
We will show that the canonical $\lambda$-graph bisystems, whose presenting subshifts are topologically conjugate, are connected by a finite chain of bipartite $\lambda$-graph bisystems. 
We will also
introduce the notion of strong shift equivalence  
in general symbolic matrix bisystems.
Properly strong shift equivalence in symbolic matrix bisystems 
satisfying FPCC
implies strong shift equivalence.

We will construct a pair of   $C^*$-algebra 
written $\OALMP, \OALPM$ from a $\lambda$-graph bisystem 
$({\frak L}^-, {\frak L}^+).$ 
The two algebras 
$\OALMP, \OALPM$ are symmetrically constructed as the $C^*$-algebras 
of certain \'etale groupoids
$\G_{{\frak L}^-}^+, \G_{{\frak L}^+}^-.$
The groupoids
$\G_{{\frak L}^-}^+, \G_{{\frak L}^+}^-$
are Deaconu--Renault groupoids for certain shift dynamical systems
$(X_{{\frak L}^-}^+, \sigma_{{\frak L}^-}), (X_{{\frak L}^+}^-, \sigma_{{\frak L}^+})$ 
associated with the $\lambda$-graph bisystem 
$({\frak L}^-, {\frak L}^+).$ 
They are also regarded as 
a generalization of the \'etale groupoids
constructed from  $\lambda$-graph systems in \cite{MaDocMath2002}. 
We will introduce a notion of $\sigma_{{\frak L}^-}$-condition (I) 
for $\lambda$-graph bisystem
that guarantees the \'etale groupoid $\G_{{\frak L}^-}^+$ being essentially principal
and uniqueness of the $C^*$-algebra $\OALMP$ subject to the operator relations
$\LGBS$ below.  
Let $\{v_1^l,\dots, v_{m(l)}^l\}$ be the vertex set $V_l.$
For an edge $e^- \in E_{l+1,l}^-,$
its source vertex and terminal vertex
are denoted by
$s(e^-)\in V_{l+1}$ and $t(e^-)\in V_{l}$, respectively. 
For an edge $e^+ \in E_{l,l+1}^+,$
$s(e^+)\in V_{l}, t(e^+)\in V_{l+1}$
are similarly defined.
The directions of edges in ${\frak L}^-$ 
are upward, whereas those of edges in ${\frak L}^+$
are downward.
The transition matrices $A^-_{l,l+1}, A^+_{l,l+1}$
for  ${\frak L}^-, {\frak L}^+$
are defined by setting
\begin{align*}
A_{l,l+1}^-(i,\beta,j)
 & =
{\begin{cases}
1 &  
    \text{ if } \ t(e^-) = v_i^{l}, \lambda^-(e^-) = \beta,
                       s(e^-) = v_j^{l+1} 
    \text{ for some }    e^- \in E_{l+1,l}^-, \\
0           & \text{ otherwise,}
\end{cases}} \\
A_{l,l+1}^+(i,\alpha,j)
 & =
{\begin{cases}
1 &  
    \text{ if } \ s(e^+) = v_i^{l}, \lambda^+(e^+) = \alpha,
                       t(e^+) = v_j^{l+1} 
    \text{ for some }    e^+ \in E_{l,l+1}^+, \\
0           & \text{ otherwise,}
\end{cases}} 
\end{align*}
for
$
i=1,2,\dots,m(l),\ j=1,2,\dots,m(l+1),  \, \beta\in \Sigma^-, \, \alpha \in \Sigma^+.$

We will prove the following theorem, that is one of main results of the paper. 
\begin{theorem}[{Theorem \ref{thm:themain6}}]
Suppose that a $\lambda$-graph bisystem $\LGBS$ 
satisfies $\sigma_{{\frak L}^-}$-condition (I).
Then the $C^*$-algebra $\OALMP$
is  the universal unital unique $C^*$-algebra
generated by
partial isometries
$S_{\alpha}$ indexed by symbols $\alpha \in \Sigma^+$
and mutually commuting projections
$E_i^{l}(\beta)$ indexed by vertices
$v_i^l \in  V_l$ and symbols $\beta \in \Sigma^-$ 
with
$\beta \in \Sigma^-_1(v_i^l)$
 subject to the  following operator relations called $\LGBS$:
\begin{align}
\sum_{\alpha \in \Sigma^+} S_{\alpha}S_{\alpha}^*  
 & = \sum_{i=1}^{m(l)} \sum_{\beta \in \Sigma^-_1(v_i^l)}  
E_i^{l}(\beta) = 1, \label{eq:RMP1}\\ 
 S_\alpha S_\alpha^* & E_i^{l}(\beta)   =   E_i^{l}(\beta) S_\alpha S_\alpha^*, 
\label{eq:RMP2}\\
 \sum_{\beta \in \Sigma_1^-(v_i^l)} E_i^{l} (\beta)   
& =  \sum_{j=1}^{m(l+1)}\sum_{\gamma \in \Sigma_1^-(v_j^{l+1})}
A_{l,l+1}^-(i,\gamma,j)E_j^{l+1}(\gamma),  \label{eq:RMP3}\\
  S_{\alpha}^*E_i^{l}(\beta) S_{\alpha} &  =  
\sum_{j=1}^{m(l+1)} A_{l,l+1}^+(i,\alpha,j)E_j^{l+1}(\beta), \label{eq:RMP4}
\end{align}
where
$\Sigma_1^-(v_i^l) =
\{ \lambda^-(e^-) \in \Sigma^- \mid  
e^- \in E_{l,l-1}^- \text{ such that } s(e^-)  = v_i^l \}$
for $v_i^l \in V_l.$
\end{theorem}
For the other $C^*$-algebra 
$\OALPM$,
we have a symmetric structure theorem to the above Theorem.

$\lambda$-graph systems are typical examples of $\lambda$-graph bisystems.
If a $\lambda$-graph bisystem $\LGBS$
 comes from a $\lambda$-graph system ${\frak L},$
then the $C^*$-algebra
$\OALMP$ coincides with the $C^*$-algebra
${\mathcal{O}}_{\frak L}$ 
of a $\lambda$-graph system ${\frak L}$
previously studied in \cite{MaDocMath2002}.

For a $\lambda$-graph bisystem
$\LGBS,$
let us denote by $\Omega_{{\frak L}^-}$
the compact Hausdorff space of path spaces of the labeled Bratteli diagram
${\frak L}^-.$
By the local property of $\lambda$-graph bisystem,
the edges labeled symbols $\alpha \in \Sigma^+$ of the other labeled Bratteli diagram ${\frak L}^+$
give rise to an endomorphism on the abelian group
$C(\Omega_{{\frak L}^-}, \Z)$ of $\Z$-valued continuous functions on 
$\Omega_{{\frak L}^-},$
 that is denoted by $\lambda_{{\frak L}^-*}^+.$
Then we have a $K$-theory formulas for the $C^*$-algebra
$\OALMP.$
\begin{theorem}[{Theorem \ref{thm:Ktheory}}]
\begin{align*}
   K_0(\OALMP) 
\cong &
C(\Omega_{{\frak L}^-}, \Z) / 
(\id -\lambda_{{\frak L}^-*}^+) C(\Omega_{{\frak L}^-}, \Z), \\
  K_1(\OALMP) 
\cong & 
\Ker(\id -\lambda_{{\frak L}^-*}^+) \text{ in } C(\Omega_{{\frak L}^-}, \Z).
\end{align*}
\end{theorem}
Similar K-theory formulas for 
the other $C^*$-algebra $\OALPM$ hold.
Since the properly strong shift equivalence class of 
the canonical symbolic matrix bisystem for a subshift is 
invariant under topological conjugacy of the subshift,
the following result tells us that 
the above K-groups $K_i(\OALMP), i=0,1 $
yield topological conjugacy invariants of subshifts.
\begin{theorem}[{Theorem \ref{thm:MoritaKtheory}}]
Let $(\M^-,\M^+)$ and $(\SN^-,\SN^+)$
be symbolic matrix bisystems.
Let 
$({\frak L}_\M^-,{\frak L}_\M^+)$ and 
$({\frak L}_\SN^-,{\frak L}_\SN^+)$
be the associated $\lambda$-graph bisystems both of which satisfy FPCC.
Suppose that 
 $(\M^-,\M^+)$ and $(\SN^-,\SN^+)$
are properly strong shift equivalent.
Then 
the $C^*$-algebras
${{\mathcal{O}}_{{\frak L}_\M^-}^+}$ and 
${{\mathcal{O}}_{{\frak L}_\SN^-}^+}$
are Morita equivalent, so that their K-groups
$K_i({{\mathcal{O}}_{{\frak L}_\M^-}^+})$ and 
$K_i({{\mathcal{O}}_{{\frak L}_\SN^-}^+})$
are isomorphic for $i=0,1.$
\end{theorem}

\begin{corollary}[{Corollary \ref{cor:13.4}}]
The K-groups $K_i({\mathcal{O}}_{{\frak L}_\Lambda^-}^+), i=0,1$
of the $C^*$-algebra
${\mathcal{O}}_{{\frak L}_\Lambda^-}^+$
of the canonical $\lambda$-graph bisystem
$({\frak L}_\Lambda^-, {\frak L}_\Lambda^+)$
of a subshift $\Lambda$ is invariant under topological conjugacy of subshifts. 
\end{corollary}

Let ${\frak L} =(V, E, \lambda,\iota)$ be a $\lambda$-graph system over $\Sigma.$
Put $\Sigma^+ = \Sigma.$ 
Let ${\frak L}^+$ be the original labeled Bratteli diagram $\frak L$ without the map 
$\iota: V \longrightarrow V.$ 
Define a new alphabet
$\Sigma^- =  \{\iota\}.$
The other 
labeled Bratteli diagram
${\frak L}^- = (V, E^-, \lambda^-)$ over alphabet $\Sigma^-$
is defined in the following way.
Define an edge $e^-\in E^-_{l+1,l}$
if
$\iota(v_j^{l+1}) = v_i^l$ 
so that 
$s(e^-) = v_j^{l+1}, t(e^-) = v_i^l$
and 
$\lambda^-(e^-) =\iota \in \Sigma^-.$ 
Then we have a labeled Bratteli diagram
${\frak L}^- = (V, E^-, \lambda^-)$ over alphabet $\Sigma^-.$
Then the local property of the $\lambda$-graph system  ${\frak L}$ makes the pair 
$\LGBS$ a $\lambda$-graph bisystem.
Hence we have a $\lambda$-graph bisystem from a $\lambda$-graph system.
Let $X_{{\frak L}^+}^-$
be the unit space of the \'etale groupoid ${\mathcal{G}}_{{\frak L}^+}^-.$
The $\iota$-map induces a shift homeomorphism on $X_{{\frak L}^+}^-$
denoted by $\sigma_{{\frak L}^+}$.
It yields an automorphism written $\sigma_{{\frak L}^+}^*$
on the commutative $C^*$-algebra $C(X_{{\frak L}^+}^-)$
of continuous functions on $X_{{\frak L}^+}^-$.
Then we see the following.
\begin{proposition}[{Proposition \ref{prop:lambdaalgebras}}]
Let $\frak L$ be a left-resolving $\lambda$-graph system over $\Sigma.$
Let $\LGBS$ be the associated $\lambda$-graph bisystem.
Then we have 
\begin{enumerate}
\renewcommand{\theenumi}{\roman{enumi}}
\renewcommand{\labelenumi}{\textup{(\theenumi)}}
\item
The $C^*$-algebra $\OALMP$ is canonically isomorphic to 
the $C^*$-algebra ${\mathcal O}_{\frak L}$ 
of the original $\lambda$-graph system  ${\frak L}.$ 
\item
The $C^*$-algebra $\OALPM$ is isomorphic to  the crossed product
$ C(X_{{\frak L}^+}^-)\rtimes_{\sigma_{{\frak L}^+}^*}\Z.$ 
\end{enumerate}
\end{proposition}
An irreducible finite directed graph naturally gives rise to a $\lambda$-graph system.
Let $A$ be its transition matrix for a given  finite directed  graph.
We then have the two-sided 
topological Markov shift
$(\Lambda_A, \sigma_A)$ for the matrix $A.$
We denote by ${{\frak L}_A}$ the associated $\lambda$-graph system for the finite directed graph.
The  $\lambda$-graph bisystem from ${{\frak L}_A}$ 
obtained by the above procedure is 
$({\frak L}_A^-,{\frak L}_A^+).$  
As a corollary of the above proposition we have 
\begin{corollary}[{Corollary \ref{cor:duality}}]
The $C^*$-algebra ${\mathcal{O}}^+_{{\frak L}_A^-}$
is isomorphic to the Cuntz--Krieger algebra $\OA$, 
whereas the other $C^*$-algebra
${\mathcal{O}}^-_{{\frak L}_A^+}$
is isomorphic to the $C^*$-algebra of  the crossed product 
$C(\Lambda_A) \rtimes_{\sigma_A^*}\Z$
of the commutative $C^*$-algebra  
$C(\Lambda_A) $ by the automorphism 
$\sigma_A^*$
induced by the homeomorphism $\sigma_A$
of the shift on $\Lambda_A.$
\end{corollary}
The above corollary suggests us that the Cuntz--Krieger algebra 
$\OA$ and the crossed product $C^*$-algebra $C(\Lambda_A) \rtimes_{\sigma_A^*}\Z$
have a relation like a ``duality'' pair.

We finally refer to the transpose of $\lambda$-graph bisystems and its $C^*$-algebras.
\begin{proposition}
For a $\lambda$-graph bisystem $\LGBS$, 
denote by 
${{\frak L}^{-t}}$ (resp. ${{\frak L}^{+t}}$) the labeled Bratteli diagram obtained by reversing the directions of all edges in ${{\frak L}^-}$ (resp. ${{\frak L}^+}$). 
Then the pair
$({{\frak L}^{+t}}, {{\frak L}^{-t}})$ becomes a $\lambda$-graph bisystem.
We then have a canonical isomorphisms of $C^*$-algebras:
\begin{equation*}
\OALMP \cong {{\mathcal{O}}_{{{\frak L}^{-t}}}^-}, \qquad
\OALPM \cong {{\mathcal{O}}_{{{\frak L}^{+t}}}^+}. 
\end{equation*}
\end{proposition}


The paper is organized in the following way:

Section 1 is Introduction in which we describe a brief survey of the paper.

 In Section 2, we review $\lambda$-graph systems, symbolic matrix systems and their $C^*$-algebras. 
The operator relations among the canonical generating partial isometries and projections in the $C^*$-algebra ${\mathcal{O}}_{\frak L}$ associated with a $\lambda$-graph system ${\frak L}$ are described. 

 In Section 3, 
we introduce a new notion of $\lambda$-graph bisystem, that is a main target of the paper. It is a generalization of $\lambda$-graph system surveyed in the preceding section.
Several examples of $\lambda$-graph bisystems are presented.
 
In Section 4, a matrix presentation of a $\lambda$-graph bisystem is introduced.
It is called a symbolic matrix bisystem, that is also a generalization of symbolic matrix system surveyed in Section 2.

In Section 5, it is shown that for any subshift, there exists a $\lambda$-graph bisystem
satisfying FPCC and presenting the subshift. 
The $\lambda$-graph bisystem is called the canonical  $\lambda$-graph bisystem for the subshift. 

In Section 6, properly strong shift equivalence in symbolic matrix  bisystems satisfying FPCC is introduced.
It is proved that 
two subshifts are topologically conjugate if and only if 
their canonical symbolic matrix bisystems  are properly strong shift equivalent.
We also introduce a notion of strong shift equivalence in general symbolic matrix bisystems.

In Section 7, two \'etale groupoids 
${\mathcal{G}}^+_{{\frak L}^-}, {\mathcal{G}}^-_{{\frak L}^+}$
are introduced as the Deaconu--Renault groupoids constructed  from cerain shift dynamical systems associated with continuous graphs in the sense of Deaconu \cite{Deaconu3}
from a $\lambda$-graph bisystem $\LGBS$. 
We then define the $C^*$-algebras $\OALMP, \OALPM$ as their groupoid $C^*$-algebras 
$C^*({\mathcal{G}}^+_{{\frak L}^-}), C^*({\mathcal{G}}^-_{{\frak L}^+}),$
respectively.

In Section 8,  condition (I) on a $\lambda$-graph bisystem $\LGBS$ is introduced.
Under the condition (I), 
the $C^*$-algebra $\OALMP$ as well as $\OALPM$
is realized as a universal unital unique $C^*$-algebra 
generated by partial isometries and projections subject to 
certain operator relations encoded by the structure of the 
$\lambda$-graph bisystem $\LGBS.$ 
It is one of the main results of the paper.

 In Section 9, 
K-theory formulas of the $C^*$-algebras $\OALMP, \OALPM$  are presented.
It is shown that if two symbolic matrix bisystems satisfying FPCC are properly strong shift equivalent, then the $C^*$-algebras associated with the $\lambda$-graph bisystems of the symbolic matrix bisystems are Morita equivalent, so that their K-theory groups yield topological conjugacy invariants of subshifts.

 In Section 10,
the two $C^*$-algebras $\OALMP, \OALPM$ for $\lambda$-graph bisystems 
coming from $\lambda$-graph systems are studied.
Let $\LGBS$ be the $\lambda$-graph bisystem defined by a $\lambda$-graph system
 ${\frak L}.$ 
 It is proved that  the $C^*$-algebra 
$\OALMP$ is isomorphic to the $C^*$-algebra 
${\mathcal{O}}_{\frak L}$ of the $\lambda$-graph system ${\frak L}$ in 
Section 2, and the other $C^*$-algebra 
$\OALPM$ is isomorphic to the crossed product 
$C(X^-_{{\frak L}^+}) \rtimes_{\sigma_{{\frak L}^+}^*}\Z$
of the commutative $C^*$-algebra on the unit space
$X^-_{{\frak L}^+}$ of the groupoid ${\mathcal{G}}^-_{{\frak L}^+}$
by the homeomorphism of the shift $\sigma_{{\frak L}^+}.$
In particular, we know that 
the  Cuntz--Krieger algebra $\OA$ for a finite nonnegative matrix $A$ and 
the $C^*$-algebra of the crossed product $C(\Lambda_A)\rtimes_{\sigma_A^*}\Z$
of the two-sided topological Markov shift $\Lambda_A$ 
by the homeomorphism of the shift
are regarded as a duality pair.
 
\medskip

Throughout the paper,
the notation $\N, \Zp$ will denote the set of positive integers, the set of nonnegative integers,
respectively.
By a nonnegative matrix we mean a finite rectangular matrix with entries in nonnegative integers.
 

\section{Subshifts, $\lambda$-graph systems and its $C^*$-algebra}
Let $\Sigma$ be a finite set, which we call an alphabet. 
Each element of $\Sigma$ is called a symbol or a label.
Denote by $\Sigma^{\Z}$ the set of bi-infinite sequences
$(x_n)_{n \in \Z}$ of elements of $\Sigma.$
We endow $\Sigma^{\Z}$ with the infinite  product topology,
so that it is a compact Hausdorff space.
Let us denote by 
$\sigma: \Sigma^{\Z} \longrightarrow \Sigma^{\Z}$
the homeomorphism defined by the left shift
$\sigma((x_n)_{n\in \Z}) =(x_{n+1})_{n\in \Z}.$
Let $\Lambda \subset \Sigma^\Z$ 
be a  $\sigma$-invariant closed subset, that is $\sigma(\Lambda) =\Lambda.$
Then the topological dynamical system
$(\Lambda, \sigma)$ 
is called a subshift over $\Sigma,$
and the space $\Lambda$ is called the shift space for $(\Lambda, \sigma).$ 
We often write a subshift $(\Lambda, \sigma)$ as $\Lambda$ for short. 
For a subshift $\Lambda$ and $n\in \Zp$,
let us denote by $B_n(\Lambda)$
the set of admissible words in $\Lambda$ with length $n$, 
that is defined by
$B_n(\Lambda) = \{(x_1,\dots,x_n) \in \Sigma^n \mid (x_i)_{i\in \Z}\in \Lambda \}.$

For $N = |\Sigma|$, 
 the subshift
$(\Sigma^\Z, \sigma)$  is called the full $N$-shift.
More generally for an $N\times N$ matrix $A = [A(i,j)]_{i,j=1}^N$
with entries $A(i,j)$ in $\{0,1\},$
the subshift $\Lambda_A$ 
defined by
\begin{equation}
\Lambda_A = \{ (x_n)_{n \in \Z} \in \{1,2,\dots, N\}^\Z \mid
A(x_n, x_{n+1}) =1 \text{ for all } n \in \Z \} \label{eq:LambdaA}
\end{equation}    
is called the topological Markov shift defined by the matrix $A$. 
A topological Markov shift is often called a shift of finite type or simply SFT.
The class of sofic shifts are a generalized class containing shifts of finite type.
Let $\mathcal{G} = (\mathcal{V}, \mathcal{E},\lambda)$
be a finite labeled directed graph with
vertex set $\mathcal{V},$
edge set $\mathcal{E}$
and labeling map 
$\lambda: \mathcal{E}\longrightarrow \Sigma.$
For $n \in \N,$ let 
$$
B_n({\mathcal{G}}) =\{ (\lambda(e_1),\dots,\lambda(e_n)) \in \Sigma^n \mid
e_i \in \mathcal{E}, t(e_i) =s(e_{i+1}), i=1,2,\dots, n-1\} $$ 
be the set of words appearing in the labeled graph ${\mathcal{G}}$,
where
$t(e_i)$ denote the terminal vertex of $e_i$ and
$s(e_{i+1})$ denote the source vertex of $e_{i+1}.$ 
Then the sofic shift $\Lambda_{\mathcal{G}}$ for the labeled graph 
${\mathcal{G}}$ is defined by
\begin{equation*}
\Lambda_{\mathcal{G}} = \{ (x_n)_{n \in \Z} \in \Sigma^\Z \mid
(x_{n+1},\dots, x_{n+k}) \in B_k({\mathcal{G}}) \text{ for all } k \in \N, \, n \in \Z \}
\end{equation*}    
(\cite{Fis}, \cite{We}).
A labeled graph ${\mathcal{G}}$ 
is said to be left-resolving (resp. right-resolving),
if $\lambda(e) = \lambda(f)$ implies $t(e) \ne t(f)$ (resp. $s(e) \ne s(f)$).
It is well-known that  any sofic shift may be presented by a left-resolving  labeled graph
(cf. \cite{Kr84}, \cite{Kr87}, \cite{LM}). 
It is also  presented by a right-resolving  labeled graph.   
 There are lots of non-sofic subshifts, for example, Dyck shifts, $\beta$-shifts, 
substitution subshifts, etc. (cf. \cite{LM}). 
Non-sofic shifts can not be presented by any finite labeled graphs.

A $\lambda$-graph system is a graphical object to present general subshifts
(\cite{MaDocMath1999}).
The idea defining it is basically due to a notion coming from operator algebras,
called Bratteli diagram (cf. \cite{Bratteli}). 
Let ${\frak L} =(V,E,\lambda,\iota)$ 
be a $\lambda$-graph system over $\Sigma$ with vertex set
$
V = \bigcup_{l=0}^\infty V_{l}
$
and  edge set
$
E = \bigcup_{l=0}^\infty E_{l,l+1}
$
that is labeled with symbols in $\Sigma$ by $\lambda: E \rightarrow \Sigma$, 
and that is supplied with a surjective map
$
\iota( = \iota_{l,l+1}):V_{l+1} \rightarrow V_l
$
for each
$
l \in  \Zp.
$
Here the vertex sets $V_{l},l \in \Zp$
are finite disjoint sets,
as well as   
$E_{l,l+1},l \in \Zp$
are finite disjoint sets.
An edge $e$ in $E_{l,l+1}$ has its source vertex $s(e)$ in $V_{l}$ 
and its terminal  vertex $t(e)$ 
in
$V_{l+1}$
respectively.
Every vertex in $V$ has a successor and  every 
vertex in $V_l$ for $l\in \N$ has a predecessor. 
It is then required for definition of $\lambda$-graph system that there exists an edge in $E_{l+1,l+2}$
with label $\alpha$ and its terminal is  $v \in V_{l+2}$
 if and only if 
 there exists an edge in $E_{l,l+1}$
with label $\alpha$ and its terminal is $\iota(v) \in V_{l+1}.$
For 
$u \in V_{l}$ and
$v \in V_{l+2},$
we put
\begin{align*}
E^{\iota}(u, v)
& = \{e \in E_{l+1,l+2} \ | \ \iota(s(e)) = u, \, t(e) = v \},\\
E_{\iota}(u, v)
& = \{e \in E_{l,l+1} \ | \ s(e) = u, \, t(e) = \iota(v) \}.
\end{align*}
As a key hypothesis for ${\frak L}$ to be a $\lambda$-graph system, 
we require the condition 
that there exists a bijective correspondence between 
$
E^{\iota}(u, v)
$
and
$
E_{\iota}(u, v)
$
that preserves labels
for each pair $(u,v) \in V_{l}\times V_{l+2}$ 
of vertices.
We call this property {\it the local property of $\lambda$-graph system}. 
For a $\lambda$-graph system ${\frak L},$
let
$W_{{\frak L}}$ be the set of finite label sequences appearing 
as concatenating finite labeled paths in ${\frak L}.$
Then there exists a unique subshift $\Lambda_{{\frak L}}$
whose admissible words 
$B_*(\Lambda_{{\frak L}})=\bigcup_{n=0}^\infty B_n(\Lambda_{{\frak L}})$
coincide with $W_{{\frak L}}$.
The subshift
$\Lambda_{\frak L}$ is called the subshift presented by 
${\frak L}.$
Conversely, we have a canonical method to construct a $\lambda$-graph system 
${\frak L}_\Lambda$ from an arbitrary subshift $\Lambda$
(\cite{MaDocMath1999}).
The $\lambda$-graph system is called the canonical $\lambda$-graph system for the subshift.

A $\lambda$-graph system has its matrix presentation, 
that is called a symbolic matrix system denoted by $(I,\M).$
In \cite{MaDocMath1999}, the notation $(\M,I)$ has been used for symbolic matrix system.
In this paper, the notation  $(I,\M)$ will be used instead.
  For a $\lambda$-graph system ${\frak L} =(V,E,\lambda,\iota)$ over $\Sigma,$
we define its {\it transition matrix system}
$(I_{l,l+1},A_{l,l+1})_{l\in \Zp}$ by setting
\begin{align}
I_{l,l+1}(i,j)
 & =
\begin{cases}
1 &  
    \text{ if } \ \iota_{l,l+1}(v_j^{l+1}) = v_i^l, \\
0           & \text{ otherwise}
\end{cases} \label{eq:tmsI} \\
A_{l,l+1}(i,\alpha,j)
 & =
\begin{cases}
1 &  
    \text{ if } \ s(e) = v_i^l, \lambda(e) = \alpha,
                       t(e) = v_j^{l+1} 
    \text{ for some }    e \in E_{l,l+1}, \\
0           & \text{ otherwise,}
\end{cases} \label{eq:tmsA}
\end{align}
for
$
i=1,2,\dots,m(l),\ j=1,2,\dots,m(l+1), \ \alpha \in \Sigma.
$ 
For $l \in \Zp$ and $i=1,2,\dots,m(l), \, j=1,2,\dots, m(l+1),$
we define 
$\M_{l,l+1}(i,j) = \alpha_1+ \cdots + \alpha_n$ if 
$A_{l,l+1}(i, \alpha_k, j) =1$ for $k=1,\dots, n$. 
That is $\M_{l,l+1}(i,j) = \alpha_1+ \cdots + \alpha_n$ 
if and only if
there exist labeled edges from $v_i^l$ to $v_j^{l+1}$ labeled 
$\alpha_1, \dots, \alpha_n.$
By the local property of $\lambda$-graph system, the matrix equations
\begin{equation}
\M_{l,l+1} I_{l+1,l+2} = I_{l,l+1} \M_{l+1,l+2}, \qquad l \in \Zp 
\label{eq:sms}
\end{equation} 
hold, where in the above equality
 $\alpha\cdot 1$ and $1\cdot \alpha$ for $\alpha \in \Sigma$
are identified with each other, and also both
 $\alpha\cdot 0$ and $0\cdot \alpha$ are recognized as $0$. 
The sequence 
$(I_{l,l+1},\M_{l,l+1}), l\in \Zp$ of pairs of matrices 
$I_{l,l+1}, \M_{l,l+1}$ 
is called the symbolic matrix system associated to the 
$\lambda$-graph system ${\frak L}$.
Conversely     
a sequence $(I_{l,l+1}, \M_{l,l+1}), l\in \Zp$ of pairs
of symbolic matrices $\M_{l,l+1}$ over alphabet $\Sigma$
and matrices $I_{l,l+1}$ over $\{0,1\}$ satisfying 
\eqref{eq:sms} gives rise to a $\lambda$-graph system over $\Sigma$,
and hence a subshift. 
 The sequence $(I_{l,l+1}, \M_{l,l+1}), l\in \Zp$ is written $(I, \M),$
or $(I_\M, \M).$

In \cite{MaDocMath1999} (cf. \cite{MaETDS2003}),
the author introduced a notion of strong shift equivalence in symbolic matrix systems.
 It has been proved that if two symbolic matrix systems are strong shift equivalent,
then the presenting subshifts are topologically conjugate.
Conversely if two subshifts are topologically conjugate, then the canonically constructed symbolic matrix systems are strong shift equivalent.
Therefore classification of subshifts are completely deduced to the classification of symbolic matrix systems up to strong shift equivalence. 
This result is a generalization of the fundamental classification theory
of topological Markov shifts  by R. Williams (\cite{Williams}, see \cite{Nasu} for sofic case).
 The author in \cite{MaDocMath1999} 
also introduced  notions of 
 K-groups and Bowen--Franks groups for symbolic matrix systems and hence for  
 $\lambda$-graph systems, and proved that they are all strong shift equivalence invariants of symbolic matrix systems. 
Hence these invariants give rise to topological conjugacy invariants of general subshifts. 
These invariants are regarded as K-theoretic invariants of the associated 
$C^*$-algebras.

A $\lambda$-graph system ${\frak L}$ 
is said to be left-resolving if $e,f \in E$ with
$t(e) = t(f)$ and $\lambda(e) = \lambda(f)$ implies $e=f$.
In what follows all $\lambda$-graph systems are assumed to be left-resolving. 
Let us denote by $\{ v_1^l, \dots, v_{m(l)}^l\}$ the set $V_l$ of vertices at level $l.$ 
The author in \cite{MaDocMath2002} introduced 
a $C^*$-algebra $\OL$ associated to the $\lambda$-graph system as a generalization of Cuntz--Krieger algebras.
The $C^*$-algebra $\OL$ was first constructed as a $C^*$-algebra 
$C^*({\mathcal{G}}_{\frak L})$ of an \'etale groupoid  
${\mathcal{G}}_{\frak L}$ associated to ${\frak L}$.
It is realized as a universal $C^*$-algebra in the following way.
\begin{theorem}[{\cite[Theorem A]{MaDocMath2002}}] \label{thm:lambdagraphC*}
Let ${\frak L}$ be a left-resolving $\lambda$-graph system over alphabet $\Sigma.$
Then
the $C^*$-algebra $\OL$ 
is realized as a  universal concrete    
$C^*$-algebra generated by partial isometries
$S_\alpha$ indexed by symbols $\alpha \in \Sigma$
and projections $E_i^l, i=1,2,\dots,m(l)$ indexed by vertices $v_i^l \in V_l, \, l\in \Zp$  
satisfying  the following  operator relations called $({\frak L})$:
\begin{gather}
 \sum_{\alpha \in \Sigma}  S_{\alpha}S_{\alpha}^*  = \sum_{i=1}^{m(l)} E_i^l   =  1, \\
 S_\alpha S_\alpha^* E_i^l =E_i^l S_\alpha S_\alpha^*,  \\ 
  E_i^l   =  \sum_{j=1}^{m(l+1)}I_{l,l+1}(i,j)E_j^{l+1}, \\
 S_{\alpha}^*E_i^l S_{\alpha}  =  
\sum_{j=1}^{m(l+1)} A_{l,l+1}(i,\alpha,j)E_j^{l+1}
\end{gather}
for
$
i=1,2,\dots,m(l),\l\in \Zp, 
 \alpha \in \Sigma.
 $
 If in particular 
 ${\frak L}$ satisfies condition (I) in the sense of \cite{MaDocMath2002}, 
the operator relations determine the $C^*$-algebra in a unique way.  
\end{theorem}
Let $\mathcal{G} = (\mathcal{V}, \mathcal{E},\lambda)$
be a finite labeled directed graph with labeling map 
$\lambda: \mathcal{E}\longrightarrow \Sigma.$
We assume that the labeling is left-resolving in the above mentioned sense.
Then we have a $\lambda$-graph system 
${\frak L}_{\mathcal{G}}$ from the finite labeled graph $\mathcal{G}$ 
by setting
$V_l = \mathcal{V}, \, E_{l,l+1} = \mathcal{E}$ for every $l \in \Zp$
and
$\iota(v) = v, \, v \in \mathcal{V}.$ 
Then the $C^*$-algebra ${\mathcal{O}}_{\mathcal{G}}$ 
is so called a graph algebra with labeled edges
 (\cite[Proposition 7.1]{MaDocMath2002}, cf. \cite{Carl}, \cite{KPRR}).
Any topological Markov shift is realized as an edge shift with labeled edges,
then the operator relations among its canonical generators above 
reduce to the usual operator relations of Cuntz--Krieger algebras.
Hence the class of $C^*$-algebras $\OL$ generalizes 
the class of Cuntz--Krieger algebras.
It actually generalizes the class of $C^*$-algebras associated with subshifts
\cite{MaIJM1997} (cf. \cite{CaMa}).

The K-groups for symbolic matrix system described above is nothing but 
the K-groups $K_i(\OL)$ of the $C^*$-algebra of the
 $\lambda$-graph system ${\frak L}$ for the symbolic matrix system (\cite{MaDocMath2002}). 

A $\lambda$-graph system seems to fit in describing one-sided structure of subshifts.
Actually even if a subshift is a topological Markov shift, 
the associated Cuntz--Krieger algebra it-self does not cover two-sided structure of the underlying topological Markov shift.   
In this paper, we will generalize $\lambda$-graph system 
and introduce  two-sided extension of it, 
and construct associated $C^*$-algebras.
 

\section{$\lambda$-graph bisystems}
Let $\Sigma^-$ and $\Sigma^+$ be two finite alphabets.
They are generally not related to each other.
We will generalize the definition of $\lambda$-graph system to two-sided versions
in the following way.
 \begin{definition}\label{def:lambdabisystem}
{\it A $\lambda$-graph bisystem}\/ $({\frak L}^-, {\frak L}^+)$ is a pair of labeled Bratteli diagrams 
${\frak L}^- =(V^-, E^-, \lambda^-)$ over $\Sigma^-$ and 
$ {\frak L}^+=(V^+, E^+, \lambda^+)$ over  $\Sigma^+$
satisfying the following five conditions:
\begin{enumerate}
\renewcommand{\theenumi}{\roman{enumi}}
\renewcommand{\labelenumi}{\textup{(\theenumi)}}
\item
$
V^-= V^+ = \bigcup_{l=0}^\infty V_l 
$ disjoint union of finite sets $V_l, l\in \Zp$
with $m(l) := | V_l | < \infty $
for $l \in \Zp.$
\item
$
E^- = \bigcup_{l=0}^\infty E_{l+1, l}^-
$
and
$
E^+ = \bigcup_{l=0}^\infty E_{l, l+1}^+
$
disjoint unions of finite sets 
$E^-_{l+1,l}, E_{l,l+1}^+, l \in \Zp$,
respectively. 
\item
(1) Every edge $e^- \in E_{l+1,l}^-$ satisfies $s(e^-) \in V_{l+1}, \,\, t(e^-) \in V_l,$ 
and
every edge $e^+ \in E_{l,l+1}^+$ satisfies $s(e^+) \in V_{l}, \,\, t(e^+) \in V_{l+1}.$ 

(2) For every vertex $v \in V_l$ with $l\ne 0,$ 
there exists $e^- \in E_{l+1,l}^-, f^- \in E_{l,l-1}^-$ such that 
$v = s(f^-) = t(e^-)$,
 and for every vertex $v \in V_0,$ 
there exists $e^- \in E_{1,0}^-$ such that 
$v = t(e^-)$.
 
For every vertex $v \in V_l$ with $l\ne 0,$ 
there exists $e^+ \in E_{l,l+1}^+, f^+ \in E_{l-1, l}^+$ such that 
$v = t(f^+) = s(e^+)$,
 and for every vertex $v \in V_0,$ 
there exists $e^+ \in E_{0,1}^+$ such that 
$v = s(e^+).$

\item
The labeling map $\lambda^-: E^-\longrightarrow \Sigma^-$
is right-resolving, that is, the condition
$s(e^-) = s(f^-), \, \lambda^-(e^-) = \lambda^-(f^-)$
implies $e^- = f^-.$

The labeling map $\lambda^+: E^+\longrightarrow \Sigma^+$
is left-resolving, that is, the condition
$t(e^+) = t(f^+), \, \lambda^+(e^+) = \lambda^+(f^+)$
implies $e^+ = f^+.$
     
\item
For every pair 
$u \in V_{l}, \, v \in V_{l+2}$ with $l \in \Zp$,
 we put
\begin{align*}
E_+^-(u,v)
=  & \{ (e^-, e^+) \in E_{l+1, l}^-\times E_{l+1,l+2}^+ \mid 
t(e^-) = u, \, s(e^-) = s(e^+), t(e^+) = v \},\\
E_-^+(u,v)
=  & \{ (f^+, f^-) \in E_{l, l+1}^+\times E_{l+2,l+1}^- \mid 
s(f^+) = u, \, t(f^+) = t(f^-), s(f^-) = v \}. 
\end{align*}
Then there exists a bijective correspondence 
$$
\varphi:E_+^-(u, v) \longrightarrow E_-^+(u,v)
$$
satisfying 
$\lambda^-(e^-) = \lambda^-(f^-), \, 
\lambda^+(e^+) = \lambda^+(f^+)$
whenever $\varphi(e^-, e^+) = (f^+, f^-).$ 
\end{enumerate}
The property (v) is called {\it the local property of $\lambda$-graph bisystem}.
The pair $\LGBS$ is called {\it a $\lambda$-graph bisystem over}\/ $\Sigma^\pm.$  
\end{definition}
We write 
$V:= V^- = V^+$ and 
$\{v_1^l, \dots, v_{m(l)}^l\}$ for the vertex set $V_l.$

A $\lambda$-graph bisystem $\LGBS$ is said to be {\it standard}\/
if  its top vertex set $V_0$ is a singleton.
A $\lambda$-graph bisystem $\LGBS$ is said to {\it have a common alphabet}\/
if  $\Sigma^- = \Sigma^+.$
In this case, we write the alphabet $\Sigma^- = \Sigma^+$
as $\Sigma,$
and we say that $\LGBS$ is a $\lambda$-graph bisystem over common alphabet
$\Sigma.$
We write an edge $e^- \in E^-$ (resp. $e^+ \in E^+$) as $e$ without 
$-$ sign (resp. $+$ sign) unless we specify.

\medskip
Let $({\frak L}^-, {\frak L}^+)$ be a $\lambda$-graph bisystem over $\Sigma^\pm.$
For a vertex $u \in V_l$, 
we define 
its follower set $F(u)$ in ${\frak L}^-$ and
its predecessor set $P(u)$ in ${\frak L}^+$
as in the following way:
\begin{align*}
F(u) := & \{ 
(\lambda^-(f_l),\lambda^-(f_{l-1}),\dots,\lambda^-(f_1)) \in {(\Sigma^{-})}^l \mid
f_l \in E_{l,l-1}^-, f_{l-1} \in E_{l-1, l-2}^-, \dots, f_1 \in E_{1,0}^-,\,\, \\  
 & \qquad \qquad\qquad \qquad \qquad \qquad 
 s(f_l) = u, t(f_l) = s(f_{l-1}), \dots, t(f_2) = s(f_1) \}.
\end{align*}
Each element of $F(u)$ is figured such as  
$$  \qquad u \overset{\lambda^-(f_l)}{\longrightarrow} 
\bigcirc\overset{\lambda^-(f_{l-1})}{\longrightarrow}
 \cdots \overset{\lambda^-(f_{2})}{\longrightarrow}
\bigcirc\overset{\lambda^-(f_{1})}{\longrightarrow}\bigcirc.
$$
Similarly,
\begin{align*}
P(u) := & \{ 
(\lambda^+(e_1),\lambda^+(e_2),\dots,\lambda^+(e_l)) \in {(\Sigma^{+})}^l \mid
e_1 \in E_{0,1}^+, e_2 \in E_{1, 2}^+, \dots, e_l \in E_{l-1,l}^+,\,\,  \\
& \qquad \qquad\qquad \qquad 
t(e_1) =s(e_2), t(e_2) = s(e_3), \dots, t(e_{l-1}) =s(e_l), t(e_l) = u \}.
\end{align*}
Each element of $P(u)$ is figured such as  
$$
 \qquad 
\bigcirc\overset{\lambda^+(e_1)}{\longrightarrow} 
\bigcirc\overset{\lambda^+(e_2)}{\longrightarrow}
 \cdots \overset{\lambda^+(e_{l-1})}{\longrightarrow}
\bigcirc\overset{\lambda^+(e_{l})}{\longrightarrow}u.
$$
A standard $\lambda$-graph bisystem $\LGBS$ 
having a common alphabet is said to satisfy
 {\it{Follower-Predecessor Compatibility Condition},}
FPCC for short, if $\LGBS$ satisfies the condition
$F(u) = P(u)$  
for every vertex $u \in V_l, l\in \N.$

We will present several examples of $\lambda$-graph bisystems.

\begin{example} \label{ex:3.2}
\end{example}

{\bf (i)  $\lambda$-graph systems.}

Let ${\frak L} =(V,E,\lambda,\iota)$ be a $\lambda$-graph system over $\Sigma.$
We may construct a $\lambda$-graph bisystem $\LGBS$ from ${\frak L}$ in the following way.
Let us recognize the map $\iota: E\longrightarrow \Sigma$
with a new symbol written $\iota$, and
define a new alphabet
$\Sigma^- :=  \{\iota\}.$
The original alphabet $\Sigma$ is written $\Sigma^+.$ 
 We define 
$
E^+_{l,l+1}:= E_{l,l+1} 
$ for $l \in \Zp$ 
and 
$ 
\lambda^+:= \lambda:E^+ \longrightarrow \Sigma^+.
$  
We then have a labeled Bratteli diagram
${\frak L}^+ := (V, E^+, \lambda^+)$ over alphabet $\Sigma^+.$
The other 
labeled Bratteli diagram
${\frak L}^- := (V, E^-, \lambda^-)$ over alphabet $\Sigma^-$
is defined in the following way.
Define an edge $e^-\in E^-_{l+1,l}$
if
$\iota(v_j^{l+1}) = v_i^l$ 
so that 
$s(e^-) = v_j^{l+1}, t(e^-) = v_i^l$
and 
$\lambda^-(e^-) :=\iota \in \Sigma^-.$ 
Then we have a labeling map
$\lambda^-: E^-_{l+1,l} \longrightarrow \{ \iota\} = \Sigma^-,$
and hence a labeled Bratteli diagram
${\frak L}^- := (V, E^-, \lambda^-)$ over alphabet $\Sigma^-.$
Then the local property of the $\lambda$-graph system  ${\frak L}$ makes the pair 
$\LGBS$ a $\lambda$-graph bisystem.
This $\lambda$-graph bisystem does not satisfy FPCC.

Figure \ref{fig:(i)} in the end of this section is the first six levels of the $\lambda$-graph bisystem
defined by the canonical $\lambda$-graph system for the $\beta$-shift for $\beta =\frac{3}{2},$
 that was used  in \cite{KMW}. 
The $\beta$-shift is not sofic. 
In the $\lambda$-graph system, the alphabet $\Sigma^+ = \Sigma = \{0,1\}.$

\medskip

{\bf (ii)  A $\lambda$-graph bisystem for full $N$-shift.}

Let $N$ be a positive integer with $N >1.$
Take finite alphabets
$\Sigma^- = \{ \alpha^-_1,\dots,\alpha^-_N\}$
and
$\Sigma^+ = \{ \alpha^+_1,\dots,\alpha^+_N\}.$
We will construct a $\lambda$-graph bisystem
$(\frak{L}_N^-,\frak{L}_N^+)$
in the following way.
Let
$V_l=\{v_l\}$ one point set for each $l \in \Zp,$
and
$
 E_{l+1,l}^- =\{ e_1^-, \dots, e_N^-\}$,
$ E_{l,l+1}^+ =\{ e_1^+, \dots, e_N^+\}
$
such that 
\begin{gather*}
s(e_i^-) = v_{l+1}, \qquad t(e_i^-) = v_l, \qquad
\lambda^-(e_i^-) = \alpha^-_i \quad \text{ for } i=1,\dots, N, \, l \in \Zp,\\
s(e_i^+) = v_l, \qquad t(e_i^+) = v_{l+1}, \qquad
\lambda^+(e_i^+) = \alpha^+_i \quad \text{ for } i=1,\dots, N, \, l \in \Zp.   
\end{gather*}
We set
$
\frak{L}_N^- = (V, E^-, \lambda^-)$
and
$\frak{L}_N^+ = (V, E^+, \lambda^+)$.
Then   
$(\frak{L}_N^-,\frak{L}_N^+)$ is a $\lambda$-graph bisystem
satisfying FPCC.

\medskip

{\bf (iii) A $\lambda$-graph bisystem for golden mean shift.}

The topological Markov shift defined by the matrix  
$F = 
\begin{bmatrix}
1  & 1 \\
1 & 0
\end{bmatrix}
$
is called the golden mean shift (cf. \cite{LM}).
Let 
$\Sigma^- = \{\alpha^-, \beta^-\}$
and
$\Sigma^+ = \{\alpha^+, \beta^+\}$.
We set
$V_0= \{v_1^0\}, \,V_1 =\{ v_1^1, v_2^1\},\,
V_l =\{ v_1^l, v_2^l, v_3^l, v_4^l\}$ for $l \ge 2.$
The labeled Bratteli diagram ${\frak L}_F^-$
is defined as follows.
Define directed edges labeled symbols in $\Sigma^-$ such as
\begin{gather*}
v_1^0\overset{\alpha^-}{\longleftarrow}v_1^1, \quad
v_1^0\overset{\alpha^-}{\longleftarrow}v_2^1, \quad
v_1^0\overset{\beta^-}{\longleftarrow}v_1^1, \\
v_1^1\overset{\alpha^-}{\longleftarrow}v_1^2, \quad
v_1^1\overset{\alpha^-}{\longleftarrow}v_3^2, \quad
v_2^1\overset{\alpha^-}{\longleftarrow}v_2^2, \quad
v_2^1\overset{\alpha^-}{\longleftarrow}v_4^2, \quad
v_2^1\overset{\beta^-}{\longleftarrow}v_1^2, \quad
v_2^1\overset{\beta^-}{\longleftarrow}v_2^2, \\
v_1^l\overset{\alpha^-}{\longleftarrow}v_1^{l+1}, \quad
v_1^l\overset{\alpha^-}{\longleftarrow}v_3^{l+1}, \quad
v_2^l\overset{\alpha^-}{\longleftarrow}v_2^{l+1}, \quad
v_2^l\overset{\alpha^-}{\longleftarrow}v_4^{l+1}, \quad
v_3^l\overset{\beta^-}{\longleftarrow}v_1^{l+1}, \quad
v_4^l\overset{\beta^-}{\longleftarrow}v_2^{l+1}, \\
\end{gather*}
for $l \ge2.$
The other labeled Bratteli diagram ${\frak L}_F^+$
is defined as follows.
Define directed edges labeled symbols in $\Sigma^+$ such as
\begin{gather*}
v_1^0\overset{\alpha^+}{\longrightarrow}v_1^1, \quad
v_1^0\overset{\alpha^+}{\longrightarrow}v_2^1, \quad
v_1^0\overset{\beta^+}{\longrightarrow}v_1^1, \\
v_1^1\overset{\alpha^+}{\longrightarrow}v_1^2, \quad
v_1^1\overset{\alpha^+}{\longrightarrow}v_2^2, \quad
v_2^1\overset{\alpha^+}{\longrightarrow}v_3^2, \quad
v_2^1\overset{\alpha^+}{\longrightarrow}v_4^2, \quad
v_2^1\overset{\beta^+}{\longrightarrow}v_1^2, \quad
v_2^1\overset{\beta^+}{\longrightarrow}v_3^2, \\
v_1^l\overset{\alpha^+}{\longrightarrow}v_1^{l+1}, \quad
v_1^l\overset{\alpha^+}{\longrightarrow}v_2^{l+1}, \quad
v_3^l\overset{\alpha^+}{\longrightarrow}v_3^{l+1}, \quad
v_3^l\overset{\alpha^+}{\longrightarrow}v_4^{l+1}, \quad
v_2^l\overset{\beta^+}{\longrightarrow}v_1^{l+1}, \quad
v_4^l\overset{\beta^+}{\longrightarrow}v_3^{l+1}, \\
\end{gather*}
for $l \ge2.$
The pair $({\frak L}^-,{\frak L}^+)$
becomes a $\lambda$-graph bisystem satisfying FPCC.
It is figured in Figure \ref{fig:fibo} in the end of this section.

\medskip

{\bf (iv)  A $\lambda$-graph bisystem for even  shift.}

The sofic shift defined by the symbolic matrix   
$\E = 
\begin{bmatrix}
\alpha  & \beta \\
\beta & 0
\end{bmatrix}
$
is called the even shift (cf. \cite{LM}).
Let 
$\Sigma^+ = \{\alpha^+, \beta^+\}$ and
$\Sigma^- = \{\alpha^-, \beta^-\}.$
We set
$V_0 =\{v_1^0\}, \,V_1 =\{ v_1^1, v_2^1\},\,V_2 =\{ v_1^2, v_2^2, v_3^2\},\,
V_l =\{ v_1^l, v_2^l, v_3^l, v_4^l\}$ for $l \ge 3.$
The labeled Bratteli diagram ${\frak L}_\E^-$ is defined as follows.
Define directed edges labeled symbols in $\Sigma^-$ such as
\begin{gather*}
v_1^0\overset{\alpha^-}{\longleftarrow}v_1^1, \quad
v_1^0\overset{\beta^-}{\longleftarrow}v_2^1, \\
v_1^1\overset{\alpha^-}{\longleftarrow}v_1^2, \quad
v_1^1\overset{\beta^-}{\longleftarrow}v_3^2, \quad
v_2^1\overset{\beta^-}{\longleftarrow}v_2^2, \\
v_1^2\overset{\alpha^-}{\longleftarrow}v_1^{3}, \quad
v_2^2\overset{\alpha^-}{\longleftarrow}v_2^{3}, \quad
v_1^2\overset{\beta^-}{\longleftarrow}v_3^{3}, \quad
v_2^2\overset{\beta^-}{\longleftarrow}v_4^{3}, \quad
v_3^2\overset{\beta^-}{\longleftarrow}v_1^{3}, \\
v_1^l\overset{\alpha^-}{\longleftarrow}v_1^{l+1}, \quad
v_2^l\overset{\alpha^-}{\longleftarrow}v_2^{l+1}, \quad
v_1^l\overset{\beta^-}{\longleftarrow}v_3^{l+1}, \quad
v_2^l\overset{\beta^-}{\longleftarrow}v_4^{l+1}, \quad
v_3^l\overset{\beta^-}{\longleftarrow}v_1^{l+1}, \quad
v_4^l\overset{\beta^-}{\longleftarrow}v_2^{l+1}, \\
\end{gather*}
for $l \ge 3.$
The other labeled Bratteli diagram ${\frak L}_\E^+$
is defined as follows. 
Define directed edges labeled symbols in $\Sigma^+$ such as
\begin{gather*}
v_1^0\overset{\alpha^+}{\longrightarrow}v_1^1, \quad
v_1^0\overset{\beta^+}{\longrightarrow}v_2^1, \\
v_1^1\overset{\alpha^+}{\longrightarrow}v_1^2, \quad
v_1^1\overset{\beta^+}{\longrightarrow}v_2^2, \quad
v_2^1\overset{\beta^+}{\longrightarrow}v_3^2, \\
v_1^2\overset{\alpha^+}{\longrightarrow}v_1^{3}, \quad
v_3^2\overset{\alpha^+}{\longrightarrow}v_3^{3}, \quad
v_1^2\overset{\beta^+}{\longrightarrow}v_2^{3}, \quad
v_2^2\overset{\beta^+}{\longrightarrow}v_1^{3}, \quad
v_3^2\overset{\beta^+}{\longrightarrow}v_4^{3}, \\
v_1^l\overset{\alpha^+}{\longrightarrow}v_1^{l+1}, \quad
v_3^l\overset{\alpha^+}{\longrightarrow}v_3^{l+1}, \quad
v_1^l\overset{\beta^+}{\longrightarrow}v_2^{l+1}, \quad
v_2^l\overset{\beta^+}{\longrightarrow}v_1^{l+1}, \quad
v_3^l\overset{\beta^+}{\longrightarrow}v_4^{l+1}, \quad
v_4^l\overset{\beta^+}{\longrightarrow}v_3^{l+1}, \\
\end{gather*}
for $l \ge 3.$
The pair $({\frak L}_{\E}^-,{\frak L}_{\E}^+)$
becomes a $\lambda$-graph bisystem that satisfies FPCC.
It is figured in Figure \ref{fig:even}
in the end of this section.

\medskip

Let $\LGBS$ be a $\lambda$-graph bisystem.
Let us denote by ${\frak{L}}^{-t}$
the labeled Bratteli diagram for which every  edge is  reversed  with the original edge.
This means that for an edge $e^- \in E_{l+1,l}^-$ such that 
$s(e^-) \in V_{l+1}, t(e^-) \in V_l,$ the reversed edge $e^{-t}$ is defined 
by $t(e^{-t}) :=s(e^-) \in V_{l+1}, s(e^{-t}): =t(e^-) \in V_l$
and $\lambda^-(e^{-t}):= \lambda^-(e^{-}) \in \Sigma^-.$
The resulting labeled Bratteli diagram is written ${\frak{L}}^{-t}.$
We similarly define a labeled Bratteli diagram ${\frak{L}}^{+t}$ 
from $ {\frak{L}}^{+}.$
It is easy to see that the pair $({\frak{L}}^{+t},{\frak{L}}^{-t})$
becomes a $\lambda$-graph bisystem. 
It is called the {\it transpose} of $\LGBS$ and written  $\LGBS^t.$

\newpage

\begin{figure}[htbp]
\begin{center}
\input{pictureiMP2.tex}
\end{center}
\caption{ $\lambda$-graph bisystem $({\frak L}^-,{\frak L}^+)$ of Example \ref{ex:3.2} (i) }
\label{fig:(i)}
\end{figure}
\noindent
where upward arrows $\longleftarrow$ 
in ${\frak L}^-$ are labeled $\iota$,
and downward arrows $\longleftarrow$ and {\mathversion{bold} $\longleftarrow$} (bold)
in ${\frak L}^+$ are labeled $0$ and $1$, respectively.  

\newpage

\begin{figure}[htbp]
\begin{center}
\input{pictureFiboMP2.tex}
\end{center}
\caption{ $\lambda$-graph bisystem $({\frak L}_F^-,{\frak L}_F^+)$ of Example \ref{ex:3.2} (iii) }
\label{fig:fibo}
\end{figure}
\noindent
where upward arrows $\longleftarrow$ and {\mathversion{bold} $\longleftarrow$} (bold)
in ${\frak L}_F^-$ are labeled $\alpha^-$ and $\beta^-$, respectively,
and downward arrows $\longleftarrow$ and {\mathversion{bold} $\longleftarrow$} (bold)
in ${\frak L}_F^+$ are labeled $\alpha^+$ and $\beta^+$, respectively.  

\newpage

\begin{figure}[htbp]
\begin{center}
\input{pictureevenMP2.tex}
\end{center}
\caption{ $\lambda$-graph bisystem $({\frak L}_{\mathcal{E}}^-,{\frak L}_{\mathcal{E}}^+)$ of Example \ref{ex:3.2} (iv) }
\label{fig:even}
\end{figure}
\noindent
where upward arrows $\longleftarrow$ and {\mathversion{bold} $\longleftarrow$} (bold)
in ${\frak L}_{\mathcal{E}}^-$ are labeled $\alpha^-$ and $\beta^-$, respectively,
and downward arrows $\longleftarrow$ and {\mathversion{bold} $\longleftarrow$} (bold)
in ${\frak L}_{\mathcal{E}}^+$ are labeled $\alpha^+$ and $\beta^+$, respectively.  

\newpage
\section{Symbolic matrix bisystems }
Let $\Sigma$ be a finite alphabet.
 We denote by ${\frak S}_\Sigma$ the set of finite formal sums of elements of $\Sigma.$
 By a symbolic matrix $\A$ over $\Sigma$
 we mean a rectangular finite matrix $\A = [\A(i,j)]_{i,j}$
 whose entries in ${\frak S}_\Sigma.$ 
We write the empty word $\emptyset$ as $0$ in ${\frak S}_\Sigma.$
For the symbolic matrix $\A,$
we write an edge $e_k$ labeled $\alpha_k$
for $k=1,\dots,n$ for a vertex $v_i$ to a vertex $v_j$
 if $\A(i,j) = \alpha_1+\cdots +\alpha_n.$  
For two alphabets $\Sigma, \Sigma',$
the notation $\Sigma\cdot\Sigma'$ denotes the set
$\{ a\cdot b \mid a \in \Sigma, \, b \in \Sigma' \}.$
The following notion of specified equivalence 
between symbolic matrices due to M. Nasu in
\cite{Nasu}, \cite{NasuMemoir}.
For two symbolic matrices
$\A$ over alphabet $\Sigma$
 and 
 $\A'$ over alphabet $\Sigma'$
 and bijection
 $\phi$ from a subset of $\Sigma$ onto a subset of $\Sigma'$,
 we call $\A$ and $\A'$ are  specified equivalent under specification
 $\phi$ if $\A'$ can be obtained from $\A$
  by replacing every symbol 
 $\alpha$ appearing in $\A$ by $\phi(\alpha)$.
 We write it as
 $\A \overset{\phi}{\simeq} \A'$. 
 We call $\phi$ a specification from $\Sigma$ to $\Sigma'$.
For two alphabet $\Sigma_1, \Sigma_2,$ 
the bijection 
$\alpha\cdot\beta\in \Sigma_1\cdot\Sigma_2 \longrightarrow 
\beta\cdot\alpha\in \Sigma_2\cdot\Sigma_1$
naturally yields a bijection from 
${\frak S}_{\Sigma_1\cdot\Sigma_2}$
to 
${\frak S}_{\Sigma_2\cdot\Sigma_1}$
that we denote by $\kappa$ 
and call the exchanging specification  
between $\Sigma_1$ and $\Sigma_2.$
\begin{definition}
{\it A symbolic matrix bisystem}\/
$(\M_{l,l+1}^-,\M_{l,l+1}^+), l\in \Zp$ is a pair of
sequences of rectangular symbolic matrices
$\M_{l,l+1}^-$ over $\Sigma^-$ and
$\M_{l,l+1}^+$ over $\Sigma^+$  
 satisfying  
the following five conditions:
\begin{enumerate}
\renewcommand{\theenumi}{\roman{enumi}}
\renewcommand{\labelenumi}{\textup{(\theenumi)}}
\item
Both $\M_{l,l+1}^-$ and $\M_{l,l+1}^+$ are $m(l)\times m(l+1)$ 
rectangular  symbolic matrices with  $m(l) \in \N$  for $l \in \Zp.$
\item
(1) For $i,$ there exists $j$ such that $\M_{l,l+1}^-(i,j)\ne 0$, and

\hspace{6mm} for $i,$ there exists $j$ such that $\M_{l,l+1}^+(i,j)\ne 0$.

(2) For $j,$ there exists $i$ such that $\M_{l,l+1}^-(i,j)\ne 0$, and

\hspace{6mm} for $j,$ there exists $i$ such that $\M_{l,l+1}^+(i,j)\ne 0$.

\item Each component of both $\M_{l,l+1}^-$ and $\M_{l,l+1}^+$ 
does not have multiple symbols. 
This means that 
if $\M_{l,l+1}^-(i,j) = \alpha_1+\cdots+\alpha_n,$
then the symbols $\alpha_1,\dots,\alpha_n$ 
are all distinct each other.
The same  condition is required for  $\M_{l,l+1}^+.$
\item 
For each $j=1,2,\dots,m(l+1),$
both the $j$th columns 
$[\M_{l,l+1}^-(i,j)]_{i=1}^{m(l)}$
and
$[\M_{l,l+1}^+(i,j)]_{i=1}^{m(l)}$
do not have multiple symbols.
Namely, if a symbol $\alpha$ appears in $\M_{l,l+1}^-(i,j)$ for some 
$i\in \{1,2,\dots, m(l)\},$
then it does not appear in  any other row $\M_{l,l+1}^-(i',j)$ for 
$i' \ne i,$
and $\M_{l,l+1}^+$ has the same property.
\item
 $\M_{l,l+1}^- \M_{l+1,l+2}^+ \overset{\kappa}{\simeq}
  \M_{l,l+1}^+ \M_{l+1,l+2}^-$ 
 for $l \in \Zp,$
that means for $i=1,2,\dots,m(l),\, j=1,2,\dots,m(l+2),$
\begin{equation}
\sum_{k=1}^{m(l+1)}\M_{l,l+1}^-(i,k) \M_{l+1,l+2}^+(k,j)
 =\sum_{k=1}^{m(l+1)} \kappa\left( \M_{l,l+1}^+(i,k)\M_{l+1,l+2}^-(k,j) \right),
\label{eq:MLocal}
\end{equation}
where $\kappa$ is the exchanging specification between 
$\Sigma$ and $\Sigma'.$
\end{enumerate}
The matrix $\M_{l,l+1}^-$ (resp.  $\M_{l,l+1}^+$) satisfying
the condition (iv) is said to be right-resolving (resp. left-resolving). 
The condition (v) exactly corresponds to the local property of $\lambda$-graph bisystems (v)
in Definition \ref{def:lambdabisystem}.
The pair $(\M^-,\M^+)$ is called {\it a symbolic matrix bisystem over}\/ $\Sigma^\pm.$
It is easy to see that a symbolic matrix bisystem is exactly 
a matrix presentation of $\lambda$-graph bisystem.
\end{definition}
A symbolic matrix bisystem $(\M^-, \M^+)$ is said to be {\it standard}\/
if  $m(0) =1$, that is its row sizes of the matrices $\M^-_{0,1}$ and $\M^+_{0,1}$ are one.
A symbolic matrix bisystem  $(\M^-, \M^+)$ is said to {\it have a common alphabet}\/
if  $\Sigma^- = \Sigma^+.$
In this case, write the alphabet $\Sigma^- = \Sigma^+$
as $\Sigma,$
and say that $(\M^-, \M^+)$ is a symbolic matrix bisystem over common alphabet $\Sigma.$
It is said to satisfy
{\it Follower-Predecessor Compatibility Condition},\/
FPCC for short,
if for every $l \in \N$ and $j = 1,2,\dots,m(l),$
the set of words  appearing in $[\M_{0,1}^-\M_{1,2}^-\cdots\M^-_{l-1,l}](1,j)$
coincides with the set of transposed words appearing in 
$[\M_{0,1}^+\M_{1,2}^+\cdots\M^+_{l-1,l}](1,j).$

Two symbolic matrix bisystems 
$(\M^-,\M^+)$ over  $\Sigma_\M^\pm$
and
$(\SN^-, \SN^+)$ over  $\Sigma_\SN^\pm$
are said to be isomorphic 
if their sizes $m(l)\times m(l+1)$  and $n(l) \times n(l+1)$ 
of the matrices $\M^{\pm}_{l,l+1}$ and $\SN^\pm_{l,l+1}$ 
coincide, that is $m(l) = n(l)$,   for each $l \in \Zp$ and
 there exists a specification $\phi$ 
from $\Sigma_\M$ to $\Sigma_\SN$ 
and an $m(l) \times m(l)$-permutation matrix
$P_l$ for each $l \in \Zp$ such that
$$
P_l \M_{l,l+1}^- \overset{\phi}{\simeq} \SN_{l,l+1}^-P_{l+1},
\qquad
P_l \M_{l,l+1}^+ \overset{\phi}{\simeq} \SN_{l,l+1}^+P_{l+1} \qquad
\text{ for }
\quad l \in \Zp.
$$

Let 
$\A = [\alpha_{ij}]_{i,j=1}^N$  
be an $N\times N$ symbolic matrix
over $\Sigma =\{ \alpha_{ij}\mid i,j=1,\dots, N \}.$
We set alphabets 
$\Sigma^- =\{ \alpha_{ij}^- \mid i,j=1,\dots, N \}$
and
$\Sigma^+ =\{ \alpha_{ij}^+\mid i,j=1,\dots, N \}.$
We will define 
$N^2 \times N^2$ symbolic matrices 
$\M_\A^-$ over $\Sigma^-$  and 
$\M_\A^+$ over $\Sigma^+$  
 in the following way. 
Define first the $N \times N$ matrix
$
\A^+ := 
\begin{bmatrix}
\alpha_{11}^+ & \cdots & \alpha_{1N}^+ \\
\vdots   &        & \vdots  \\
\alpha_{N1}^+ & \cdots & \alpha_{NN}^+
\end{bmatrix}
$
and the diagonal matrix
$\alpha^-_{ij}I_N$ 
whose diagonal entries are
$(\alpha^-_{ij}, \dots, \alpha^-_{ij}).$
We define
$N^2 \times N^2$ symbolic matrices 
$\M_\A^-, \M_\A^+$ 
by setting
\begin{equation*}
\M_\A^- := 
\begin{bmatrix}
\alpha_{11}^- I_N & \alpha_{21}^-I_N  & \cdots & \alpha_{N1}^-I_N  \\
\alpha_{12}^-I_N  & \alpha_{22}^-I_N  & \cdots  & \alpha_{N2}^-I_N  \\
\vdots              & \vdots              & \ddots & \vdots    \\
\alpha_{1N}^-I_N  & \alpha_{2N}^- I_N &  \cdots& \alpha_{NN}^- I_N 
\end{bmatrix},
\qquad
\M_\A^+ := 
\begin{bmatrix}
\A^+ & 0    & \dots & 0 \\
0     & \A^+& \ddots & \vdots \\
\vdots     & \ddots  & \ddots & 0 \\
0     & \dots     &  0 & \A^+
\end{bmatrix}.
\end{equation*}
For $N=2$ and a $2\times 2$ symbolic matrix 
$\A = 
\begin{bmatrix}
a & b \\
c & d
\end{bmatrix},
$
we have 
\begin{equation}
\M_\A^- = 
\begin{bmatrix}
a^- & 0   & c^- & 0 \\
0   & a^- & 0   & c^- \\
b^- & 0   & d^- & 0  \\
0   & b^- & 0   & d^- 
\end{bmatrix}, \qquad
\M_\A^+ =
\begin{bmatrix}
a^+ & b^+ & 0 & 0 \\
c^+ & d^+ & 0 & 0 \\
0 & 0 & a^+ & b^+ \\
0 & 0 & c^+ & d^+ \\ 
\end{bmatrix}.  \label{eq:symbma2}
\end{equation}
Let 
$\kappa: 
\Sigma^-\cdot \Sigma^+ \longrightarrow \Sigma^+\cdot \Sigma^-
$
be the exchanging specification defined by
$\kappa(\beta\cdot\alpha) = \alpha\cdot\beta$
for $\alpha \in \Sigma^+, \beta \in \Sigma^-.$
We have the following lemma by straightforward calculation. 
\begin{lemma}
For an $N\times N$ symbolic matrix
$\A = [\alpha_{ij}]_{i,j=1}^N,$
we have a specified equivalence
$\M_\A^-\cdot \M_\A^+ \overset{\kappa}{\simeq} \M_\A^+\cdot \M_\A^-.$
\end{lemma}
\begin{proof}
The matrices
$\M_\A^-\cdot \M_\A^+$
and 
$\M_\A^+\cdot \M_\A^-$
are  $N\times N$ matrices over $N\times N$ symbolic matrices.
Their  $(i,j)$th block matrices  are
\begin{equation*}
\begin{bmatrix}
\alpha_{ji}^-\alpha_{11}^+ & \alpha_{ji}^-\alpha_{12}^+ &\cdots &\alpha_{ji}^-\alpha_{1N}^+  \\
\alpha_{ji}^-\alpha_{21}^+ & \alpha_{ji}^-\alpha_{22}^+ &\cdots &\alpha_{ji}^-\alpha_{2N}^+  \\
\vdots              & \vdots              & \ddots & \vdots    \\
\alpha_{ji}^-\alpha_{N1}^+ & \alpha_{ji}^-\alpha_{N2}^+ &\cdots &\alpha_{ji}^-\alpha_{NN}^+  \\
\end{bmatrix}, 
\qquad
\begin{bmatrix}
\alpha_{11}^+\alpha_{ji}^- & \alpha_{12}^+\alpha_{ji}^- &\cdots &\alpha_{1N}^+\alpha_{ji}^-  \\
\alpha_{21}^+\alpha_{ji}^- & \alpha_{22}^+\alpha_{ji}^- &\cdots &\alpha_{2N}^+\alpha_{ji}^-  \\
\vdots              & \vdots              & \ddots & \vdots    \\
\alpha_{N1}^+\alpha_{ji}^- &\alpha_{N2}^+\alpha_{ji}^- &\cdots &\alpha_{NN}^+ \alpha_{ji}^- \\
\end{bmatrix}
\end{equation*}
respectively, so that we have
 a specified equivalence
$\M_\A^-\cdot \M_\A^+ \overset{\kappa}{\simeq} \M_\A^+\cdot \M_\A^-.$
\end{proof}
The following proposition is straightforward.
\begin{proposition}
For an $N\times N$ symbolic matrix
$\A = [\alpha_{ij}]_{i,j=1}^N,$
we put
\begin{gather*}
\M_{\A_{0,1}}^- =
[\alpha_{11}^-, \alpha_{12}^-,\cdots, \alpha_{1N}^-,
 \alpha_{21}^-, \alpha_{22}^-,\cdots, \alpha_{2N}^-, \cdots,
\alpha_{N1}^-, \alpha_{N2}^-,\cdots, \alpha_{NN}^-], \\ 
\M_{\A_{0,1}}^+ =
[\alpha_{11}^+, \alpha_{12}^+,\cdots, \alpha_{1N}^+,
 \alpha_{21}^+, \alpha_{22}^+,\cdots, \alpha_{2N}^+, \cdots,
\alpha_{N1}^+, \alpha_{N2}^+,\cdots, \alpha_{NN}^+], \\
\M_{\A_{l,l+1}}^- =\M_\A^- \quad \text{ for } l= 1,2,\dots, \\
\M_{\A_{l,l+1}}^+ =\M_\A^+ \quad \text{ for } l= 1,2,\dots. \\
\end{gather*}
Then $(\M_{\A_{l,l+1}}^-, \M_{\A_{l,l+1}}^+)_{l\in \Zp}$ 
becomes a standard symbolic matrix bisystem.
\end{proposition}
If in particular we identify the symbols $\alpha_{ij}^-$ 
with $\alpha_{ij}^+$ for $i,j=1,\dots,N$ 
and put $\Sigma = \Sigma^- = \Sigma^+$, then the symbolic matrix bisystem
$(\M_{\A_{l,l+1}}^-, \M_{\A_{l,l+1}}^+)_{l\in \Zp}$ satisfies FPCC.
\begin{remark}
Let $(I, \M)$ be a symbolic matrix system over $\Sigma$
as in Section 2.
It satisfies the equality \eqref{eq:sms}.
We set
$\Sigma^+ = \Sigma$ and 
$\Sigma^- = \{ 1 \}.
$
By putting 
\begin{equation*}
\M^-_{l,l+1} = I_{l,l+1}, \qquad \M^+_{l,l+1} = \M_{l,l+1}, \qquad \l\in \Zp,
\end{equation*}
we have a symbolic matrix bisystem $(\M^-, \M^+).$
\end{remark}

\section{Subshifts and $\lambda$-graph bisystems}
We will show that any $\lambda$-graph bisystem $\LGBS$ gives rise to two  subshifts
written $\Lambda_{{\frak L}^-}$ and $\Lambda_{{\frak L}^+}$.
They are called the presenting subshifts by $\LGBS.$
If in particular $\LGBS$ satisfies FPCC, 
then the two subshifts coincide, and determine one specific subshift.  
Conversely any subshift yields a $\lambda$-graph bisystem satisfying FPCC whose presenting subshift 
is the original subshift.
We fix an arbitrary $\lambda$-graph bisystem $\LGBS$ over $\Sigma^{\pm}$. 
Let us denote by $W_{{\frak L}^-}, W_{{\frak L}^+}$
the set of words appearing  in the labeled Bratteli diagram 
 ${\frak L}^-, {\frak L}^+$ respectively.
They are defined by  
\begin{align*}
W_{{\frak L}^-} 
=& \{ 
(\lambda^-(f_{l+n}), \lambda^-(f_{l+n-1}),\dots,\lambda^-(f_{1+n})) \in {(\Sigma^-)}^l 
\mid  f_m \in E_{m,m-1}^-, m=n+1, \dots, n+l, \\
& \quad 
t(f_m) =s(f_{m-1}), m= n+2,\dots, n+l \text{ for some } n \in \Zp \text{ and } l \in \N \}, \\ 
W_{{\frak L}^+} 
=& \{
 (\lambda^+(e_{n+1}), \lambda^+(e_{n+2}),\dots,\lambda^+(e_{n+l})) \in {(\Sigma^+)}^l 
\mid e_m \in E_{m-1,m}^+, m=n+1, \dots, n+l, \\  
& \quad 
t(e_m)= s(e_{m+1}), m= n+1,\dots, n+l -1 \text{ for some } n \in \Zp \text{ and } l \in \N \}.
\end{align*}
It is easy to see that if  $\LGBS$ satisfies FPCC, 
then $W_{{\frak L}^-} =W_{{\frak L}^+}.$
In this case we write $W_{\LGBS}:= W_{{\frak L}^+} (= W_{{\frak L}^-}).$
\begin{lemma}\label{lem:5.1}
For a $\lambda$-graph bisystem $\LGBS$(not necessarily satisfying FPCC),
there exist a unique pair of subshifts  
$\Lambda_{{\frak L}^-}$ and $\Lambda_{{\frak L}^+}$
such that their sets of admissible words 
are 
$W_{{\frak L}^-}$ and $W_{{\frak L}^+},$
respectively.
\end{lemma}
\begin{proof}
It suffices to show that the set 
 $W_{{\frak L}^+}$ is a language in the sense of \cite[Definition 1.3.1]{LM}
because of \cite[Proposition 1.3.4]{LM}.
It is clear that 
any subword of a word of $W_{{\frak L}^+}$ belongs to $W_{{\frak L}^+}.$
We will show any word of $W_{{\frak L}^+}$ may extend to its both sides.
For $w \in W_{{\frak L}^+},$
it is obvious that 
there exists $\alpha^+ \in \Sigma^+$ such that 
$w \alpha^+  \in W_{{\frak L}^+}.$
By the local property of $\lambda$-graph bisystem,
we may find
 $\beta^+ \in \Sigma^+$ such that 
$\beta^+ w\in W_{{\frak L}^+}.$ 
Hence a word of $W_{{\frak L}^+}$ can extend to its both sides,
proving the set $W_{{\frak L}^+}$ is a language.  
We similarly see that $W_{{\frak L}^-}$ is a language.
Therefore they give rise to subshifts
written 
$\Lambda_{{\frak L}^+}$ and $\Lambda_{{\frak L}^-}$, respectively,
such that 
their admissible words 
are $W_{{\frak L}^+}$ and   $W_{{\frak L}^+}$, respectively.
\end{proof}
The subshifts $\Lambda_{{\frak L}^-}$ and $\Lambda_{{\frak L}^+}$
are  called the subshifts presented by $\LGBS.$
If  in particular  $\LGBS$ satisfies FPCC, 
then $W_{{\frak L}^-} =W_{{\frak L}^+}$
so that 
their presenting subshifts $\Lambda_{{\frak L}^-}$ and $\Lambda_{{\frak L}^+}$
coincide. 
Hence a $\lambda$-graph bisystem $\LGBS$ satisfying FPCC
yields a unique subshift which is called the subshift presented by 
$\LGBS$ and written 
$\Lambda_{\LGBS}.$
  
\medskip

We will next construct a $\lambda$-graph bisystem satisfying FPCC
from an arbitrary subshift $\Lambda$
such that the presented subshift 
by the $\lambda$-graph bisystem coincides with the original subshift.
We fix a subshift $\Lambda$ over $\Sigma.$
For $k,l \in \Z$ with $k < l$ and $x=(x_n)_{n\in \Z} \in \Lambda,$
we set
\begin{align*}
W_{k,l}(x) := 
& \{ (\mu_{k+1},\mu_{k+2},\dots,\mu_{l-1}) \in B_{l-k-1}(\Lambda) \mid  \\
& \qquad  (\dots, x_{k-1}, x_k, \mu_{k+1},\mu_{k+2},\dots,\mu_{l-1}, 
x_l, x_{l+1},\dots ) \in \Lambda \},
\end{align*}
\begin{equation*}
\cdots \underset{x_{k-1}}{\longrightarrow} \underset{x_k}{\longrightarrow}
\square\square \cdots \square 
\underset{x_l}{\longrightarrow} \underset{x_{l+1}}{\longrightarrow}
\cdots
\end{equation*}
and
$
W_{k,k}(x) :=
\emptyset. 
$
In the above picture, the finite sequence of boxes
$\square\square \cdots \square$
denotes the words 
$(\mu_{k+1},\mu_{k+2},\dots,\mu_{l-1})$ of length $l-k-1$
that can put in between the left infinite sequence
$(\dots, x_{k-1}, x_k)$ of $x$
and the right infinite sequence
$(x_l, x_{l+1},\dots )$ of $x$. 
Two bi-infinite sequences $x, y \in \Lambda$
are said to be $(k,l)$-{\it{centrally equivalent}}\/ if $W_{k,l}(x) = W_{k,l}(y),$
and written $x\overset{c}{\underset{(k,l)}{\sim}}y.$
This means that the set of words
put in between $x_{(-\infty,k]}$ and $x_{[l,\infty)}$
coincides with 
the set of words put in between
$y_{(-\infty,k]}$ and $y_{[l,\infty)},$
where $x_{(-\infty,k]} = (x_n)_{n\le k}$
and
$x_{[l,\infty)}= (x_n)_{n\ge l}$.
Define the set of equivalence classes
$$
\Omega_{k,l}^c = \Lambda/\overset{c}{\underset{(k,l)}{\sim}},
$$
that is a finite set because the set of words whose lengths are less than or equal to 
$l-k-1$ is finite.
Let 
$\{ C_1^{k,l},C_2^{k,l}, \dots, C_{m(k,l)}^{k,l} \}
$ be the set 
$\Omega_{k,l}^c$
of $\overset{c}{\underset{(k,l)}{\sim}}$ equivalence classes.

For $x=(x_n)_{n\in \Z}  \in C_i^{k,l}$ and $\alpha \in \Sigma,$
suppose that 
$ (\mu_{k+1},\mu_{k+2},\dots,\mu_{l-2},\alpha) \in W_{k,l}(x)$ 
for some $\mu =( \mu_{k+1},\mu_{k+2},\dots,\mu_{l-2}) \in B_{l-k-2}(\Lambda)$
so that the bi-infinite sequence
\begin{equation*}
 x(\mu, \alpha) : = 
(\dots, x_{k-1}, x_k, \mu_{k+1},\mu_{k+2},\dots,\mu_{l-2},\alpha, x_l, x_{l+1},\dots ) \in \Lambda
\end{equation*}
belongs to $\Lambda.$
If $x(\mu, \alpha)$ belongs to $C_j^{k,l-1},$
then we write
$\alpha C_i^{k,l} \subset C_j^{k,l-1}.$

Similarly for
$x=(x_n)_{n\in \Z}  \in C_i^{k,l}$ and $\beta \in \Sigma,$
suppose that 
$ (\beta, \nu_{k+2},\nu_{k+3},\dots,\nu_{l-1}) \in W_{k,l}(x)$ 
for some $\nu =( \nu_{k+2},\nu_{k+3},\dots,\nu_{l-1}) \in B_{l-k-2}(\Lambda)$
so that the bi-infinite sequence
\begin{equation*}
 x(\beta,\nu) : = 
(\dots, x_{k-1}, x_k, \beta,\nu_{k+2},\nu_{k+3},\dots,\nu_{l-1}, x_l, x_{l+1},\dots ) \in \Lambda
\end{equation*}
belongs to $\Lambda.$
If $x(\beta, \nu)$ belongs to $C_h^{k+1,l},$
then we write
$C_i^{k,l}\beta \subset C_h^{k+1,l}.$
\begin{lemma}
Keep the above notation.
\hspace{4cm}
\begin{enumerate}
\renewcommand{\theenumi}{\roman{enumi}}
\renewcommand{\labelenumi}{\textup{(\theenumi)}}
\item
The notation $\alpha C_i^{k,l} \subset C_j^{k,l-1}$
is well-defined, that is, it does not depend on the choice of 
$x=(x_n)_{n\in \Z}  \in C_i^{k,l}$ and
$\mu =( \mu_{k+1},\mu_{k+2},\dots,\mu_{l-2}) \in B_{l-k-2}(\Lambda)$
as long as  
$ (\mu_{k+1},\mu_{k+2},\dots,\mu_{l-2},\alpha) \in W_{k,l}(x).$
\item
The notation $C_i^{k,l}\beta \subset C_h^{k+1,l}$
is well-defined, that is, it does not depend on the choice of 
$x=(x_n)_{n\in \Z}  \in C_i^{k,l}$ and
$\nu =( \nu_{k+2},\nu_{k+3},\dots,\nu_{l-1}) \in B_{l-k-2}(\Lambda)$
as long as  
$ (\beta, \nu_{k+2},\nu_{k+3},\dots,\nu_{l-1}) \in W_{k,l}(x).$ 
\end{enumerate}
\end{lemma}
\begin{proof}
(i)
Take 
$x=(x_n)_{n\in \Z}, z=(z_n)_{n\in \Z}  \in C_i^{k,l}$
such that 
$x\overset{c}{\underset{(k,l)}{\sim}}z.$
Suppose that $\alpha \in \Sigma$ satisfies
$ (\mu_{k+1},\mu_{k+2},\dots,\mu_{l-2},\alpha) \in W_{k,l}(x)$
and
$ (\nu_{k+1},\nu_{k+2},\dots,\nu_{l-2},\alpha) \in W_{k,l}(z)$
for some
 $\mu =( \mu_{k+1},\mu_{k+2},\dots,\mu_{l-2}) $
and
$\nu =( \nu_{k+1},\nu_{k+2},\dots,\nu_{l-2}) \in B_{l-k-2}(\Lambda).$
Consider
the bi-infinite sequences
$x(\mu,\alpha), z(\nu,\alpha) \in \Lambda.$
Since
$W_{k,l}(x)=W_{k,l}(z),$
we have
$W_{k,l-1}(x(\mu,\alpha))=W_{k,l-1}(z(\nu,\alpha)).$
This implies that
$x(\mu,\alpha)\overset{c}{\underset{(k,l-1)}{\sim}}z(\nu,\alpha),$
proving the class
$C_j^{k,l-1}$ containing $x(\mu,\alpha)$
does not depend on the choice of  
$x=(x_n)_{n\in \Z}  \in C_i^{k,l}$ and
$\mu =( \mu_{k+1},\mu_{k+2},\dots,\mu_{l-2}) \in B_{l-k-2}(\Lambda)$
as long as  
$ (\mu_{k+1},\mu_{k+2},\dots,\mu_{l-2},\alpha) \in W_{k,l}(x).$
Hence the class 
$C_j^{k,l-1}$ is well-defined. 
(ii) is similarly shown to (i).
\end{proof}
The following lemma is now clear.
\begin{lemma}
For $x, y \in \Lambda$, we have
$x\overset{c}{\underset{(k,l)}{\sim}}z$
if and only if
$\sigma(x)\overset{c}{\underset{(k-1,l-1)}{\sim}}\sigma(z).$
\end{lemma}
Hence we may identify 
$\Omega^c_{k,l}$ with $\Omega^c_{k-1,l-1}$
and 
$C_i^{k,l}$ with $C_i^{k-1,l-1}$ for $i=1,2,\dots,m(k,l)=m(k-1,l-1)$
through the shift
$\sigma: \Lambda\longrightarrow \Lambda$
so that we identify
$\Omega^c_{k,l}$ with $\Omega^c_{k+n,l+n}$
and 
$C_i^{k,l}$ with $C_i^{k+n,l+n}$ for all $n \in \Z$ and 
$i=1,2,\dots,m(k,l) =m(k+n,l+n).$

Let us in particular specify the following equivalence classes
 $\Omega^c_{-l,1}$ and $\Omega^c_{-1,l}$
and define the vertex sets  $V_l^-$ and $V_l^+$ for $l=0,1,2,\dots$
by setting
\begin{align*}
V_0^-&:= \{ \Lambda\}, \hspace{11mm} V_0^+:= \{ \Lambda\},\\ 
V_1^-&:= \Omega^c_{-1,1}, \qquad V_1^+:= \Omega^c_{-1,1},\\
V_2^-&:= \Omega^c_{-2,1}, \qquad V_2^+:= \Omega^c_{-1,2}, \\
V_3^-&:= \Omega^c_{-3,1}, \qquad V_3^+:= \Omega^c_{-1,3}, \\
       & \cdots  \hspace{25mm} \cdots \\
V_l^-&:= \Omega^c_{-l,1}, \qquad V_l^+:= \Omega^c_{-1,l}, \\
      & \cdots \hspace{25mm} \cdots
\end{align*}
Write $V_0 = V_0^- = V_0^+.$
Put $m(l) = m(-l,1) (=m(-1,l))$ and let
$$
V_l^- = \{ C_1^{-l,1}, C_2^{-l,1},\dots, C_{m(l)}^{-l,1}\}, \qquad
V_l^+ = \{ C_1^{-1,l}, C_2^{-1,l},\dots, C_{m(l)}^{-1,l}\}.
$$
Through the bijective correspondence
$$
x \in V_l^- \longleftrightarrow \sigma^{1-l}(x) \in V_l^+,
$$
the classes $ C_{i}^{-l,1}$ and $C_{i}^{-1,l}$ for each $i=1,2,\dots,m(l)$ 
are identified with each other and denoted by 
$C_i^l$.
Hence the sets 
 $V_l^-$ and $V_l^+$ 
are identified  for each $l\in \Zp.$
They are  denoted by $V_l.$
We regard $C_i^l$ 
as a vertex denoted by $v_i^l$ for $i=1,2,\dots,m(l),$
and define an edge $e^+$ labeled $\alpha\in \Sigma$ from
$v_i^l$ to $v_j^{l+1}$
if $\alpha C_j^{l+1}\subset C_i^l$.
We write $s(e^+) = v_i^l$ the source vertex of $e^+$
and
$t(e^+) = v_j^{l+1}$ the terminal vertex of $e^+,$
and the label $\lambda^+(e^+) = \alpha.$
The set of such edges from
$v_i^l$ to $v_j^{l+1}$ for some $i=1,2,\dots,m(l), \, j=1,2,\dots, m(l+1)$
is denoted by
$E_{l,l+1}^+.$
This situation is written
\begin{equation*}
v_i^l \overset{\alpha}{\underset{e^+}{\longrightarrow}} v_j^{l+1}
\qquad \text{ if } \qquad \alpha C_j^{l+1}\subset C_i^l.
\end{equation*}
Similarly 
we define an edge $e^-$ labeled $\beta\in \Sigma$ from
$v_j^{l+1}$ to $v_i^l$  
if $C_j^{l+1}\beta \subset C_i^l$.
We write $s(e^-) = v_j^{l+1}$ the source vertex of $e^-$
and
$t(e^-) = v_i^l$ the terminal vertex of $e^-,$
and the label $\lambda^-(e^-) = \beta.$
The set of such edges from
 $v_j^{l+1}$ to $v_i^l$  for some 
$i=1,2, \dots,m(l), \, j=1,2, \dots, m(l+1)$
is denoted by
$E_{l+1,l}^-.$
This situation is written
\begin{equation*}
v_j^{l+1} \overset{\beta}{\underset{e^-}{\longrightarrow}} v_i^l
\qquad \text{ if }\qquad C_j^{l+1}\beta \subset C_i^l.
\end{equation*}
We set 
$
E^+ =\bigcup_{l=0}^{\infty} E_{l,l+1}^+, \,
E^- =\bigcup_{l=0}^{\infty} E_{l+1,l}^-.
$ 
The natural labeling maps 
$\lambda^+: E^+ \longrightarrow \Sigma$
and 
$\lambda^-: E^- \longrightarrow \Sigma$
are denoted by 
$\lambda^+$ and $\lambda^-,$
respectively.
We now have  the pair 
$ {\frak L}^+_\Lambda =(V,E^+, \lambda^+)$ and
$ {\frak L}^-_\Lambda =(V,E^-, \lambda^-)$
of labeled Bratteli diagrams.
\begin{proposition}\label{prop:5.4}
The pair $\LLGBS$ of labeled Bratteli diagrams
is a $\lambda$-graph bisystem satisfying FPCC and presenting the subshift $\Lambda.$ 
\end{proposition}
\begin{proof}
We first show that 
the pair $\LLGBS$ satisfies the local property of $\lambda$-graph bisystem.
For $v_i^{l}(= C_i^{l}) \in V_{l}, \, v_j^{l+2} (= C_j^{l+2}) \in V_{l+2},$ 
take
$(e^-, e^+) \in E_+^-(v_i^{l}, v_j^{l+2})$ so that 
$t(e^-) = v_i^{l}, \,t(e^+)=v_j^{l+2} $ and 
$s(e^-) = s(e^+)$ denoted by $v_k^{l+1}.$
This means that 
$\lambda^+(e^+) C_j^{l+2} \subset C_k^{l+1}$
and
$C_k^{l+1} \lambda^-(e^-) \subset C_i^{l}.$
Put 
$\alpha = \lambda^+(e^+), \, \beta =\lambda^-(e^-) \in \Sigma.$
For 
$x = (x_n)_{n\in \Z} \in C_j^{l+2}$ with 
$x \in C_j^{-(l+2),1},$   
take
$\nu =(\nu_{-l},\nu_{-l+1},\dots,\nu_{-1}) \in B_{l}(\Lambda)$
with
$(\beta, \nu_{-l},\nu_{-l+1},\dots,\nu_{-1},\alpha) \in W_{-(l+2),1}(x)$
so that 
\begin{equation*}
 x(\beta,\nu\alpha) = 
(\dots, x_{-(l+3)}, x_{-(l+2)}, \beta,\nu_{-l},\nu_{-l+1},\dots,\nu_{-1}, \alpha, x_1, x_2,\dots ) \in \Lambda
\end{equation*}
where $\nu \alpha = (\nu_{-l},\nu_{-l+1},\dots,\nu_{-1}, \alpha).$
Let $C_{k'}^{l+1}$ be the class of $ x(\beta,\nu\alpha) $
in $\Omega_{-{(l+1)},1}^c,$ 
so that we have
$C_j^{l+2}\beta \subset C_{k'}^{l+1},$
and there exists an edge, denoted by $f^-,$
from $v_j^{l+2}$ to $v_{k'}^{l+1}$ labeled $\beta.$
The class of  $ x(\beta,\nu\alpha) $
in $\Omega_{-(l+1),0}^c$ 
is identified with the class of
 $ \sigma^{-l}(x(\beta,\nu\alpha)) $
in $\Omega_{-1, l}^c.$ 
As 
$\lambda^+(e^+) C_j^{l+2}\lambda^-(e^-) \subset C_i^{l},$
it belongs to $C_i^{l},$
so that 
there exists an edge, denoted by $f^+,$ 
from $v_i^{l}$ to $v_{k'}^{l}$
such that 
$\lambda^+(f^+) = \alpha.$
We thus have 
$(f^+,f^-) \in E^+_-(v_i^{l}, v_j^{l+2}).$
It is easy to see that 
the correspondence
$$
(e^-, e^+) \in E_+^-(v_i^{l}, v_j^{l+2})
\longrightarrow
(f^+,f^-) \in E^+_-(v_i^{l}, v_j^{l+2})
$$
gives rise to a desired bijection  
between
$E_+^-(v_i^{l}, v_j^{l+2})$
and
$E_-^+(v_i^{l}, v_j^{l+2}).$

\medskip

We second show that 
both ${\frak{L}}_\Lambda^-$ and ${\frak{L}}_\Lambda^+$
present the subshift $\Lambda.$
Let $W_{{\frak{L}}_\Lambda^+}$
be the set of admissible words appearing in the labeled graph
${\frak{L}}_\Lambda^+$
defined before Lemma \ref{lem:5.1}.
For any word 
$\mu =(\mu_1,\mu_2,\dots,\mu_l) \in W_{{\frak{L}}_\Lambda^+},$
there exist
$e_m \in E_{m-1,m}^+, m=n+1, \dots, n+l$
such that 
$\mu_1 =\lambda^+(e_{n+1}), 
 \mu_2 =\lambda^+(e_{n+2}),\dots,
\mu_l =\lambda^+(e_{n+l})
$ and 
$t(e_m)= s(e_{m+1}), m= n+1,\dots, n+l -1
$
for some $n \in \Z_+.$
Let
$v_i^n = s(e_{n+1}),\,
v_j^{n+l-1} = t(e_{n+l}).
$
We then have
$$
\lambda^+(e_{n+1})\lambda^+(e_{n+2})\cdots
\lambda^+(e_{n+l}) C_i^l \subset C_j^{n+l-1}.
$$
This means that $(\mu_1,\dots,\mu_l) \in B_l(\Lambda).$

Conversely
it is obvious that any word $\mu =(\mu_1,\dots,\mu_l) \in B_l(\Lambda)$
appears as a labeled path in ${\frak L}_\Lambda^+.$
Hence the set 
$W_{{\frak{L}}_\Lambda^+}$
coincides with the set $B_*(\Lambda)$ of admissible words of $\Lambda,$
so that  ${\frak L}_\Lambda^+$ and similarly  ${\frak L}_\Lambda^-$
present $\Lambda$. 
\end{proof}
\begin{definition}
The $\lambda$-graph bisystem 
$({\frak L}_\Lambda^-, {\frak L}_\Lambda^+)$
for a subshift $\Lambda$ is called the {\it canonical}\/  $\lambda$-graph bisystem 
for $\Lambda$. 
Its symbolic matrix bisystem 
$({\M}_\Lambda^-, {\M}_\Lambda^+)$
is called the {\it canonical}\/ symbolic matrix bisystem for $\Lambda$.   
\end{definition}
The $\lambda$-graph bisystems presented in Example \ref{ex:3.2}
(ii), (iii) and (iv) are the canonical $\lambda$-graph bisystems for the full $N$-shift,
the golden mean shift and the even shift, respectively.


\section{Strong shift equivalence}
In the theory of subshifts, one of most interesting and important subjects 
is their classification.
R. Williams  in \cite{Williams} proved that two topological Markov shifts are topologically conjugate
if and only if their underlying matrices have a special algebraic relation,
called strong shift equivalence.
His result and its proof  have been giving a great influence on 
further reserch of symbolic dynamical systems
(cf. \cite{LM}).
After Williams, M. Nasu in \cite{Nasu},
generalized Williams's result to sofic shifts.
The author introduced  a notion of (properly) strong shift equivalence in symbolic matrix systems 
in \cite{MaDocMath1999} (cf. \cite{MaETDS2003}) and 
proved that two subshifts are topologically conjugate 
if and only if their canonical symbolic matrix systems are (properly) strong shift equivalent.

As seen in the preceding section, any subshift is presented by a $\lambda$-graph bisystem
satisfying FPCC,
and hence by a symbolic matrix bisystem having a common alphabet.
In the first part of this section, 
we will introduce a notion of properly strong shift equivalence in 
symbolic matrix bisystems satisfying FPCC, and prove that 
 two subshifts are topologically conjugate 
if and only if their canonical symbolic matrix bisystems are properly strong shift equivalent.
Throughout the first part of this section, 
we assume that all $\lambda$-graph bisystems and symbolic matrix bisystems have common alphabets 
and satisfy FPCC.

Let
$\A$ and $\A'$ be symbolic matrices over alphabets 
$\Sigma$ and $\Sigma'$ respectively.
Let 
 $\phi$ be a bijection from a subset of $\Sigma$ onto a subset of $\Sigma'.$
Recall that 
 $\A$ and $\A'$ are said to be  specified equivalent under specification
 $\phi$ if $\A'$ can be obtained from $\A$
  by replacing every symbol 
 $\alpha$ appearing in components of $\A$ by $\phi(\alpha)$.
 We write it as
 $\A \overset{\phi}{\simeq} \A'$. 
 We call $\phi$ a specification from $\Sigma$ to $\Sigma'$.
If we do not specify the specification $\phi$, we simply write it
 $\A \simeq \A'$.
For an alphabet $\Sigma,$
we denote by ${\frak S}_\Sigma$ the set of finite formal sums of elements of $\Sigma.$
For alphabets $C,D$,
 put
$C \cdot D = \{ cd \, | \, c \in C,d \in D \}.$
For 
$x = {\displaystyle\sum_{j}}c_j \in {\frak S}_C$
and
$ y ={\displaystyle\sum_{k}} d_k \in {\frak S}_D$,
 define
 $xy = {\displaystyle\sum_{j,k}} c_jd_k \in {\frak S}_{C\cdot D}$.
Recall that 
the exchanging specification $\kappa$ from $C\cdot D$ to $D\cdot C$
is a bijection 
from ${\frak S}_{C\cdot D}$ to ${\frak S}_{D\cdot C}$ 
defined by
$$
\kappa(\sum_{j,k} c_jd_k ) =
\sum_{j,k} d_k c_j \in {\frak S}_{D\cdot C}
\quad\text{ for }
\sum_{j,k} c_jd_k \in {\frak S}_{C\cdot D}.
$$

We first define properly strong shift equivalence in $1$-step 
between two symbolic matrix bisystems satisfying FPCC 
as a generalization of 
strong shift equivalences in $1$-step between
two nonnegative matrices defined by R. Williams in \cite{Williams},
 and between
two symbolic  matrices defined by Nasu in \cite{Nasu}.

Let $(\M^-,\M^+)$ and $(\SN^{-}, \SN^{+})$ be symbolic matrix bisystems over alphabets
$\Sigma_\M$ and $\Sigma_{\SN}$ respectively,
both of them satisfy FPCC,
where
$\M^-_{l,l+1}, \M^+_{l,l+1}$ are $m(l)\times m(l+1)$ symbolic matrices and
$\SN^-_{l,l+1}, \SN^+_{l,l+1}$ are $n(l)\times n(l+1)$ symbolic matrices.

\begin{definition}\label{def:PSSE}
Two symbolic matrix bisystems 
 $(\M^-,\M^+)$ and $(\SN^-, \SN^+)$
 are said to be {\it properly strong shift equivalent in}\/ $1$-{\it step}\/  
 if there exist alphabets 
$C,D$ and specifications 
$$
 \varphi: \Sigma_\M \rightarrow C\cdot D,
 \qquad
 \phi: \Sigma_\SN \rightarrow D \cdot C
$$
 and  sequences $c(l),d(l)$ on $l \in \Zp$
 such that for each $l\in \Zp$, there exist
\begin{enumerate}
\renewcommand{\theenumi}{\arabic{enumi}}
\renewcommand{\labelenumi}{\textup{(\theenumi)}}
\item a $c(l)\times d(l+1)$ matrix ${\P}_l$ over $C,$ 
\item a $d(l)\times c(l+1)$ matrix ${\Q}_l$ over $D,$
\item a $d(l)\times d(l+1)$ matrix $\X_l $ over   $D$ for $l$ being odd,
\item a $c(l)\times c(l+1)$ matrix $\X_l $ over   $D$ for $l$ being even,
\item a $c(l)\times c(l+1)$ matrix $\Y_l $ over   $C$ for $l$ being odd,
\item a $d(l)\times d(l+1)$ matrix $\Y_l $ over   $C$ for $l$ being even,
\end{enumerate}
 satisfying the following equations: 
\begin{gather}
\M^+_{l,l+1} 
\overset{\varphi}{\simeq} {\P}_{2l}{\Q}_{2l+1},\qquad
\SN^+_{l,l+1} 
\overset{\phi}{\simeq} {\Q}_{2l}{\P}_{2l+1}, \label{eq:PSSE1} \\
\M^-_{l,l+1} \overset{\kappa\varphi}{\simeq} \X_{2l}\Y_{2l+1}, \qquad 
\SN^-_{l,l+1} \overset{\kappa\phi}{\simeq} \Y_{2l} \X_{2l+1} \\ 
\intertext{and}
\Y_{2l+1}{\P}_{2l+2} \overset{\kappa}{\simeq} {\P}_{2l+1}\Y_{2l+2},\qquad
\X_{2l+1}{\Q}_{2l+2}\overset{\kappa}{\simeq} {\Q}_{2l+1}\X_{2l+2}, \label{eq:PSSE3}\\ 
\X_{2l}{\P}_{2l+1} \overset{\kappa}{\simeq} {\P}_{2l}\X_{2l+1},\qquad
\Y_{2l}{\Q}_{2l+1}\overset{\kappa}{\simeq} {\Q}_{2l}\Y_{2l+1}, \label{eq:PSSE4}
\end{gather}
where $\kappa$ is the exchanging specification defined by
$\kappa(a\cdot b) = b\cdot a$,
and $\kappa \varphi, \kappa\phi$ denote the compositions
$\kappa \circ\varphi, \kappa\circ\phi$, respectively.

\medskip
We write this situation as
$(\M^-,\M^+)\underset{1-\pr}{\approx} (\SN^-, \SN^+).
$
By \eqref{eq:PSSE1}, we know 
$c(2l) = m(l)$ and $ d(2l) = n(l)$ for $l \in \Zp$.

Two symbolic matrix bisystems 
 $(\M^-,\M^+)$ and $(\SN^-,\SN^+)$ 
 are said to be {\it  properly strong shift equivalent in }\/ $\ell$-{\it step}\/ 
 if there exists a sequence of  symbolic matrix bisystems
$(\M_{(i)}^-,\M_{(i)}^+),\, i=1,2,\dots, \ell-1$
such that
\begin{equation*}
(\M^-,\M^+) \underset{1-\pr}{\approx} (\M_{(1)}^-,\M_{(1)}^+)
       \underset{1-\pr}{\approx} 
       \cdots 
       \underset{1-\pr}{\approx} 
(\M_{(\ell-1)}^-,\M_{(\ell-1)}^+) \underset{1-\pr}{\approx} 
( \SN^-, \SN^+).
\end{equation*}
We denote this situation by
$(\M^-, \M^+) \underset{\ell-\pr}{\approx} (\SN^-,\SN^+)
$
and simply call it a {\it properly strong shift equivalence}.
\end{definition}
\begin{proposition}\label{prop:DM4.1}
Properly strong shift equivalence in symbolic matrix bisystems is an equivalence relation.
\end{proposition}
\begin{proof}
It is clear that properly
strong shift equivalence is symmetric and transitive.
It suffices  to show that  
$
(\M^-,\M^+) \underset{1-\pr}{\approx} (\M^-,\M^+).
$
Put
$
C = \Sigma_\M, \, D = \{0,1\}.
$
Define
$
\varphi:a \in \Sigma_\M \rightarrow a\cdot 1 \in C \cdot D
$
and
$ 
\phi:a \in \Sigma_\M \rightarrow 1\cdot a \in D \cdot C.
$
Let
$E_k $ be the $k \times k$ identity matrix.
Set
$
c(2l) = c(2l+1) = d(2l) = m(l), \, d(2l+1) = m(l+1) 
$
for
$ l \in \Zp,$
and
\begin{gather*}
\P_{2l} = \P_{2l+1} = \M_{l,l+1}^+, \quad 
  \Q_{2l} = E_{m(l)}, \quad  \Q_{2l+1} = E_{m(l+1)},\\
\Y_{2l} = \Y_{2l+1} = \M_{l,l+1}^-,\quad
\X_{2l} = E_{m(l)},\quad
\X_{2l+1} = E_{m(l+1)}.
\end{gather*}
It is straightforward to see 
that they give a properly strong shift equivalence in 1-step 
between 
$(\M^-,\M^+)$ and $(\M^-,\M^+)$.
\end{proof}

We will prove the following theorem.
\begin{theorem}\label{thm:DM4.2}
Two  subshifts are topologically conjugate
if and only if  
their canonical symbolic matrix bisystems
are properly strong shift equivalent.
\end{theorem}

We will first show the only if part of the theorem above.
In our proof,
we will use Nasu's factorization theorem for topological conjugacy 
between subshifts into bipartite codes (\cite{Nasu}).
We now introduce the notion of bipartite symbolic matrix bisystem.

\begin{definition}\label{def:bipartitesmbs}
A symbolic matrix bisystem $(\M^-,\M^+)$ over 
common alphabet $\Sigma$ is said to be {\it bipartite}
if there exist disjoint subsets 
$C,D \subset \Sigma$ with $\Sigma = C\sqcup D$
 and sequences $c(l), d(l)$ on $l \in \Zp$
 with 
$m(l) = c(l) + d(l), \, l \in \N$
 such that
 for each $l\in \Zp$, 
 there exist
\begin{enumerate}
\renewcommand{\theenumi}{\arabic{enumi}}
\renewcommand{\labelenumi}{\textup{(\theenumi)}}
\item a $c(l)\times d(l+1)$ matrix ${\P}_{l,l+1}$ over $C,$ 
\item a $d(l)\times c(l+1)$ matrix ${\Q}_{l,l+1}$ over $D,$
\item a $d(l)\times d(l+1)$ matrix $\X_{l,l+1} $ over   $D$ for $l$ being odd,
\item a $c(l)\times c(l+1)$ matrix $\X_{l,l+1} $ over   $D$ for $l$ being even,
\item a $c(l)\times c(l+1)$ matrix $\Y_{l,l+1} $ over   $C$ for $l$ being odd,
\item a $d(l)\times d(l+1)$ matrix $\Y_{l,l+1} $ over   $C$ for $l$ being even,
\end{enumerate}
satisfying the following equations: 
\begin{equation}
\M^+_{l,l+1}
=
\begin{bmatrix}
0 & {\P}_{l,l+1} \\
{\Q}_{l,l+1} & 0
\end{bmatrix},  \qquad  
\M^-_{l,l+1} 
=
\begin{cases}
{\begin{bmatrix}
 \Y_{l,l+1} & 0 \\
 0 & \X_{l,l+1} 
\end{bmatrix}}
& \text{ if } l \text{ is odd}, \\[14pt] 
{\begin{bmatrix}
 \X_{l,l+1} & 0 \\
 0 & \Y_{l,l+1} 
\end{bmatrix}}
& \text{ if } l \text{ is even}.
\end{cases} \label{eq:5.4MM}
\end{equation}
Under the assumption that $(\M^-, \M^+)$ is standard so that
$m(0) =1$,  we require
that $c(0) = d(0) =1$ and the above equalities
\eqref{eq:5.4MM} for $l =0$  mean  
\begin{equation}
\M^+_{0,1}
=
\begin{bmatrix}
{\Q}_{0,1}\,\,\,   {\P}_{0,1} 
\end{bmatrix},  \qquad  
\M^-_{0,1} 
=
\begin{bmatrix}
 {\X}_{0,1} \,\,\, {\Y}_{0,1}
\end{bmatrix}. \label{eq:5.40MM}
\end{equation}
\end{definition}
We thus see
\begin{lemma}\label{lem:DM4.3}
For a bipartite symbolic matrix bisystem  $(\M^-,\M^+)$ as above,
put
$$
{\P}_l = {\P}_{l,l+1},\quad
{\Q}_l = {\Q}_{l,l+1},\quad
\Y_l = \Y_{l,l+1},\quad
\X_l = \X_{l,l+1}
$$
and
$$
{\M}_{l,l+1}^{CD+} 
:= {\P}_{2l}{\Q}_{2l+1},\qquad
{\M}_{l,l+1}^{DC+} 
:= {\Q}_{2l}{\P}_{2l+1},
$$
$$
\M_{l,l+1}^{CD-} 
:\overset{\kappa}{\simeq} \X_{2l}\Y_{2l+1}, \qquad 
\M_{l,l+1}^{DC-} 
:\overset{\kappa}{\simeq} \Y_{2l}\X_{2l+1}.  
$$
Then the both pairs 
$(\M^{CD-}, \M^{CD+})$
and
$(\M^{DC-},\M^{DC+})$
are
symbolic matrix bisystems over alphabets $C\cdot D$ and $D \cdot C$
respectively and
 they are properly strong shift equivalent in 1-step,
where
$
\M_{l,l+1}^{CD-} 
:\overset{\kappa}{\simeq} \X_{2l}\Y_{2l+1} 
$
means that
$
\M_{l,l+1}^{CD-}(i,j)
$
is defined by 
$ 
\kappa(\X_{2l}\Y_{2l+1}(i,j)) \in C\cdot D
$ for all $i=1,2,\dots,m(l),\, j= 1,2,\dots,m(l+1)$,
and
$\M_{l,l+1}^{DC-} 
:\overset{\kappa}{\simeq} \Y_{2l}\X_{2l+1}
$ 
is similarly defined.  
\end{lemma}
\begin{proof}
By the relations
$
\M^-_{l,l+1} \M^+_{l+1,l+2} \overset{\kappa}{\simeq} \M^+_{l,l+1}\M^-_{l+1,l+2}
$
for $l\in \Zp$,
we have for $l$ being odd, 
\begin{equation*}
\begin{bmatrix}
 \Y_{l} & 0 \\
 0 & \X_{l} 
\end{bmatrix}
\begin{bmatrix}
0 & \P_{l+1} \\
 \Q_{l+1} & 0
\end{bmatrix}
\overset{\kappa}{\simeq}
\begin{bmatrix}
0 & \P_{l} \\
 \Q_{l} & 0
\end{bmatrix}
\begin{bmatrix}
 \X_{l+1} & 0 \\
 0 & \Y_{l+1} 
\end{bmatrix}
\end{equation*}
so that 
\begin{equation}
\Y_l \P_{l+1} \overset{\kappa}{\simeq} \P_l \Y_{l+1},
\qquad
\X_l \Q_{l+1} \overset{\kappa}{\simeq} \Q_l \X_{l+1}. \label{eq:5.51}
\end{equation}
For  $l$ being even, 
\begin{equation*}
\begin{bmatrix}
 \X_{l} & 0 \\
 0 & \Y_{l} 
\end{bmatrix}
\begin{bmatrix}
0 & \P_{l+1} \\
 \Q_{l+1} & 0
\end{bmatrix}
\overset{\kappa}{\simeq}
\begin{bmatrix}
0 & \P_{l} \\
 \Q_{l} & 0
\end{bmatrix}
\begin{bmatrix}
 \Y_{l+1} & 0 \\
 0 & \X_{l+1} 
\end{bmatrix}
\end{equation*}
so that 
\begin{equation}
\X_l \P_{l+1} \overset{\kappa}{\simeq} \P_l \X_{l+1},
\qquad
\Y_l \Q_{l+1} \overset{\kappa}{\simeq} \Q_l \Y_{l+1}. \label{eq:5.52}
\end{equation}
We then have 
\begin{align*}
\M^{CD-}_{l,l+1}\cdot\M^{CD+}_{l+1,l+2}
\simeq & \X_{2l}\Y_{2l+1}{\P}_{2l+2}{\Q}_{2l+3} \\
\simeq & \X_{2l}{\P}_{2l+1} \Y_{2l +2}{\Q}_{2l+3} \\
\simeq & \P_{2l} \X_{2l+1} \Q_{2l+2} \Y_{2l+3} \\
\simeq & \P_{2l} \Q_{2l+1} \X_{2l+2} \Y_{2l+3} \\
\simeq & \M^{CD+}_{l,l+1}\cdot\M^{CD-}_{l+1,l+2}, \\
\intertext{and}
\M^{DC+}_{l,l+1}\cdot\M^{DC-}_{l+1,l+2}
\simeq & \Y_{2l} \X_{2l+1}{\Q}_{2l+2}{\P}_{2l+3} \\
\simeq & \Y_{2l}{\Q}_{2l+1} \X_{2l +2}{\P}_{2l+3} \\
\simeq & \Q_{2l} \Y_{2l+1} \P_{2l+2} \X_{2l+3} \\
\simeq & \Q_{2l} \P_{2l+1} \Y_{2l+2} \X_{2l+3} \\
\simeq & \M^{DC-}_{l,l+1}\cdot\M^{DC+}_{l+1,l+2}.
\end{align*}
By looking at  the above specified equivalences, 
we know that 
\begin{align*}
\M^{CD-}_{l,l+1}\cdot\M^{CD+}_{l+1,l+2}
\overset{\kappa}{\simeq}
& \M^{CD+}_{l,l+1}\cdot\M^{CD-}_{l+1,l+2}, \\
\intertext{and}
\M^{DC+}_{l,l+1}\cdot\M^{DC-}_{l+1,l+2}
\overset{\kappa}{\simeq}
& \M^{DC-}_{l,l+1}\cdot\M^{DC+}_{l+1,l+2}.
\end{align*}
Hence
both pairs 
$(\M^{CD-},\M^{CD+})$
and
$(\M^{DC-},\M^{DC+})$
are
symbolic matrix bisystems.
The corresponding equations to 
\eqref{eq:PSSE3} and \eqref{eq:PSSE4}
come from \eqref{eq:5.51}, \eqref{eq:5.52}
 so that 
 they are properly strong shift equivalent in 1-step.
\end{proof}

\begin{definition}
A $\lambda$-graph bisystem
$\LGBS$
over common alphabet
$\Sigma$ is said to be {\it bipartite}
if
there exist disjoint subsets $C,D \subset \Sigma$ such that
$\Sigma = C\cup D$ and disjoint subsets 
$V^C_l,V^D_l \subset V_l$ for each $l \in \Zp$
such that
$V^C_l \cup V^D_l = V_l$ and
\begin{enumerate}
\renewcommand{\theenumi}{\arabic{enumi}}
\renewcommand{\labelenumi}{\textup{(\theenumi)}}
\item
for each $e^+ \in E^+_{l,l+1},$
\begin{align*}
\lambda^+(e^+) \in C
\quad\text{ if and only if }&\quad 
s(e^+) \in  V^C_l, \quad t(e^+) \in V^D_{l+1}, \\ 
\lambda^+(e^+) \in D
\quad\text{ if and only if }& \quad
s(e^+) \in  V^D_l, \quad t(e^+) \in V^C_{l+1},
\end{align*}
\item
for each $e^- \in E^-_{l+1,l},$
\begin{align*}
\lambda^-(e^-) \in C
\quad\text{ if and only if }&\quad 
{\begin{cases}
s(e^-) \in  V^C_{l+1}, \quad t(e^-) \in V^C_{l} & \text{ for } l \text{ being odd,} \\ 
s(e^-) \in  V^D_{l+1}, \quad t(e^-) \in V^D_{l} & \text{ for } l \text{ being even,}
\end{cases}} \\
\lambda^-(e^-) \in D
\quad\text{ if and only if }&\quad 
{\begin{cases}
s(e^-) \in  V^D_{l+1}, \quad t(e^-) \in V^D_{l} & \text{ for } l \text{ being odd,} \\ 
s(e^-) \in  V^C_{l+1}, \quad t(e^-) \in V^C_{l} & \text{ for } l \text{ being even.}
\end{cases}} 
\end{align*}
\end{enumerate}
\end{definition}

\begin{proposition}\label{lem:DM4.4}
A symbolic matrix bisystem is bipartite 
if and only if the associated $\lambda$-graph bisystem
 is bipartite.
\end{proposition}
\begin{proof}
It is clear that
a bipartite symbolic matrix bisystem 
gives rise to a bipartite
$\lambda$-graph bisystem.
Conversely,
suppose that
a $\lambda$-graph bisystem 
$\LGBS$
is bipartite.
Let $c(l)$ and $d(l)$ be the cardinalities of the sets
$V^D_l$ and $V^C_l$ respectively.
We may identify
$V^D_l$ and $V^C_l$ with
the sets
$\{1,2, \dots,c(l)\}$ and
$\{1,2, \dots,d(l)\}$ respectively.
For $i \in V^C_l, j \in V^D_{l+1}$, 
put 
$$
{\P}_{l,l+1}(i,j) = \lambda^+(e^+_1) + \cdots + \lambda^+(e^+_{n_p})
$$
where
$e^+_k\in E^+_{l,l+1}, k=1, 2, \dots,n_p$ are all edges of
$E^+_{l,l+1}$ satisfying
$s(e^+_k)= i, t(e^+_k) = j,$
so that ${\P}_{l,l+1}(i,j)\in {\frak S}_C.$
Similarly we define
for $i \in V^D_l, j \in V^C_{l+1}$, 
put 
$$
{\Q}_{l,l+1}(i,j) = \lambda^+(f^+_1) + \cdots + \lambda^+(f^+_{n_q})
$$
where
$f^+_k\in E^+_{l,l+1}, k=1, 2, \dots,n_q$ are  all edges of
$E^+_{l,l+1}$ satisfying
$s(f^+_k)= i, t(f^+_k) = j,$
so that $\Q_{l,l+1}(i,j) \in {\frak S}_D.$
For $i \in V^C_l, j \in V^C_{l+1}$ with $l$ being odd, 
put 
$$
\Y_{l,l+1}(i,j) = \lambda^-(e^-_1) + \cdots + \lambda^-(e^-_{n_y})
$$
where
$e^-_k\in E^-_{l+1,l}, k=1, 2, \dots,n_y$ are  all edges of
$E^-_{l+1,l}$ satisfying
$s(e^-_k)= j, t(e^-_k) =i,$
so that $\Y_{l,l+1}(i,j)\in {\frak S}_C.$
For $i \in V^C_l, j \in V^C_{l+1}$ with $l$ being even, 
put 
$$
\X_{l,l+1}(i,j) = \lambda^-(f^-_1) + \cdots + \lambda^-(f^-_{n_x})
$$
where
$f^-_k\in E^-_{l+1,l}, k=1, 2, \dots,n_x$ are all edges of
$E^-_{l+1,l}$ satisfying
$s(f^-_k)= j, t(f^-_k) =i,$
so that $\X_{l,l+1}(i,j)\in {\frak S}_D.$
For $i \in V^D_l, j \in V^C_{l+1}$,
we similarly define 
 $\Y_{l,l+1}(i,j)\in {\frak S}_C$ for $l$ being even, 
and  $\X_{l,l+1}(i,j)\in {\frak S}_D$ for $l$ being odd.
Let $(\M^-, \M^+)$ be the corresponding symbolic matrix bisystem for $\LGBS.$
It is now clear that 
they satisfy the equalities \eqref{eq:5.4MM}.
Then
the symbolic matrix bisystem $(\M^-, \M^+)$ for
$\LGBS$
is  bipartite. 
\end{proof}

Nasu introduced the notion of bipartite subshift in \cite{Nasu} and \cite{NasuMemoir}.
 A subshift $\Lambda$ over alphabet $\Sigma$ is said to be
 bipartite if there exist disjoint subsets
 $C, D \subset \Sigma$ with $\Sigma = C\sqcup D$ such that
 any $(x_n)_{n \in \Z} \in \Lambda$ is either
$$
 x_n \in C \text{ and } x_{n+1} \in D \text{ for all } n \in \Z
\quad
\text{  or }
\quad
 x_n \in D \text{ and } x_{n+1} \in C \text{ for all } n \in \Z.
$$
Let
$\Lambda^{(2)}$ be the 2-higher power shift for $\Lambda$ 
that is defined by the subshift 
\begin{equation*}
\Lambda^{(2)} = \{ (x_{[2n, 2n+1]})_{n \in \Z} \in (\Sigma^2)^\Z 
\mid (x_n)_{n \in \Z} \in \Lambda\}
\end{equation*}
over alphabet $\Sigma^2$ where
$x_{[2n, 2n+1]} = (x_{2n}, x_{2n+1}), n \in \Z.$
Put 
\begin{align*}
\Lambda_{CD} & = 
  \{ (c_id_i)_{i \in \Z} \in \Lambda^{(2)} \mid c_i \in C, d_i \in D \},\\
\Lambda_{DC} & =
  \{ (d_ic_{i+1})_{i \in \Z} \in \Lambda^{(2)} \mid  c_i \in C, d_i \in D \}.
\end{align*}
They are subshifts over alphabets $C\cdot D$ and $D\cdot C$ 
respectively.
Hence $\Lambda^{(2)}$ is partitioned into the two subshifts 
$\Lambda_{CD}$ and $\Lambda_{DC}$.

\begin{proposition}\label{prop:DM4.5}
A subshift $\Lambda$ is bipartite if and only if 
its canonical symbolic matrix bisystem $(\M^-, \M^+)$ is bipartite.
\end{proposition}
\begin{proof}
It is clear that
a bipartite  symbolic matrix bisystem yields a bipartite $\lambda$-graph bisystem, that  gives rise to 
a bipartite subshift by its construction of the subshift from the 
$\lambda$-graph bisystem.

Suppose that
$\Lambda$ 
is bipartite with respect to alphabets $C,D$.
It suffices to show that its canonical $\lambda$-graph bisystem
$\LGBS$ is bipartite.
Let us denote by 
${\frak L}^- = (V^-, E^-, \lambda^-)$ and
${\frak L}^+ = (V^+, E^+, \lambda^+).$
Let $[x]_{k,l}\in \Omega_{k.l}^c$ 
denote the $(k,l)$-central equivalence class of $x \in \Lambda.$
Define for $Z, W =C\text{ or } D,$ 
\begin{align*}
V_l^{ZW-}= & \{[x]_{-l,1}\in V_l^- \mid x_{-l} \in Z, \, x_1 \in W \}, \\
V_l^{ZW+}= & \{[x]_{-1,l}\in V_l^+ \mid x_{-1} \in Z, \, x_l \in W \}.
\end{align*}
Since $\Lambda$ is bipartite,
we know that 
$V_l^{CD-}, \, V_l^{DC-}, \, V_l^{CD+}, \, V_l^{DC+} $ are all empty if $l$ is odd,
whereas
$V_l^{CC-}, \, V_l^{DD-}, \, V_l^{CC+}, \, V_l^{DD+} $ are all empty if $l$ is even
so that  
\begin{align*}
V_l^- =&
{\begin{cases}
V_l^{CC-} \cup V_l^{DD-} & \text{ if } l \text{ is odd,} \\
V_l^{DC-} \cup V_l^{CD-} & \text{ if } l \text{ is even,}  
\end{cases}} \\
V_l^+ =&
{\begin{cases}
V_l^{CC+} \cup V_l^{DD+} & \text{ if } l \text{ is odd,} \\
V_l^{DC+} \cup V_l^{CD+} & \text{ if } l \text{ is even.}  
\end{cases}} \\
\end{align*}
Let $\pi: x \in V_l^+ \longrightarrow \sigma^{l-1}(x)\in V_l^-$ 
be the bijection that satisfies for $Z, W = C \text{ or } D,$
$\pi(V_l^{ZW+}) = V_l^{ZW-}, l \in \Zp.$
We identify 
$V_l^+$ with $V_l^-$ through the map 
$\pi: V_l^+ \longrightarrow  V_l^-.$
We set 
$V_l := V_l^-$ and
$V_l^C:=V_l^{ZC-},\, V_l^D:=V_l^{ZD-}$ for $Z =C \text{ or } D.$ 
 Then we have
$$
V_l = V_l^C \sqcup V_l^D \quad \text{ for all } l \in \N
\quad \text{ and }
\quad V_0 = V_0^C = V_0^D =\{ \emptyset\}.
$$
We regard $V_0,  V_0^C, V_0^D$ as all singletons. 
For each $e^- \in E^-_{l+1,l},$
it is easy to see that 
\begin{align*}
\lambda^-(e^-) \in C
\quad\text{ if and only if }&\quad 
{\begin{cases}
s(e^-) \in  V^{DC-}_{l+1}, \quad t(e^-) \in V^{CC-}_{l} & \text{ for } l \text{ being odd,} \\ 
s(e^-) \in  V^{DD-}_{l+1}, \quad t(e^-) \in V^{CD-}_{l} & \text{ for } l \text{ being even,}
\end{cases}} \\
\lambda^-(e^-) \in D
\quad\text{ if and only if }&\quad 
{\begin{cases}
s(e^-) \in  V^{CD-}_{l+1}, \quad t(e^-) \in V^{DD-}_{l} & \text{ for } l \text{ being odd,} \\ 
s(e^-) \in  V^{CC-}_{l+1}, \quad t(e^-) \in V^{DC-}_{l} & \text{ for } l \text{ being even.}
\end{cases}} 
\end{align*}
and
for each $e^+ \in E^+_{l,l+1}$
\begin{align*}
\lambda^+(e^+) \in C
\quad\text{ if and only if }&\quad 
s(e^+) \in  V^{ZC+}_l, \quad t(e^+) \in V^{ZD+}_{l+1}, \\ 
\lambda^+(e^+) \in D
\quad\text{ if and only if }& \quad
s(e^+) \in  V^{WD+}_l, \quad t(e^+) \in V^{WC+}_{l+1}
\end{align*}
for $Z, W = C \text{ or } D.$
Therefore 
the $\lambda$-graph bisystem
$\LGBS$ is bipartite.
\end{proof}

Let $\Lambda$ be a bipartite subshift over $\Sigma$
with respect to alphabets
$C, D.$
As in Lemma \ref{lem:DM4.3}, we have
two symbolic matrix bisystems
$(\M^{CD-},\M^{CD+})$
and
$(\M^{DC-},\M^{DC+})$
over alphabets
$C\cdot D$ and
$D\cdot C$
from the bipartite canonical symbolic matrix bisystem 
$(\M^-_{\Lambda},\M^+_{\Lambda})$ for
$\Lambda$ respectively.
They are naturally identified with
the canonical symbolic matrix bisystems for the subshifts $\Lambda_{CD}$
and $\Lambda_{DC}$ respectively.
By Lemma \ref{lem:DM4.3}, we thus see a corollary below of Proposition \ref{prop:DM4.5}.
\begin{corollary}\label{cor:DM4.6}
For a bipartite subshift $\Lambda$ 
with respect to alphabets $C,D$,
we have a properly strong shift equivalence in 1-step:
$$
(\M^{CD-},\M^{CD+}) 
 \underset{1-\pr}{\approx} 
(\M^{DC-},\M^{DC+}).
$$
\end{corollary}

The following notion of bipartite conjugacy has been introduced by 
Nasu in \cite{Nasu}, \cite{NasuMemoir}.
The conjugacy from $\Lambda_{CD}$ onto $\Lambda_{DC}$ that maps
$(c_id_i)_{i\in \Z}$ to
$(d_ic_{i+1})_{i\in \Z}$ 
is called the forward bipartite conjugacy.
The conjugacy from $\Lambda_{CD}$ onto $\Lambda_{DC}$ that maps
$(c_id_i)_{i\in \Z}$ to
$(d_{i-1}c_{i})_{i\in \Z}$ 
is called the backward bipartite conjugacy.
A topological conjugacy between subshifts is called 
a symbolic conjugacy if it is a 1-block map given by a bijection between the underlying alphabets of the subshifts.
Nasu proved the following factorization theorem for topological conjugacies between subshifts. 

\begin{lemma}[{M. Nasu \cite{Nasu}}]
Any topological conjugacy $\psi$ between subshifts is factorized into
finite compositions of the form
$$
\psi = 
\kappa_n \zeta_n \kappa_{n-1} \zeta_{n-1} \cdots \kappa_1 \zeta_1 \kappa_0
$$
where
$\kappa_0,\dots, \kappa_n$ are symbolic conjugacies and
$\zeta_1,\dots, \zeta_n$ are either forward or backward bipartite conjugacies.
\end{lemma}
Thanks to the Nasu's  result above,
we reach the following theorem.
\begin{theorem}\label{thm:DM4.8}
If two subshifts are topologically conjugate,
their canonical symbolic matrix bisystems
are properly strong shift equivalent.
\end{theorem}

We will prove the converse implication of the theorem above.
We will indeed prove the following proposition.

\begin{proposition}\label{prop:DM4.9}
If two  symbolic matrix bisystems are properly strong shift equivalent in 1-step,
their presenting subshifts  are topologically conjugate. 
\end{proposition}

To prove the proposition, we provide a notation and a lemma.
For $(\M^-, \M^+)$, set  the $ m(l) \times m(l+k)$ matrices:
\begin{align*}
\M^-_{l,l+k} & = \M^-_{l,l+1}\cdot \M^-_{l+1,l+2}\cdots \M^-_{l+k-1,l+k}, \\
\M^+_{l,l+k} & = \M^+_{l,l+1}\cdot \M^+_{l+1,l+2}\cdots \M^+_{l+k-1,l+k}
\end{align*}
for each $l,k \in \Zp$.
Let us denote by $\Lambda_\M$ the presented subshift by $(\M^-, \M^+).$

\begin{lemma}\label{lemDM4.10}
Assume that two symbolic matrix bisystems
$(\M^-,\M^+)$ over $\Sigma_\M$ and
 $(\SN^-, \SN^+)$ over $\Sigma_\SN$ 
are properly strong shift equivalent in 1-step.
Let
$\varphi : \Sigma_\M \rightarrow C\cdot D$
and
$\phi : \Sigma_\SN \rightarrow D\cdot C$
be specifications 
that give rise to the properly strong shift equivalence in 1-step
between them.
For any word 
$x_1x_2 \in B_2(\Lambda_{\M})$ 
of length two in the presenting subshift $\Lambda_{\M}$,
put  
$\varphi(x_i) = c_id_i, i=1,2$ 
where
$c_i \in C,d_i \in D$.
Then there uniquely exists a symbol 
$y_0 \in \Sigma_\SN$ such that
$\phi(y_0) = d_1c_2$.
\end{lemma}
\begin{proof}
Note that by definition the specification
$\phi$ is not necessarily surjective onto $D\cdot C.$
Since $\phi$ is injective, 
it suffices to show the existence of $y_0$
such that $\phi(y_0) = d_1c_2$.
As
$x_1x_2 \in B_2(\Lambda_{\M})$, 
for any fixed
$l \ge 3$, we may find 
$j=1,2,\dots,m(l+2)$ and
$k=1,2,\dots,m(l)$ 
such that
$x_1x_2$ appears in $\M^+_{l,l+2}(k,j)$,
and hence in some component of
$\M^+_{l,l+1}\M^+_{l+1,l+2}.$
Under specifications appeared in Definition \ref{def:PSSE},
 we know  the following specified equivalences:
\begin{align*}
\M^-_{l-1,l} \M^+_{l,l+1}\M^+_{l+1,l+2} 
& \simeq 
\X_{2l-2} \Y_{2l-1} \P_{2l}  \Q_{2l+1} \P_{2l+2} \Q_{2l+3} \\
& \simeq 
\X_{2l-2}  \P_{2l-1} \Q_{2l} \P_{2l+1} \Q_{2l+2} \Y_{2l+3}.
\end{align*}
Since
$\varphi(x_1x_2) = c_1d_1c_2d_2$
that  appears in some component of 
 $\P_{2l} \Q_{2l+1} \P_{2l+2} \Q_{2l+3}$
 and hence of 
 $\P_{2l-1} \Q_{2l} \P_{2l+1} \Q_{2l+2}.$
This implies that 
the word
$d_1c_2$ appears in some  component of 
${\Q}_{2l}{\P}_{2l+1}.$
By the specified equivalence
$
{\SN}^+_{l,l+1} 
\overset{\phi}{\simeq} {\Q}_{2l}{\P}_{2l+1}
$
in \eqref{eq:PSSE1}, 
we may find a unique symbol $y_0$ in $\Sigma_\SN$ such that 
$\phi(y_0) = d_1c_2$.
\end{proof}

{\it Proof of Proposition \ref{prop:DM4.9}.}\/
Suppose that
 $(\M^-,\M^+)$ and $(\SN^-,\SN^+)$ 
are properly strong shift equivalent in 1-step.
We use the same notation as in Definition \ref{def:PSSE}.
By the preceding lemma, 
we have a $2$-block map $\varPhi$
from
$B_2(\Lambda_\M)$ to $\Sigma_\SN$ defined by
$\varPhi(x_1x_2) = y_0$
where
$\varphi(x_i) = c_id_i, i=1,2$
and
$\phi(y_0) = d_1c_2$.
Let
$\varPhi_{\infty}$ be
the sliding block code induced by
$\varPhi$
so that
$\varPhi_{\infty}$
is a map from
$\Lambda_\M$ to
${(\Sigma_\SN)}^{\Z}$
(see \cite{LM} for sliding block code).
We also write as 
$\varPhi$ the map from $B_*(\Lambda_\M)$ 
to the set of all words of 
$\Sigma_\SN$ defined by
$$
\varPhi(x_1x_2 \cdots x_n)= 
\varPhi(x_1x_2) 
\varPhi(x_2x_3)\cdots 
\varPhi(x_{n-1}x_n). 
$$
We will prove that
$\varPhi_{\infty}(\Lambda_\M) \subset \Lambda_\SN$.
To prove this,  it suffices to show that
for any word $w$ in $\Lambda_\M,$
$\varPhi(w)$ is an admissible word in $\Lambda_\SN$.
For
$w=w_1w_2\cdots w_n \in B_n(\Lambda_\M)$
and any
 fixed
$l \ge n+1$, we find 
$j=1,2,\dots,m(l+n)$ and
$k=1,2,\dots,m(l)$ 
such that
$w$ appears in $\M^+_{l,l+n}(k,j)$.
Take
$i=1,2,\dots,m(l-1)$ with
$\M^-_{l-1,l}(i,k) \ne 0,$
so that $w$ appears in $\M^-_{l-1,l}\M^+_{l,l+n}(i,j)$.
Put
$\varphi(w_i) = c_id_i, i=1,2,\dots,n$.
Under specifications appeared in Definition \ref{def:PSSE}, 
we have the following specified equivalences:
\begin{align*}
\M^-_{l-1,l}\M^+_{l,l+n} 
& {\simeq} 
\X_{2l-2} \Y_{2l-1}{\P}_{2l}{\Q}_{2l+1}{\P}_{2l+2}\cdots 
{\Q}_{2l+2n-1}{\P}_{2l+2n}\Q_{2l+2n+1} \\
& {\simeq} 
\X_{2l-2} \P_{2l-1}{\Q}_{2l}{\P}_{2l+1}{\Q}_{2l+2}\cdots 
{\P}_{2l+2n-1}{\Q}_{2l+2n} \Y_{2l+2n+1}, 
\end{align*}
so that the word
$d_1c_2d_2c_3\cdots d_{n-1}c_n$ 
appears in some  component of 
$
{\Q}_{2l}{\P}_{2l+1}{\Q}_{2l+2}\cdots {\P}_{2l+2n-1}.
$
Hence the word
$ 
  \phi^{-1}(d_1c_2)
  \phi^{-1}(d_2c_3)\cdots 
  \phi^{-1}(d_{n-1}c_n)
$
appears in the corresponding component of
$\SN^+_{l,l+1} 
\SN^+_{l+1,l+2} \cdots
\SN^+_{l+n-2,l+n-1}.
$
Thus we see that
$
\varPhi(w)
$
is an  admissible word in 
$\Lambda_\SN$
and that 
 the sliding block code
 $\varPhi_{\infty}$ maps $\Lambda_\M$ to $\Lambda_\SN$.
  Similarly, 
  we can construct a sliding block code
 $\varPsi_{\infty}$ from
 $\Lambda_\SN$ to $\Lambda_\M$
 that is the  inverse of $\varPhi_{\infty}$.
 Thus two subshifts
 $\Lambda_\SN$ and $\Lambda_\M$
 are topologically conjugate.  
\qed

\smallskip

Therefore we conclude the following theorem

\begin{theorem}\label{thm:DM4.11}
If two  symbolic matrix bisystems are properly strong shift equivalent,
their associated subshifts  are topologically conjugate. 
\end{theorem}

By Theorem \ref{thm:DM4.8} and Theorem \ref{thm:DM4.11}, 
we conclude Theorem \ref{thm:DM4.2}.

\smallskip

\begin{remark}
If there exist the matrices
${\P}_l,{\Q}_l$  for all sufficiently large number $l$
in Definition \ref{def:PSSE},
we may show that the presenting subshifts are topologically conjugate
by following the proof of Proposition \ref{prop:DM4.9}.
\end{remark}

\medskip

Properly strong shift equivalence exactly corresponds to 
a finite sequence of bipartite decompositions of symbolic matrix bisystems and $\lambda$-graph bisystems.
The definition of properly strong shift equivalence for symbolic matrix bisystems however
needs rather complicated formulations than that of strong shift equivalence
for nonnegative matrices in \cite{Williams}.
  We will in the rest of this section introduce the notion of  strong shift equivalence 
between two symbolic matrix bisystems 
that is simpler 
and weaker than properly strong shift equivalence.
It is also a generalization of the notion of strong shift equivalence 
between nonnegative matrices defined by Williams in \cite{Williams}, 
between symbolic matrices defined by Nasu in \cite{Nasu},
and between 
symbolic matrix systems defined in \cite{MaDocMath1999}. 
From now on, we will treat general symbolic matrix bisystems.
We do not assume that symbolic matrix bisystems satisfy FPCC.

Let $(\M^-, \M^+), (\SN^-,\SN^+)$
 be two symbolic matrix bisystems over alphabets
$\Sigma_\M^\pm, \Sigma_\SN^\pm,$ respectively.
Let
$m(l), n(l)$
be the sequences for which 
$\M^-_{l,l+1}, \M^+_{l,l+1}$ are $m(l)\times m(l+1)$ symbolic matrices and
$\SN^-_{l,l+1}, \SN^+_{l,l+1}$ are $n(l)\times n(l+1)$ symbolic matrices,
 respectively.

\begin{definition}\label{def:SSE}
Two symbolic matrix bisystems $(\M^-,\M^+)$ and $(\SN^-, \SN^+)$ 
 are said to be {\it  strong shift equivalent in 1-step} 
 if  there exist alphabets 
$C,D$ and
 specifications 
\begin{gather*}
 \varphi_1: \Sigma_\M^-\cdot \Sigma_\M^+ \longrightarrow C\cdot D,
 \qquad
 \varphi_2: \Sigma_\SN^-\cdot \Sigma_\SN^+ \longrightarrow D\cdot C \\
\intertext{and}
\phi_C^{\pm}:  \Sigma_\M^\pm \cdot C\longrightarrow C \cdot \Sigma_\SN^\pm,
\qquad
\phi_D^{\pm}:  \Sigma_\SN^\pm \cdot D\longrightarrow D \cdot \Sigma_\M^\pm,
\qquad (\text{double-sign corresponds})
\end{gather*}
 such that
 for each $l\in \Zp$, 
 there exist
an $ m(l)\times n(l+1)$ matrix $ {\H}_l $ over $ C $
 and
an $ n(l)\times m(l+1)$ matrix $ {\K}_l $ over $ D $ 
 satisfying the following equations: 
$$
\M^-_{l,l+1} \M^+_{l+1,l+2} 
\overset{\varphi_1}{\simeq} {\H}_l{\K}_{l+1},\qquad
\SN^-_{l,l+1} \SN^+_{l+1,l+2} 
\overset{\varphi_2}{\simeq} {\K}_l{\H}_{l+1}
$$
and
\begin{gather*}
\M^+_{l,l+1} \H_{l+1} \overset{\phi_C^+}{\simeq} {\H}_l \SN^+_{l+1,l+2},\qquad
\SN^+_{l,l+1}{\K}_{l+1}\overset{\phi_D^+}{\simeq}{\K}_l \M^+_{l+1,l+2}, \\
 \M^-_{l,l+1} \H_{l+1}\overset{\phi_C^-}{\simeq}{\H}_l \SN^-_{l+1,l+2}, \qquad
\SN^-_{l,l+1}{\K}_{l+1} \overset{\phi_D^-}{\simeq} {\K}_l \M^-_{l+1,l+2}. \\
\end{gather*}
\medskip
We write this situation as
$
(\M^-,\M^+) \underset{1-\st}{\approx}(\SN^-, \SN^+).
$
\end{definition}
Two symbolic matrix bisystems 
 $(\M^-, \M^+)$ and $(\SN^-, \SN^+)$ 
 are said to be {\it  strong shift equivalent in $\ell$-step} 
 if there exist symbolic matrix bisystems
$(\M^-_{(i)}, \M^+_{(i)}), i=1,2,\dots, \ell-1$
such that
\begin{equation*}
(\M^-, \M^+) \underset{1-\st}{\approx} (\M^-_{(1)}, \M^+_{(1)})
      \underset{1-\st}{\approx}
        \cdots 
       \underset{1-\st}{\approx} 
(\M^-_{(\ell-1)}, \M^+_{(\ell-1)}) \underset{1-\st}{\approx} 
( \SN^-, \SN^+).
\end{equation*}
We denote this situation by
$
(\M^-,\M^+) \underset{\ell-\st}{\approx} (\SN^-, \SN^+)
$
and simply call it a {\it  strong shift equivalence}.
\begin{remark}
If $(\M^-, \M^+)$ and $(\SN^-,\SN^+)$ come from
symbolic matrix systems
$(I_\M,\M)$ and $(I_{\SN}, \SN),$ respectively,
then the above definition of strong shift equivalence coincides with
the strong shift equivalence in symbolic matrix systems \cite[p. 304]{MaDocMath1999}.
\end{remark}

Similarly to the case of properly strong shift equivalence,
we see
that  strong shift equivalence on symbolic matrix bisystems is  an equivalence relation.

\begin{proposition}\label{prop:PSSESSE}
For symbolic matrix bisystems satisfying FPCC,
properly strong shift equivalence in 1-step implies
strong shift equivalence in 1-step.
\end{proposition}
\begin{proof}
Let
$\P_l,\Q_l, \X_l$ and $\Y_l$ be the matrices 
in Definition \ref{def:PSSE}
between 
$(\M^-, \M^+) $
and
$(\SN^-, \SN^+).$
We set
$$
\H_l = \X_{2l}\P_{2l+1},\qquad
\K_l = \Y_{2l}\Q_{2l+1}.
$$
It is straightforward to see that they give rise to a strong shift equivalence in 1-step between
$(\M^-, \M^+)$ and $(\SN^-, \SN^+)$.
\end{proof}

\section{\'Etale groupoids for $\lambda$-graph bisystems
and its $C^*$-algebras}
Let $\LGBS$ be a $\lambda$-graph bisystem over alphabet $\Sigma^\pm.$
We will first construct two continuous graphs 
$E_{{\frak L}^-}^+$
and
$E_{{\frak L}^+}^-$
in the sense of Deaconu \cite{Deaconu3} 
(cf. \cite{Deaconu}, \cite{Deaconu2}, \cite{Deaconu4}),
and its shift dynamical systems
$(X_{{\frak L}^-}^+, \sigma_{{\frak L}^-})$
and
$(X_{{\frak L}^+}^-, \sigma_{{\frak L}^+}),$
respectively.
Let ${\frak L}^- = (V^-, E^-, \lambda^-)$ and 
${\frak L}^+ = (V^+, E^+, \lambda^+).$
Let $V_l$ be the common vertex sets  $V^-_l= V^+_l$
and denote it by $\{ v_1^l, \dots, v_{m(l)}^l\}.$

We first define two spaces of label sequences as follows:
\begin{align*}
\Omega_{{\frak L}^-} :=
& \{ ( \dots, u_{3}, \beta_{-3}, u_{2}, \beta_{-2}, u_{1}, \beta_{-1}) 
\in \prod_{l=1}^\infty (V_l \times \Sigma^-) \mid \\
& \quad u_l \in V_l, \,\, \beta_{-l} = \lambda^-(e_l^-) 
\text{ for some } e_l^-\in E_{l,l-1}^-,\, l=1,2,\dots, \\
& \quad \text{ such that }
 t(e_{l+1}^-) =u_l =s(e_{l}^-), l= 1,2,\dots \}. \\ 
\end{align*}
Each element of $\Omega_{{\frak L}^-}$ is a left-infinite sequence written 
\begin{equation*}
\cdots \longrightarrow 
u_3\overset{\beta_{-3}}{\longrightarrow} 
u_2\overset{\beta_{-2}}{\longrightarrow} 
u_1\overset{\beta_{-1}}{\longrightarrow}. 
\end{equation*}
As ${\frak L}^-$ is right-resolving, the edge 
$e_l^- \in E_{l,l-1}^-$
is uniquely determined by its source vertex $u_l \in V_l$ and 
its label $\beta_{-l}.$  
An element 
$( \dots, u_{3}, \beta_{-3}, u_{2}, \beta_{-2}, u_{1}, \beta_{-1})
\in \Omega_{{\frak L}^-}$
is denoted by
$$
\omega = (u_l, \beta_{-l})_{l=1}^\infty \in \Omega_{{\frak L}^-}.
$$
For the pair 
$(u_1,\beta_{-1}),$ 
the edge $e_1^-\in E_{1,0}^-$ satisfying
$\lambda^-(e_1^-) =\beta_{-1}, s(e_1^-) =u_1$
is unique, 
so that the terminal vertex $t(e_1^-) \in V_0$ 
is uniquely determined by $(u_1,\beta_{-1})$, that is denoted by 
$u_0$ or $u_0(\omega).$

The other space $\Omega_{{\frak L}^+}$
 is defined similarly as follows:
\begin{align*}
\Omega_{{\frak L}^+} :=
& \{ ( \alpha_1, u_{1}, \alpha_{2}, u_{2}, \alpha_{3}, u_{3},\dots ) 
\in \prod_{l=1}^\infty (\Sigma^+\times V_l) \mid \\
& \quad u_l \in V_l,\,\, \alpha_{l} = \lambda^+(e_l^+) 
\text{ for some } e_l^+\in E_{l-1,l}^+,\, l=1,2, \dots, \\
& \quad \text{ such that }
 t(e_l^+) = u_l = s(e_{l+1}^+), l= 1,2,\dots \}. \\ 
\end{align*}
Each element of $\Omega_{{\frak L}^+}$ is a right-infinite sequence written 
\begin{equation*}
\overset{\alpha_{1}}{\longrightarrow} u_1
\overset{\alpha_{2}}{\longrightarrow} u_2
\overset{\alpha_{3}}{\longrightarrow} u_3
\cdots.
\end{equation*}
As ${\frak L}^+$ is left-resolving, the edge $e_l^+ \in E_{l-1,l}^+$
is uniquely determined by its terminal vertex $u_l \in V_l$ and 
its label $\alpha_l.$  
An element
$ ( \alpha_1, u_{1}, \alpha_{2}, u_{2}, \alpha_{3}, u_{3},\dots )\in \Omega_{{\frak L}^+} $
is denoted by
$\omega =(\alpha_l, u_l)_{l=1}^\infty \in \Omega_{{\frak L}^+}.$ 
Similarly to $\Omega_{{\frak L}^-},$
the left-resolving property of ${\frak L}^+$
ensures us that 
the edge $e_1^+\in E_{0,1}^+$ satisfying
$\lambda^+(e_1^+) =\alpha_{1}, t(e_1^+) =u_1$
is unique for the pair 
$(\alpha_1, u_1),$  
so that the source vertex $s(e_1^+) \in V_0$ is uniquely determined by $(\alpha_1,u_1)$, 
that is denoted by 
$u_0$ or $u_0(\omega).$

We endow both the spaces 
$\Omega_{{\frak L}^-}$ and $\Omega_{{\frak L}^+}$
with the relative topology of the infinite  product topology
on $\prod_{l=1}^\infty (V_l\times\Sigma^-)$ and $\prod_{l=1}^\infty (\Sigma^+\times V_l)$,
respectively,
so that they are compact Hausdorff spaces. 

We will next define two continuous graphs 
written 
$E_{{\frak L}^-}^+$ and $E_{{\frak L}^+}^-$
from  
$\LGBS$ in the following way:
\begin{align*}
E_{{\frak L}^-}^+ :=
& \{ (\omega, \alpha, \omega' )
\in \Omega_{{\frak L}^-}\times \Sigma^+\times \Omega_{{\frak L}^-} \mid 
 \omega = (u_l, \beta_{-l})_{l=1}^\infty, \,
   \omega' = (u'_l, \beta_{-l+1})_{l=1}^\infty, \\
& \quad \alpha = \lambda^+(e_{l,l+1}^+) 
\text{ for some } e_{l,l+1}^+\in E_{l,l+1}^+ \, \\
& \quad \text{ such that }
 u_l = s(e_{l,l+1}^+), \,  u'_{l+1} =t(e_{l,l+1}^+), l=0,1,2,\dots \}. \\ 
\end{align*}
Each element of $E_{{\frak L}^-}^+$ is figured such as 
\begin{equation*}
\begin{CD}
\cdots @>{\beta_{-4}}>>u_3 @>{\beta_{-3}}>>u_2 @>{\beta_{-2}}>>u_1
@>{\beta_{-1}}>>u_0\\
@. @V{e_{3,4}^+}V{\alpha}V @V{e_{2,3}^+}V{\alpha}V @V{e_{1,2}^+}V{\alpha}V 
@V{e_{0,1}^+}V{\alpha}V @.\\
\cdots @>{\beta_{-4}}>>u'_4 @>{\beta_{-3}}>>u'_3 @>{\beta_{-2}}>>u'_2 
@>{\beta_{-1}}>> u'_1 @>{\beta_0}>>u'_0.
\end{CD}
\end{equation*}
Similarly
\begin{align*}
E_{{\frak L}^+}^-:=
& \{ (\omega, \beta, \omega') \in \Omega_{{\frak L}^+}\times \Sigma^-\times \Omega_{{\frak L}^+} \mid 
 \omega = (\alpha_l, u_l)_{l=1}^\infty, \,
   \omega' = (\alpha_{l-1}, u'_l)_{l=1}^\infty, \\
& \quad \beta = \lambda^-(e_{l+1,l}^-) 
\text{ for some } e_{l+1,l}^-\in E_{l+1,l}^- \\
& \quad \text{ such that }
 u_l = t(e_{l+1,l}^-), \,  u'_{l+1} =s(e_{l+1,l}^-), l= 0, 1,2,\dots \}. \\ 
\end{align*}
Each element of $E_{{\frak L}^+}^-$ is figured such as
\begin{equation*}
\begin{CD}
@. u_0 @>{\alpha_1}>> u_1 @>{\alpha_2}>> u_2 @>{\alpha_3}>> u_3
@>{\alpha_4}>> \cdots \\
@. @A{e_{1,0}^-}A{\beta}A @A{e_{2,1}^-}A{\beta}A @A{e_{3,2}^-}A{\beta}A 
@A{e_{4,3}^-}A{\beta}A @.\\
u'_0 @>{\alpha_0}>> u'_1 @>{\alpha_1}>> u'_2 @>{\alpha_2}>> u'_3
@>{\alpha_3}>> u'_4  @>{\alpha_4}>>\cdots. \\
\end{CD}
\end{equation*}
Following Deaconu \cite{Deaconu3}
 (cf. \cite{Deaconu}, \cite{Deaconu2}, \cite{Deaconu4}), 
we construct a shift dynamical system:
\begin{align*}
X_{{\frak L}^-}^+ :=
& \{ (\alpha_i, \omega^i )_{i=1}^\infty \in \prod_{i=1}^\infty (\Sigma^+\times \Omega_{{\frak L}^-})
\mid 
\omega^i = (u_l^{i}, \beta_{-l+i+1})_{l=1}^\infty \in \Omega_{{\frak L}^-}, i=1,2,\dots, \\
& (\omega^i, \alpha_{i+1}, \omega^{i+1}) \in E_{{\frak L}^-}^+, i=1,2,\dots, \,
   (\omega^0, \alpha_1, \omega^1) \in E_{{\frak L}^-}^+ \text{ for some } 
\omega^0 \in \Omega_{{\frak L}^-}
  \}
\end{align*}
and the shift map 
$\sigma_{{\frak L}^-}: 
X_{{\frak L}^-}^+ \longrightarrow X_{{\frak L}^-}^+$ by setting
\begin{equation*}
\sigma_{{\frak L}^-}((\alpha_i, \omega^i )_{i=1}^\infty ) 
= (\alpha_{i+1}, \omega^{i+1} )_{i=1}^\infty,
\qquad (\alpha_i, \omega^i )_{i=1}^\infty  \in X_{{\frak L}^-}^+.
\end{equation*}
The set $X_{{\frak L}^-}^+$ is endowed with
 the relative topology of the infinite product topology of
$\Sigma^+ \times \Omega_{{\frak L}^-}.$
It is a zero-dimensional compact Hausdorff space.
The shift map
$\sigma_{{\frak L}^-}: X_{{\frak L}^-}^+ \longrightarrow X_{{\frak L}^-}^+$
is continuous and a local homeomorphism.
As ${\frak L}^+$ is left-resolving, 
for any element
$x =(\alpha_i, \omega^i )_{i=1}^\infty \in X_{{\frak L}^-}^+,$ 
there uniquely exists 
$\omega^0 \in \Omega_{{\frak L}^-}$ such that 
$ (\omega^0, \alpha_1, \omega^1) \in E_{{\frak L}^-}^+$
by the local property of $\lambda$-graph bisystem.
We denote $\omega^0$ 
 by $\omega^0(x)$, which is uniquely determined by $x \in X_{{\frak L}^-}^+.$  
Therefore an element 
$x =(\alpha_i, \omega^i )_{i=1}^\infty \in X_{{\frak L}^-}^+$
defines a two-dimensional diagram as follows:
\begin{equation*}
\begin{CD}
\omega^0\cdots @>{\beta_{-2}}>>u_2^{0} @>{\beta_{-1}}>>u_1^{0} 
@>{\beta_0}>>u_0^{0}\\
@. @VV{\alpha_{1}}V @VV{\alpha_{1}}V @VV{\alpha_{1}}V \\
\omega^1\cdots @>{\beta_{-2}}>>u_3^{1} @>{\beta_{-1}}>>u_2^{1} 
@>{\beta_0}>>u_1^{1} @>{\beta_1}>>u_0^{1}\\
@. @VV{\alpha_{2}}V @VV{\alpha_{2}}V @VV{\alpha_{2}}V @VV{\alpha_{2}}V @.\\
\omega^2\cdots @>{\beta_{-2}}>>v_4^{2} @>{\beta_{-1}}>>v_3^{2} 
@>{\beta_0}>>v_2^{2} @>{\beta_1}>>v_1^{2} @>{\beta_2}>>v_0^{2}\\
@. @VV{\alpha_{3}}V @VV{\alpha_{3}}V @VV{\alpha_{3}}V @VV{\alpha_{3}}V 
@VV{\alpha_{3}}V @.\\
\omega^3\cdots @>{\beta_{-2}}>>v_5^{3} @>{\beta_{-1}}>>v_4^{3} 
@>{\beta_0}>>v_3^{3} @>{\beta_1}>>v_2^{3} @>{\beta_2}>>v_1^{3}
@>{\beta_3}>>v_0^{3}\\
@. @VV{\alpha_{4}}V @VV{\alpha_{4}}V @VV{\alpha_{4}}V @VV{\alpha_{4}}V @VV{\alpha_{4}}V @VV{\alpha_{4}}V @.\\
\end{CD}
\end{equation*}
We may similarly construct a shift dynamical system
$( X_{{\frak L}^+}^-,\sigma_{{\frak L}^+})$ from the other continuous graph
$E_{{\frak L}^+}^-.$ 

By the now standard 
Deaconu--Renault groupoid construction 
(see. \cite{Deaconu3}, \cite{Deaconu4}, \cite{Renault2}, \cite{Renault3},  cf.
\cite{Sims}), we have an amenable and \'etale groupoid written
$\G_{{\frak L}^-}^{+}$
from the shift dynamical system
$( X_{{\frak L}^-}^+,\sigma_{{\frak L}^-})$.
 The space of the groupoid is defined by 
\begin{equation*}
\G_{{\frak L}^-}^{+} :=
\{ (x,n,y) \in  X_{{\frak L}^-}^+\times \Z \times X_{{\frak L}^-}^+ \mid 
\exists k,l \in \Zp ; n = k-l, \sigma_{{\frak L}^-}^k(x) =\sigma_{{\frak L}^-}^l(y)\}.  
\end{equation*}
The  unit space $(\G_{{\frak L}^-}^{+})^{(0)}$ is defined by
\begin{equation*}
(\G_{{\frak L}^-}^{+})^{(0)} :=
\{ (x, 0, x) \in \G_{{\frak L}^-}^{+} \mid x \in  X_{{\frak L}^-}^+\}.  
\end{equation*}
It is identified with the space
$ X_{{\frak L}^-}^+$ as a topological space.
The range map and the source map
are defined by
$
r(x,n,y) = x, \,s(x,n,y) = y.
$
The product and the inverse operations
are defined by 
$$
(x,n,y)(y,m,z) = (x,n+m,z), \qquad 
(x,n,y)^{-1} = (y,-n,x).
$$

We may similarly construct the other amenable and \'etale
groupoid $ \G_{{\frak L}^+}^{-}$ from the shift dynamical system $(X_{{\frak L}^+}^{-}, \sigma_{{\frak L}^+}).$

 Now we will define our $C^*$-algebra $\OALMP$  in the following way.
\begin{definition}
The $C^*$-algebra $\OALMP$ associated with  a $\lambda$-graph bisystem 
$\LGBS$ is defined to be the  $C^*$-algebra 
$C^*(\G_{{\frak L}^-}^+)$
of the groupoid
$\G_{{\frak L}^-}^+.$
Similarly we define 
the  $C^*$-algebra $\OALPM$ from the other groupoid 
$\G_{{\frak L}^+}^-.$
\end{definition}
For general theory of $C^*$-algebras of \'etale groupoids, see 
(\cite{Renault}, \cite{Renault2}, \cite{Renault3},  
cf. \cite{Deaconu3}, \cite{Deaconu4}, \cite{Sims}, etc.).
Let $C_c(\G_{{\frak L}^-}^+) $ be the set of complex-valued  continuous functions on
$\G_{{\frak L}^-}^+$ with compact support. 
It has a natural product structure and $*$-involution of 
given by
\begin{align*}
(f*g)(s) &  = 
 \sum_{\underset{r(t) = r(s)}{t;}} f(t) g(t^{-1}s) 
           = 
 \sum_{\underset{s = t_1 t_2}{t_1,t_2;}} f(t_1) g(t_2), \\
  f^*(s) & = \overline{f(s^{-1})}, 
  \qquad f,g \in C_c(\G_{{\frak L}^-}^+), \quad s \in \G_{{\frak L}^-}^+.     
\end{align*}
The algebra 
$C_c(\G_{{\frak L}^-}^+)$ is a dense $*$-subalgebra of  $C^*(\G_{{\frak L}^-}^+)$.

We will study the algebraic structure of the $C^*$-algebra $\OALMP$.

Recall that $F(v_i^l)$  denotes the follower set
of $v_i^l \in V_l$ in ${\frak L}^-,$
that is defined after Definition \ref{def:lambdabisystem}.
We define the cylinder set 
$U_{\Omega_{{\frak L}^-}}(v_i^l;{\xi)} \subset \Omega_{{\frak L}^-}$
for $\xi =(\xi_1, \xi_2,\dots,\xi_l) \in F(v_i^l)$
by setting
\begin{align*}
U_{\Omega_{{\frak L}^-}}(v_i^l;{\xi)}
 :=&
 \{ ( \dots, u_{3},\beta_{-3}, u_{2}, \beta_{-2}, u_{1}, \beta_{-1}) 
\in \Omega_{{\frak L}^-} \mid \\
& \qquad u_l = v_i^l, \beta_{-l} = \xi_1,\beta_{-l+1} = \xi_2,
\dots
\beta_{-1} = \xi_l
 \}.
\end{align*}
Each element of
$U_{\Omega_{{\frak L}^-}}(v_i^l;{\xi)}$
is figured such as
\begin{equation*}
\cdots \longrightarrow v_i^l 
\overset{\xi_1}{\longrightarrow}\bigcirc
\overset{\xi_2}{\longrightarrow}\bigcirc
\longrightarrow \cdots
\overset{\xi_l}{\longrightarrow}\bigcirc.
\end{equation*}
We note that the vertices $u_{l-1}, u_{l-2}, \dots,u_{0}$
 of the terminals of the labeled edges 
$\xi_1, \xi_2,\dots,\xi_l$ 
are  automatically determined by $v_i^l$ and $\xi$, because 
${\frak L}^-$ is right-resolving.
The set of cylinder sets $U_{\Omega_{{\frak L}^-}}(v_i^l;{\xi)}$ 
form a basis of open sets 
of $\Omega_{{\frak L}^-}.$
Let us define a clopen set of $X_{{\frak L}^-}^+$ by setting
\begin{equation}
U_{X^+_{{\frak L}^-}}(v_i^l;{\xi}) =
\{ x \in X^+_{{\frak L}^-} \mid \omega^0(x) \in U_{\Omega_{{\frak L}^-}}(v_i^l;{\xi)} \}.
\label{eq:uxvilxi}
\end{equation} 
Recall that $B_m(\Lambda_{{\frak L}^+})$ 
denotes the set of admissible words of 
the subshift $\Lambda_{{\frak L}^+}$
with length $m.$
For $v_i^l\in V_l$,
$
\xi =(\xi_1, \xi_2,\dots,\xi_l) \in F(v_i^l)
$ and
$
\mu = (\mu_1,\dots,\mu_m) \in B_m(\Lambda_{{\frak L}^+})$ with $m\le l$,
define the cylinder set
$U_{X_{{\frak L}^-}^+}(\mu,v_i^l;{\xi})
$ of $X_{{\frak L}^-}^+$ by setting 
\begin{equation}
U_{X_{{\frak L}^-} ^+}(\mu,v_i^l;{\xi})
=  \{ (\alpha_i, \omega^i)_{i=1}^\infty \in X_{{\frak L}^-}^+ \mid 
\alpha_1 = \mu_1, \alpha_2 = \mu_2,\dots, \alpha_m = \mu_m, \,
 \omega^m \in  U_{\Omega_{{\frak L}^-}}(v_i^l;{\xi)} \}.
\label{eq:muvil}
\end{equation}
Each element of 
$U_{X_{{\frak L}^-}^+}(\mu,v_i^l;{\xi})$ is figured such as 
\begin{equation*}
\begin{CD}
\omega^0\cdots @>>>u_{l-m}^{0} @>{\xi_1}>>\cdots 
@>{\xi_{l-m}}>>u_{0}^{0} \\ 
@. @VV{\mu_1}V  @VV{\mu_1}V   @VV{\mu_1}V \\
\omega^1\cdots @>>>u_{l-m+1}^{1} @>{\xi_1}>>\cdots 
@>{\xi_{l-m}}>>u_{1}^{1} @>{\xi_{l-m+1}}>>u_{0}^{1} \\
@. @VV{\mu_2}V  @VV{\mu_2}V  @VV{\mu_2}V @VV{\mu_2}V \\
\cdots @>>>\vdots @>{\xi_1}>>\cdots 
@>>>\vdots @>{\xi_{l-m+1}}>>\vdots @>{\xi_{l-1}}>> \\
@. @VV{\mu_m}V @VV{\mu_m}V @VV{\mu_m}V @VV{\mu_m}V  @VV{\mu_m}V \\
\omega^m\cdots @>>>u_l^m =v_i^{l} @>{\xi_1}>>\cdots 
@>>>\vdots @>>>u_{2}^m @>{\xi_{l-1}}>>u_{1}^{m} @>{\xi_l}>>u_{0}^{m}\\
@. @VVV  @VVV @VVV  @VVV @VVV   @VVV\\
\end{CD}
\end{equation*}
For 
$x = (\alpha_i, \omega^i)_{i=1}^\infty \in X_{{\frak L}^-}^+,$
we put
$\lambda_i(x) = \alpha_i \in \Sigma^+,$
$\omega^i(x) =\omega^i \in \Omega_{{\frak L}^-}$
for $i \in \N,$ 
respectively, so that 
$x = (\lambda_i(x), \omega^i(x))_{i=1}^\infty.$
Now ${\frak L}^+$ is left-resolving
so that there uniquely exists 
$\omega^0(x) \in \Omega_{{\frak L}^-}$ satisfying
$(\omega^0(x),\alpha_1,\omega^1) \in E_{{\frak L}^-}^+.$
Define 
$U(\mu)\subset \G_{{\frak L}^-}^+$ 
for 
$\mu = (\mu_1,\dots,\mu_k) \in B_k(\Lambda_{{\frak L}^+}),$
and
$U(v_i^l;\xi)\subset \G_{{\frak L}^-}^+$
for $v_i^l \in V_l, \xi\in F(v_i^l)$ by
\begin{align*}
 U(\mu) 
 = & \{ (x,k,z) \in \G_{{\frak L}^-}^+ \mid
 \sigma_{{\frak L}^-}^k(x) = z,   
\lambda_1(x) =\mu_1,\dots,\lambda_k(x) =\mu_k \}, \quad \text{ and}\\
 U({v}_i^l;\xi) 
 = & \{ (x,0,x) \in \G_{{\frak L}^-}^+ \ \mid
 \ \omega^0(x) \in U_{\Omega_{{\frak L}^-}}({v}_i^l; \xi) \}
(=U_{X_{{\frak L}^-}^+}(v_i^l;{\xi})).
\end{align*}
They are clopen sets of $\G_{{\frak L}^-}^+$.
We set
$$
S_{\mu} = \chi_{U(\mu)}, \qquad 
E_i^{l-}(\xi) = \chi_{U({v}_i^l;\xi)  }
\qquad \text{ in }\quad C_c(\G_{{\frak L}^-}^+ )
$$
where 
$
\chi_{F}
\in C_c(\G_{{\frak L}^-}^+ )
$ denotes 
the characteristic function of a clopen set
$F$
on the groupoid
$
\G_{{\frak L}^-}^+.
$
We in particular write $S_\mu$ as $S_\alpha$
for the symbol $\mu = \alpha \in \Sigma^+$.
For $\mu \not\in B_*(\Lambda_{{\frak L}^+}),$
$\xi\not\in F(v_i^l),$
we recognize that 
$S_\mu =0, E_i^{l-}(\xi) =0.$
The following lemma is straightforward.
\begin{lemma}
\hspace{4cm}
\begin{enumerate}
\renewcommand{\theenumi}{\roman{enumi}}
\renewcommand{\labelenumi}{\textup{(\theenumi)}}
\item
For $\mu = (\mu_1,\dots,\mu_m) \in B_m(\Lambda_{{\frak L}^+})$, we have
$S_\mu = S_{\mu_1}\cdots S_{\mu_m}$.
\item
For $\xi \in F(v_i^l)$, the operator
$E_i^{l-}(\xi)$ is a projection such that the family
$E_i^{l-}(\xi), \xi \in F(v_i^l), i=1,\dots,m(l), l\in \Zp$
are mutually commuting projections.
\end{enumerate}
\end{lemma}
The transition matrices $A^-_{l,l+1}, A^+_{l,l+1}$
for  ${\frak L}^-, {\frak L}^+$ respectively
are defined by setting
\begin{align*}
A_{l,l+1}^-(i,\beta,j)
 & =
{\begin{cases}
1 &  
    \text{ if } \ t(e^-) = v_i^{l}, \lambda^-(e^-) = \beta,
                       s(e^-) = v_j^{l+1} 
    \text{ for some }    e^- \in E_{l+1,l}^-, \\
0           & \text{ otherwise,}
\end{cases}} \\
A_{l,l+1}^+(i,\alpha,j)
 & =
{\begin{cases}
1 &  
    \text{ if } \ s(e^+) = v_i^{l}, \lambda^+(e^+) = \alpha,
                       t(e^+) = v_j^{l+1} 
    \text{ for some }    e^+ \in E_{l,l+1}^+, \\
0           & \text{ otherwise}
\end{cases}} 
\end{align*}
for
$
i=1,2,\dots,m(l),\ j=1,2,\dots,m(l+1),  \, \beta\in \Sigma^-, \, \alpha \in \Sigma^+.$ 
Let us denote by 
$\ALM$  the $C^*$-subalgebra of $\OALMP$  generated by
$E_i^{l-}(\xi), v_i^l \in V_l, \xi\in F(v_i^l).$
We define the other $C^*$-subalgebra 
$\ALP$ of $\OALPM$
in a similar way.
Let us denote by $C(\Omega_{{\frak L}^-})$
the commutative $C^*$-algebra of complex valued continuous functions on 
$\Omega_{{\frak L}^-}$.
For a subset $B \subset \A$ of a $C^*$-algebra $\A,$
we denote by $C^*(B)$ the $C^*$-subalgebra of $\A$ generated by $B.$ 
\begin{lemma}\label{lem:abelian}
\hspace{4cm}
\begin{enumerate}
\renewcommand{\theenumi}{\roman{enumi}}
\renewcommand{\labelenumi}{\textup{(\theenumi)}}
\item
Each operator $E_i^{l-}(\xi)$ indexed by vertex $v_i^l \in  V_l$ and 
admissible word $\xi =(\xi_1,\dots, \xi_l) \in F(v_i^l)$
is a projection satisfying the following operator relations:
\begin{gather}
 \sum_{\xi \in F(v_i^l)} \sum_{i=1}^{m(l)} 
 E_i^{l-}(\xi)    =  1,  \label{eq:L11}\\ 
 E_i^{l-} (\xi)   = \sum_{\beta\in \Sigma^-}  \sum_{j=1}^{m(l+1)}
A_{l,l+1}^-(i,\beta,j)E_j^{l+1-}(\beta\xi),  \label{eq:L12}
\end{gather}
where the word
$\beta\xi $ in \eqref{eq:L12} is defined by 
$\beta \xi = (\beta, \xi_1,\dots,\xi_l)$
for $\beta \in \Sigma^-,
\xi = (\xi_1,\dots,\xi_l) \in F(v_i^l)$,
and
$E_j^{l+1}(\beta\xi)=0$ unless $\beta \xi \in F(v_j^{l+1}).$
\item
The correspondence 
\begin{equation}
\varphi_l: E_i^{l-} (\xi) \in \ALM \longrightarrow 
\chi_{U_{\Omega_{{\frak L}^-}}(v_i^l;{\xi)}} \in C(\Omega_{{\frak L}^-}) \label{eq:ALMC}
\end{equation}
 gives rise to an isomorphism of $C^*$-algebras between 
$\ALM$ and  $C(\Omega_{{\frak L}^-}).$
\end{enumerate}
\end{lemma}
\begin{proof}
(i) Any element $x \in X_{{\frak L}^-}^+$ defines an element
$\omega^0(x) \in \Omega_{{\frak L}^-}.$
As the set $X_{{\frak L}^-}^+$ 
 is a disjoint union:
$$
X_{{\frak L}^-}^+
= \bigcup_{i=1}^{m(l)}\bigcup_{\xi \in F(v_i^l)}  U_{X_{{\frak L}^-}^+}(v_i^l;\xi)
$$
for a fixed $l \in \N$,
we know the equality \eqref{eq:L11}.
The disjoint union
\begin{equation}
U_{\Omega_{{\frak L}^-}}(v_i^l;\xi)
   = \bigcup_{\beta\in \Sigma^-}  
\bigcup^{m(l+1)}_{
\underset{ A_{l,l+1}^-(i,\beta,j)=1}{j=1; }}
U_{\Omega_{{\frak L}^-}}(v_j^{l+1};\beta\xi)  
\label{eq:L123}
\end{equation} 
yields the identity
\begin{equation}
\chi_{U(v_i^l;{\xi)}}
   = \sum_{\beta\in \Sigma^-}  \sum_{j=1}^{m(l+1)}
A_{l,l+1}^-(i,\beta,j) \chi_{U(v_j^{l+1};\beta\xi)}
\label{eq:L121}
\end{equation} 
that leads to the equality \eqref{eq:L12}.

(ii)
Define the $C^*$-subalgebras for $l \in \N$
\begin{align*}
\mathcal{A}_{{\frak L}^-,l} 
:=&  C^*(E_i^{l-}(\xi) \mid \xi \in F(v_i^l), i=1,2,\dots,m(l)) \subset \ALM,\\
C(\Omega_{{\frak L}^-,l}) 
:=& C^*(\chi_{U_{\Omega_{{\frak L}^-}}(v_i^l;{\xi)}}  \mid \xi \in F(v_i^l), i=1,2,\dots,m(l))
\subset C(\Omega_{{\frak L}^-}).
\end{align*}
Then the correspondence \eqref{eq:ALMC} gives rise to an isomorphism
$\varphi_l: \mathcal{A}_{{\frak L}^-,l}\longrightarrow C(\Omega_{{\frak L}^-,l}) $
of the commutative finite dimensional $C^*$-algebras.
By the identity \eqref{eq:L12} together  with \eqref{eq:L121},
we have embeddings
\begin{equation*}
\mathcal{A}_{{\frak L}^-,l} 
\hookrightarrow \mathcal{A}_{{\frak L}^-,l+1},
\qquad
C(\Omega_{{\frak L}^-,l}) 
\hookrightarrow 
C(\Omega_{{\frak L}^-,l+1}),
\end{equation*}
that are compatible to the isomorphisms
$\varphi_l: \mathcal{A}_{{\frak L}^-,l}\longrightarrow C(\Omega_{{\frak L}^-,l}), l \in \N.$
As the algebras $\ALM,  
C(\Omega_{{\frak L}^-l})$
are inductive limits 
$\ALM = {\displaystyle{\lim_{l\to \infty}}}\mathcal{A}_{{\frak L}^-,l},
C(\Omega_{{\frak L}^-l})
=
{\displaystyle{\lim_{l\to\infty}}}
C(\Omega_{{\frak L}^-,l}) 
$
respectively,
we conclude that the $C^*$-algebras
$\ALM $ and 
$C(\Omega_{{\frak L}^-l})$
are isomorphic.
\end{proof}
The following lemma is a key to proving the identity \eqref{eq:saeas}.
\begin{lemma}\label{lem:keysaesa}
For $\xi =(\xi_1,\xi_2,\dots,\xi_l) \in F(v_i^l), \alpha \in \Sigma^+$
and $x \in X^+_{{\frak L}^-}$, the following two conditions are equivalent:
\begin{enumerate}
\renewcommand{\theenumi}{\roman{enumi}}
\renewcommand{\labelenumi}{\textup{(\theenumi)}}
\item
There exists $z \in X^+_{{\frak L}^-}$ such that 
\begin{equation}
\omega^0(z) \in U_{\Omega_{{\frak L}^-}}(v_i^l;\xi), \qquad 
\sigma_{{\frak L}^-}(z) = x, \qquad
\lambda_1(z ) = \alpha. \label{eq:key1}
\end{equation}
\item There exist $\beta \in \Sigma^-$ and $j=1,2,\dots, m(l+1)$ such that 
\begin{equation}
\xi \beta \in F(v_j^{l+1}), \qquad
A^+_{l,l+1}(i,\alpha,j) =1,\qquad
x\in U_{X^+_{{\frak L}^-}}(v_j^{l+1};\xi\beta). \label{eq:key2}
\end{equation}
\end{enumerate}
Such $z$ and $\beta$ bijectively correspond to each other.
  \end{lemma}
\begin{proof}
(i) $\Longrightarrow$ (ii):
Suppose that  $z \in X^+_{{\frak L}^-}$
satisfies the conditions \eqref{eq:key1}.
Since we see $(\omega^0(z), \alpha,\omega^1(z)) \in E^+_{{\frak L}^-},$
we have $\omega^1(z) = \omega^0(x).$
Let $\omega^0(x) = (u_l, \beta_{-l})_{l=1}^\infty \in \Omega_{{\frak L}^-},$
so that 
$\beta_{-1} =\xi_l, \, \beta_{-2} =\xi_{l-1}, \dots, \beta_{-l} =\xi_1,$
such as the following figure:
\begin{equation*}
\begin{CD}
\cdots @>>>v_i^l@>\xi_1>>\bigcirc @>\xi_2 >>
\cdots 
@>{\xi_l}>>\bigcirc \\
@. @VV{\alpha}V @VV{\alpha}V 
@. @VV{\alpha}V @.\\
\cdots @>>>u_{l+1} @>{\beta_{-l}}>> u_l @>{\beta_{-l+1}}>> 
\cdots 
@>{\beta_{-1}}>>u_1 @>{\beta_0}>>u_0.
\end{CD}
\end{equation*}
Put
$\beta = \beta_0$ and
$v_j^{l+1} = u_{l+1}.$
Then the condition \eqref{eq:key2} holds.

(ii) $\Longrightarrow$ (i):
Conversely let $\beta \in \Sigma^-$ and $j=1,2,\dots,m(l+1)$ 
satisfy the condition \eqref{eq:key2}.
Let $\omega^0(x) = (u_l, \beta_{-l})_{l=1}^\infty \in \Omega_{{\frak L}^-},$
so that 
$\beta_{-1} =\xi_l, \, \beta_{-2} =\xi_{l-1}, \dots, \beta_{-l} =\xi_1.$
By the hypothesis 
 $\xi =(\xi_1,\xi_2,\dots,\xi_l) \in F(v_i^l)$, there exist unique labeled edges
$f_{n,n-1}^- \in E_{n,n-1}^-$ for $n= 1,2,\dots,l$ such that 
$s(f_{l,l-1}^-) = v_i^l, \, t(f_{n,n-1}^-) = s(f_{n-1, n-2}^-), \,
 \lambda^-(f_{n,n-1}^-) = \xi_{l-n+1}$
for $n= 1,2,\dots,l$.
We put
$u'_n = s(f_{n,n-1}^-) $ for 
 $n= 1,2,\dots,l$, so that $u'_l = v_i^l$.
By the hypothesis 
$A^+_{l,l+1}(i,\alpha,j) =1$ with the local property of $\lambda$-graph bisystem,
we may find  labeled edges 
$e_{n,n+1}^+ \in E_{n,n+1}^+$ for $n=0,1,\dots,l$ such that 
$s(e_{n,n+1}^+) = u'_n, \,  
t(e_{n,n+1}^+) = u_n, \, 
\lambda^+(e_{n,n+1}^+) = \alpha$.
By the condition 
$ (u_l, \beta_{-l})_{l=1}^\infty \in \Omega_{{\frak L}^-},$
there exists an labeled edge
$e_{l+2, l+1}^- \in  E_{l+2, l+1}^- $
such that 
$s(e_{l+2, l+1}^-) = u_{l+2}, \, 
t(e_{l+2, l+1}^-) = u_{l+1} (= v_j^{l+1}), \, 
\lambda^-(e_{l+2, l+1}^-) = \beta_{-l-1}.
$
By applying the local property of $\lambda$-graph bisystem for the pair
$e_{l,l+1}^+$ and $e_{l+2, l+1}^-$ 
satisfying
$t(e_{l,l+1}^+) = t(e_{l+2, l+1}^-),$ 
we may find labeled edges 
$f_{l+1,l}^- \in E_{l+1,l}^-$
and
$e_{l+1, l+2}^+ \in E_{l+1, l+2}^+$ 
 such that 
\begin{gather*}
t(f_{l+1,l}^-) = u'_l (=v_i^l), \quad
s(f_{l+1,l}^-) = s(e_{l+1, l+2}^+), \quad
t(e_{l+1, l+2}^+) = u_{l+2}, \\
\lambda^-(f_{l+1,l}^-) =\beta_{-l-1}, \quad
\lambda^+(e_{l+1, l+2}^+) =\alpha.
\end{gather*}
We put
$u'_{l+1} =s(f_{l+1,l}^-) \in V_{l+1}$.
Like this way, 
by successively applying the local property of $\lambda$-graph bisystem,
we may find 
$\omega' =({u'}_l, \beta_{-l-1})_{l=1}^\infty \in \Omega_{{\frak L}^-}(v_i^l;\xi)$
such that 
$\beta_{-1} = \xi_l, \, \dots, \beta_{-l} = \xi_1$ 
and
$(\omega',  \alpha, \omega^0(x)) \in E_{{\frak L}^-}^+.$
By defining 
$
z = (\alpha_i, \omega^i)_{i=1}^\infty 
\in \prod_{i=1}^\infty (\Sigma^+\times \Omega_{{\frak L}^-})
$
such that 
$\alpha_1 = \alpha, \omega^1 = \omega^0(x)$ and
$(\alpha_i, \omega^i)_{i=2}^\infty =x$, 
we have  
$z \in X^+_{{\frak L}^-}$, 
$\sigma_{{\frak L}^-}(z) = x$,  $\lambda_1(z) =\alpha$
and $\omega^0(z) = \omega' \in \Omega_{{\frak L}^-}(v_i^l;\xi).$
\end{proof}
\begin{lemma} For $\alpha \in \Sigma^+$ and $\xi\in F(v_i^l),$
we have 
\begin{equation}
S_{\alpha}^*E_i^{l-}(\xi) S_{\alpha}   =  
\sum_{\beta\in \Sigma^-}\sum_{j=1}^{m(l+1)} A_{l,l+1}^+(i,\alpha,j)E_j^{l+1-}(\xi\beta),
\label{eq:saeas}
\end{equation}
where $E_j^{l+1-}(\xi \beta)=0$ unless
$\xi \beta \in F(v_j^{l+1}).$
 \end{lemma}
\begin{proof}
It suffices to show that the equality
\begin{equation}
\chi_{U(\alpha)}^* *\chi_{U(v_i^l;\xi)}*\chi_{U(\alpha)}
=\sum_{\beta\in \Sigma^-}
  \sum_{j=1}^{m(l+1)} A_{l,l+1}^+(i,\alpha,j) \chi_{U(v_j^{l+1};\xi\beta)}
\label{eq:chiUalpha}
\end{equation}
holds. 
We then have for $s = (x,n,z) \in \G_{{\frak L}^-}^+,$
\begin{align*}
 &[\chi_{U(\alpha)}^* *\chi_{U(v_i^l;\xi)}*\chi_{U(\alpha)}](x,n,z) 
 \\
=& \sum_{
\underset{r(t) = r(s)}{t; }
} \overline{\chi_{U(\alpha)}(t^{-1})}
(\chi_{U(v_i^l;\xi)}*\chi_{U(\alpha)})(t^{-1}s)  
\qquad (\text{Put } t=(y,m,w)\in \G_{{\frak L}^-}^+)
\\
=& \sum_{m, w} \chi_{U(\alpha)}(w,-m,x)
(\chi_{U(v_i^l;\xi)}*\chi_{U(\alpha)})(w, n-m,z)  
\\
=& \sum_{
\underset{x = \sigma_{{\frak L}^-}(w), \lambda_1(w) = \alpha}{w;} } 
(\chi_{U(v_i^l;\xi)}*\chi_{U(\alpha)})(w, n+1,z),
\end{align*}
because  
$\chi_{U(\alpha)}(w,-m,x) =1$
 if and only if  
$m=-1, \, x= \sigma_{{\frak L}^-}(w),\, \lambda_1(w) =\alpha$.
Now 
\begin{align*}
& 
(\chi_{U(v_i^l;\xi)}*\chi_{U(\alpha)})(w, n+1,z)   \\
=& \sum_{
\underset{r(t') = w}{t';}
} 
\chi_{U(v_i^l;\xi)}(t')\chi_{U(\alpha)} ({t'}^{-1}\cdot (w,n+1, z))  
\qquad (\text{Put } (y',m',w') =t'\in \G_{{\frak L}^-}^+) \\
=& \sum_{
\underset{y' = w}{m', y';}
} 
\chi_{U(v_i^l;\xi)}(y', m', w')\chi_{U(\alpha)} (w',-m'+n+1, z)  \\
=&\sum_{ 
\underset{z=\sigma_{{\frak L}^-}(w'), \lambda_1(w') = \alpha}
{w';}} 
\chi_{U(v_i^l;\xi)}(w, n, w'),
\end{align*}
because  $\chi_{U(\alpha)}(w',-m'+n+1,z) =1$ 
if and only if
$m'=n, \,  z= \sigma_{{\frak L}^-}(w'), \, \lambda_1(w') = \alpha.$
Since
$\chi_{U(v_i^l;\xi)}(w, n, w')=1$ 
if and only if
 $w=w', n=0$ and $\omega^0(w) \in U_{\Omega_{{\frak L}^-}^+}(v_i^l;\xi),$
we have
\begin{equation*}
 [\chi_{U(\alpha)}^* *\chi_{U(v_i^l;\xi)}*\chi_{U(\alpha)}](x,n,z) 
= \sum_{
\underset{x = z=\sigma_{{\frak L}^-}(w), \lambda_1(w) = \alpha}{w; }} 
\chi_{U_{\Omega_{{\frak L}^-}^+}(v_i^l;\xi)}(\omega^0(w)).
\end{equation*}
By Lemma \ref{lem:keysaesa} with the hypothesis $\xi \in F(v_i^l),$
we have
$\omega^0(w)\in U_{\Omega_{{\frak L}^-}^+}(v_i^l;\xi)$ 
with $x=z = \sigma_{{\frak L}^-}(w), \lambda_1(w) =\alpha$
if and only if 
$x= z \in U_{X_{{\frak L}^-}^+}(v_j^{l+1};\xi\beta)$
for some $j$ and $\beta \in \Sigma^-$ 
such that 
$ A_{l,l+1}^+(i,\alpha,j) =1$
and
$\xi\beta \in F(v_j^{l+1}).$
Hence we obtain the equality 
\begin{equation*}
   \sum_{\underset{x = z=\sigma_{{\frak L}^-}(w), \lambda_1(w) = \alpha}{w;}} 
\chi_{U_{\Omega_{{\frak L}^-}^+}(v_i^l;\xi)}(\omega^0(w)) 
=
 \sum_{\beta \in \Sigma^-} \sum_{j=1}^{m(l+1)} 
A_{l,l+1}^+(i,\alpha,j) \chi_{U(v_j^{l+1};\xi\beta)}(x,n,z),
\end{equation*}
proving the equality 
\eqref{eq:chiUalpha}.
\end{proof}
By the formula \eqref{eq:saeas}
for a fixed $l\in \Zp$,
we have the identity
\begin{equation}
S_{\alpha}^* S_{\alpha}   =  
\sum_{i=1}^{m(l)} \sum_{j=1}^{m(l+1)} \sum_{\eta \in F(v_j^{l+1})}
A_{l,l+1}^+(i,\alpha,j)E_j^{l+1-}(\eta).
\label{eq:sassa}
\end{equation}

For $\mu =(\mu_1,\dots,\mu_m), \nu=(\nu_1,\dots,\nu_n)\in B_*(\Lambda_{{\frak L}^+})$
and
$v_i^l\in V_l, \xi \in F(v_i^l)$
with $m,n \le l$,
let $U(\mu,\nu, v_i^l;\xi)$ be
 the clopen set of $\G_{{\frak L}^-}^+$ 
defined by 
\begin{align*}
& U(\mu,\nu, v_i^l;\xi) \\
=& \{  (x,m-n, z) \in \G_{{\frak L}^-}^+
\mid   \lambda_{[1,m]}(x) =\mu, \lambda_{[1,n]}(z) = \nu,\, 
 \sigma_{{\frak L}^-}^m(x) = \sigma_{{\frak L}^-}^n(z)\in \Omega_{{\frak L}^-}(v_i^l; \xi) \}
\end{align*}
where 
$x=(\lambda_i(x), \omega^i(x))_{i=1}^\infty, 
 z=(\lambda_i(z), \omega^i(z))_{i=1}^\infty
\in X_{{\frak L}^-}^+$
and
$\lambda_{[1,m]} (x)=(\lambda_1(x),\dots,\lambda_m(x))$,
\,
$\lambda_{[1,n]}(z) =(\lambda(z)_1,\dots,\lambda_n(z))
 \in B_*(\Lambda_{{\frak L}^+}).$
For $\mu = \nu$, we write 
$U(\mu,\mu, v_i^l;\xi) $ as 
$U(\mu,v_i^l;\xi)$, that is identified with 
$
U_{X_{{\frak L}^-} ^+}(\mu,v_i^l;{\xi})
$
defined in \eqref{eq:muvil}. 
Then we have
\begin{lemma}\label{lem:smeilsm}
\begin{equation*}
S_\mu E_i^{l-}(\xi) S_\nu^* = \chi_{U(\mu,\nu,v_i^l;\xi)}.
\end{equation*}
In particular, for the clopen set $U(\mu,v_i^l;\xi)$ with
$\mu \in (\mu_1,\dots,\mu_m)\in B_m(\Lambda_{{\frak L}^+}), v_i^l \in V_l$
defined in \eqref{eq:muvil},
we have 
\begin{equation}
S_\mu E_i^{l-}(\xi) S_\mu^* = \chi_{U(\mu,v_i^l;\xi)}. \label{eq:smueilsmu}
\end{equation}
\end{lemma}
\begin{proof}
It suffices to show the equality
\begin{equation}
\chi_{U(\mu)} * \chi_{U(v_i^l;\xi)} *\chi_{U(\nu)}^* =
\chi_{U(\mu,\nu,v_i^l;\xi)} \label{eq: chiumvilunu}
\end{equation}
holds.
Suppose  
$\mu =(\mu_1,\dots,\mu_m), \nu=(\nu_1,\dots,\nu_n)\in B_*(\Lambda_{{\frak L}^+}).$
For $s =(x, p, z) \in \G_{{\frak L}^-}^+,$
we have
\begin{align*}
& [\chi_{U(\mu)} * \chi_{U(v_i^l;\xi)} *\chi_{U(\nu)}^*](x, p, z) \\
 =&  \sum_{\underset{r(t) = r(s)}{t;}}
\chi_{U(\mu)}(t) [ \chi_{U(v_i^l;\xi)} *\chi_{U(\nu)}^*](t^{-1}s) 
\qquad(\text{Put } t=(y,q,w)) \\
=&  \sum_{q,w}
\chi_{U(\mu)}(x,q,w) [ \chi_{U(v_i^l;\xi)} *\chi_{U(\nu)}^*](w, p-q, z). 
\end{align*}
We know
$(x,q,w) \in U(\mu)$ if and only if 
$q = m, \, \lambda_{[1,m]}(x) = \mu$ and $\sigma_{{\frak L}^-}^m(x) = w,$
so that we have
\begin{align*}
  & [\chi_{U(\mu)} * \chi_{U(v_i^l;\xi)} *\chi_{U(\nu)}^*](x, p, z) \\
= &  [ \chi_{U(v_i^l;\xi)} *\chi_{U(\nu)}^*]
(\sigma_{{\frak L}^-}^m(x), p-m, z) \\
=&  \sum_{\underset{r(t') = \sigma_{{\frak L}^-}^m(x)}{t';}}
 \chi_{U(v_i^l;\xi)}(t')
\chi_{U(\nu)}^*({t'}^{-1}\cdot (\sigma_{{\frak L}^-}^m(x), p-m, z)) 
\quad (\text{Put } t' = (y', q', w'), y' = \sigma_{{\frak L}^-}^m(x) )\\
=&  \sum_{q', w'}
 \chi_{U(v_i^l;\xi)}(y', q', w')
\chi_{U(\nu)}^*(w', p-m-q', z) \\
=&  \sum_{q', w'}
 \chi_{U(v_i^l;\xi)}(\sigma_{{\frak L}^-}^m(x), q', w')
\chi_{U(\nu)}(z,-p+m+q', w').
\end{align*}
Now we have
$(z,-p+m+q', w') \in U(\nu)$ if and only if
$\sigma_{{\frak L}^-}^n(z) = w', \lambda_{[1,n]}(z) = \nu$ and $ -p + m+q' = n,$
so that 
\begin{equation*}
  [\chi_{U(\mu)} * \chi_{U(v_i^l;\xi)} *\chi_{U(\nu)}^*](x, p, z) 
=  
 \chi_{U(v_i^l;\xi)}
( \sigma_{{\frak L}^-}^m(x), p+n-m,  \sigma_{{\frak L}^-}^n(z)).
\end{equation*}
As
$( \sigma_{{\frak L}^-}^m(x), p+n-m,  \sigma_{{\frak L}^-}^n(z)) \in U(v_i^l;\xi)$
with $\lambda_{[1,m]}(x) = \mu,\,\lambda_{[1,n]}(z) = \nu$
if and only if 
$(x, p, z) \in U(\mu,\nu,v_i^l;\xi),$
we have the identity \eqref{eq: chiumvilunu}.
\end{proof}
\begin{lemma}
The set of finite linear combinations of elements of the form
\begin{equation}
S_\mu E_i^{l-}(\xi) S_\nu^*, \qquad
\mu\in B_m(\Lambda_{{\frak L}^+}),\,  \nu \in  B_n(\Lambda_{{\frak L}^+}), \,  
\xi\in F(v_i^l), \, i=1,2,\dots,m(l), \,
m,n \le l \label{eq:dense}
\end{equation}
is dense in the $C^*$-algebra
$\OALMP.$
\end{lemma}
\begin{proof}
Since the sets
of the form $U(\mu,\nu,v_i^l;\xi)$ form a basis of open sets of the groupoid
$\G_{{\frak L}^-}^+,$ 
the set of finite linear combinations of elements of the form of \eqref{eq:dense}
becomes a dense $*$-subalgebra of $\OALMP$
because of Lemma \ref{lem:smeilsm}.
\end{proof}
Put for $\alpha \in \Sigma^+$
\begin{equation*}
X^+_{{\frak L}^-}(\alpha) = \{ x \in X^+_{{\frak L}^-} \mid \lambda_1(x) =\alpha\}. 
\end{equation*}
Regard it as a clopen subset 
$\{(x,0,x) \in \G_{{\frak L}^-}^+ \mid \lambda_1(x) =\alpha\}$
of $(\G_{{\frak L}^-}^+)^{(0)}$
and hence of $\G_{{\frak L}^-}^+.$
Then we have
\begin{lemma}\label{lem:sasas}
\hspace{4cm}
\begin{enumerate}
\renewcommand{\theenumi}{\roman{enumi}}
\renewcommand{\labelenumi}{\textup{(\theenumi)}}
\item
$S_\alpha S_\alpha^* = \chi_{X^+_{{\frak L}^-}(\alpha)}.$
\item
$S_\alpha S_\alpha^* E_i^{l-}(\xi) = E_i^{l-}(\xi) S_\alpha S_\alpha^*$ \,
for $\xi \in F(v_i^l).$
\end{enumerate}
\end{lemma}
\begin{proof}
(i)
By Lemma \ref{lem:smeilsm}, we have for a fixed $l \in \N,$
\begin{equation*}
S_\alpha S_\alpha^*
=  \sum_{i=1}^{m(l)} \sum_{\xi\in F(v_i^l)} S_\alpha E_i^{l-}(\xi) S_\alpha^* 
=  \sum_{i=1}^{m(l)} \sum_{\xi\in F(v_i^l)} \chi_{U(\alpha,v_i^l;\xi)} 
=  \chi_{
\cup_{i=1}^{m(l)} \cup_{\xi\in F(v_i^l)} U(\alpha,v_i^l;\xi)}. 
\end{equation*}
As
$
\bigcup_{i=1}^{m(l)} \bigcup_{\xi\in F(v_i^l)} U(\alpha,v_i^l;\xi)
=X^+_{{\frak L}^-}(\alpha) ,
$
we obtain the equality
$S_\alpha S_\alpha^* = \chi_{X^+_{{\frak L}^-}(\alpha)}.$

(ii) 
For $s = (x,n,z) \in \G_{{\frak L}^-}^+,$
we have
\begin{align*}
 &[\chi_{X^+_{{\frak L}^-}(\alpha)} *\chi_{U(v_i^l;\xi)}](x,n,z) \\
=& \sum_{\underset{r(t) = r(s)}{t; }} 
\overline{
\chi_{X^+_{{\frak L}^-}(\alpha)}(t^{-1})}
\chi_{U(v_i^l;\xi)}(t^{-1}s)  
\qquad (\text{Put } t=(y,m,w)\in \G_{{\frak L}^-}^+) \\
=& \sum_{m, w} \chi_{X^+_{{\frak L}^-}(\alpha)}(w,-m,x)
\chi_{U(v_i^l;\xi)}(w, n-m,z).  
\end{align*}
Now 
$\chi_{X^+_{{\frak L}^-}(\alpha)}(w,-m,x) =1$
if and only if $m=0, \, w = x, \, \lambda_1(x) = \alpha.$
Since we have 
$\chi_{U(v_i^l;\xi)}(x, n, z)=1$
if and only if $x=z, \, n=0, \, \omega^0(x) \in U_{\Omega_{{\frak L}^-}}(v_i^l;\xi),$
so that 
\begin{equation*}
 [\chi_{X^+_{{\frak L}^-}(\alpha)} *\chi_{U(v_i^l;\xi)}](x,n,z) 
= 
{
\begin{cases}
1 & \text{ if } x=z, \, n=0,\, \lambda_1(x) =\alpha, \, 
                   \omega^0(x) \in U_{\Omega_{{\frak L}^-}}(v_i^l;\xi),\\
0 & \text{ otherwise. }
\end{cases}
}
\end{equation*}
On the other hand
\begin{align*}
 &[\chi_{U(v_i^l;\xi)}*\chi_{X^+_{{\frak L}^-}(\alpha)}](x,n,z) \\
=& \sum_{\underset{r(t) = r(s)}{t;}} \overline{
\chi_{U(v_i^l;\xi)}(t^{-1})}
   \chi_{X^+_{{\frak L}^-}(\alpha)}(t^{-1}s)  
\qquad (\text{Put } t=(y,m,w)\in \G_{{\frak L}^-}^+) \\
=& \sum_{m, w} 
\chi_{U(v_i^l;\xi)}(w, -m, x)
   \chi_{X^+_{{\frak L}^-}(\alpha)}(w,n-m,z).  
\end{align*}
Now 
$\chi_{X^+_{{\frak L}^-}(\alpha)}(w,n-m,z) =1$
if and only if $n=m, \, w = z, \, \lambda_1(z) = \alpha.$
Since we have
$\chi_{U(v_i^l;\xi)}(z, -n,x)=1$
if and only if $z=x, \, n=0, \, \omega^0(x) \in U_{\Omega_{{\frak L}^-}}(v_i^l;\xi),$
so that 
\begin{equation*}
 [\chi_{U(v_i^l;\xi)}*\chi_{X^+_{{\frak L}^-}(\alpha)}](x,n,z) 
= 
{
\begin{cases}
1 & \text{ if } x=z, \, n=0, \, \lambda_1(x) =\alpha, \,
                   \omega^0(x) \in U_{\Omega_{{\frak L}^-}}(v_i^l;\xi),\\
0 & \text{ otherwise, }
\end{cases}
}
\end{equation*}
proving    
$\chi_{X^+_{{\frak L}^-}(\alpha)} *\chi_{U(v_i^l;\xi)}
=\chi_{U(v_i^l;\xi)}*
\chi_{X^+_{{\frak L}^-}(\alpha)} 
$
and hence
$S_\alpha S_\alpha^* E_i^{l-}(\xi) = E_i^{l-}(\xi) S_\alpha S_\alpha^*.$
\end{proof}

\begin{proposition}\label{prop:main6}
\hspace{4cm}
\begin{enumerate}
\renewcommand{\theenumi}{\roman{enumi}}
\renewcommand{\labelenumi}{\textup{(\theenumi)}}
\item
The $C^*$-algebra $\OALMP$
is generated  by
partial isometries
$S_{\alpha}$ indexed by
$\alpha \in \Sigma^+$
and mutually commuting projections
$E_i^{l-}(\xi)$ indexed by vertices $v_i^l \in  V_l$ and 
admissible words $\xi =(\xi_1,\dots, \xi_l) \in F(v_i^l).$ 
\item
The partial isometries
$S_{\alpha}, \alpha \in \Sigma^+$
and the mutually commuting projections
$E_i^{l-}(\xi), \xi \in F(v_i^l)$
satisfy the  following operator relations called $\LGBS$:
\begin{align}
\sum_{\alpha \in \Sigma^+} S_{\alpha}S_{\alpha}^*   
& =  \sum_{i=1}^{m(l)} \sum_{\xi \in F(v_i^l)} E_i^{l-}(\xi) =  1, \label{eq:L2}\\ 
 S_\alpha S_\alpha^* & E_i^{l-}(\xi)   =   E_i^{l-}(\xi) S_\alpha S_\alpha^*, 
\label{eq:L3}\\
 E_i^{l-} (\xi)   & = \sum_{\beta\in \Sigma^-}  \sum_{j=1}^{m(l+1)}
A_{l,l+1}^-(i,\beta,j)E_j^{l+1-}(\beta\xi),  \label{eq:L4}\\
S_{\alpha}^*E_i^{l-}(\xi) S_{\alpha} &  =  
\sum_{\beta\in \Sigma^-}
\sum_{j=1}^{m(l+1)} A_{l,l+1}^+(i,\alpha,j)E_j^{l+1-}(\xi\beta). \label{eq:L6}
\end{align}
\end{enumerate}
\end{proposition}
Similarly we have
\begin{proposition}\label{prop:main7}
\hspace{4cm}
\begin{enumerate}
\renewcommand{\theenumi}{\roman{enumi}}
\renewcommand{\labelenumi}{\textup{(\theenumi)}}
\item
The $C^*$-algebra $\OALPM$
is generated  by
partial isometries
$T_{\beta}$ indexed by $\beta \in \Sigma^-$
and mutually commuting projections
$E_i^{l+}(\eta)$ indexed by vertices $v_i^l \in  V_l$ and 
admissible words $\eta =(\eta_1,\dots, \eta_l) \in P(v_i^l).$ 
\item
The partial isometries
$T_{\beta}, \beta \in \Sigma^-$
and the mutually commuting projections
$E_i^{l+}(\eta), \eta \in P(v_i^l)$
satisfy the  following operator relations called $({\frak L}^+,{\frak L}^-)$:
\begin{align}
\sum_{\beta \in \Sigma^-} T_{\beta}T_{\beta}^*   
& =  \sum_{i=1}^{m(l)} \sum_{\eta \in P(v_i^l)}  E_i^{l+}(\eta) =  1, \label{eq:P2}\\ 
 T_\beta T_\beta^* & E_i^{l+}(\eta)   =   E_i^{l+}(\eta) T_\beta T_\beta^*, 
\label{eq:P3}\\
 E_i^{l+} (\eta)   & = \sum_{\alpha\in \Sigma^+}  \sum_{j=1}^{m(l+1)}
A_{l,l+1}^+(i,\alpha,j)E_j^{l+1+}(\eta\alpha),  \label{eq:P4}\\
T_{\beta}^*E_i^{l+}(\eta) T_{\beta} &  =  
\sum_{\alpha\in \Sigma^+}
\sum_{j=1}^{m(l+1)} A_{l,l+1}^-(i,\beta,j)E_j^{l+1+}(\alpha\eta). \label{eq:P6}
\end{align}
\end{enumerate}
\end{proposition}
We will prove in the following section that 
the above operator relations among the generators of the $C^*$-algebras
exactly determine the algebraic structure of the $C^*$-algebras.

\section{Structure of the $C^*$-algebra $\OALMP$}
In what follows, an endomorphism on a unital $C^*$-algebra means 
a $*$-endomorphism that is not  necessarily unital. 
For a unital $C^*$-algebra $\A$, let us denote by
$\End(\A)$ the set of endomorphisms on $\A$. 
In \cite{MaCrelle}, the notion of $C^*$-symbolic dynamical system
$(\A,\rho,\Sigma)$ was introduced as a generalization of both a $\lambda$-graph system and an automorphism on a unital $C^*$-algebra.
Following \cite{MaCrelle}, a finite family 
$\rho_\alpha \in \End(\A), \alpha \in \Sigma$
of endomorphisms on a unital $C^*$-algebra $\A$
indexed by a finite alphabet $\Sigma$ 
is said to be {\it essential}\/ if $\rho_\alpha(1) \ne 0$
for all $\alpha \in \Sigma$
and the ideal of $\A$ generated by $\rho_\alpha(1), \alpha \in \Sigma$
coincides with $\A$.
It is said to be {\it faithful}\/ if for any nonzero 
$a \in \A$, there exists a symbol $\alpha \in \Sigma$ such that 
$\rho_\alpha(a) \ne 0.$
A $C^*$-{\it symbolic dynamical system}\/ is defined by a triplet 
$(\A,\rho, \Sigma)$ consisting on a unital $C^*$-algebra $\A$ and 
a finite family of endomorphisms    
$\{\rho_\alpha \}_{\alpha \in \Sigma}$ of $\A$, that is essential and faithful
(\cite{MaCrelle}, cf. \cite{MaCM2009}, \cite{MaMZ2010}).
A $C^*$-symbolic dynamical system
$(\A,\rho,\Sigma)$ gives rise to a subshift
$\Lambda$ and a $C^*$-algebra written 
$\A\rtimes_\rho\Lambda$
called the $C^*$-symbolic crossed product in \cite{MaCrelle}.
The  $C^*$-algebra $\A\rtimes_\rho\Lambda$
is constructed by certain Hilbert $C^*$-bimodule associated to 
the endomorphisms $\{\rho_\alpha \}_{\alpha \in \Sigma}$ of $\A$
(cf.  \cite{KPW}, \cite{KW}, \cite{Pim}, etc.).
If $(\A,\rho,\Sigma)$ satisfies condition (I)
in the sense of \cite{MaCM2009},
the algebraic structure of the
$C^*$-algebra $\A\rtimes_\rho\Lambda$
is uniquely determined by certain operator relations among its canonical generators. 
In this section, we will show that our $C^*$-algebra
$\OALMP$ may be realized as a $C^*$-symbolic crossed product 
$\ALM\rtimes_{\rho^+}\Lambda_{{\frak L}^+},$
so that  we will know that 
the operator relations called 
$\LGBS$ in Proposition \ref{prop:main6}
uniquely determine the algebraic structure of the $C^*$-algebra
$\OALMP$
under certain condition called 
 condition (I) on the $C^*$-symbolic dynamical system
$(\ALM, {\rho^+}, \Sigma^+)$.


Let $\ALM$  be the $C^*$-subalgebra
of $\OALMP$
generated by the mutually commuting projections
$E_i^{l-}(\xi)$ for $ v_i^l \in V_l, \xi\in F(v_i^l)$.
The $C^*$-subalgebra
$\ALP$ of
$\OALPM$ is similarly defined.
Let us denote by
$\rho_\alpha^+\in \End(\ALM), \alpha \in \Sigma^+$
the endomorphism on $\ALM$ defined by setting
\begin{equation}
\rho_\alpha^+(E_i^{l-}(\xi))  =  
\sum_{\beta \in \Sigma^-}
\sum_{j=1}^{m(l+1)} A_{l,l+1}^+(i,\alpha,j)E_j^{l+1-}(\xi\beta), \qquad \xi\in F(v_i^l).
\label{eq:defofrhoalphaA}
\end{equation}
Then we have a finite family of endomorphisms
$\rho_\alpha^+, \alpha \in \Sigma^+$
on $\ALM.$
The other endomorphisms
$\rho_\beta^-\in \End(\ALP), \beta \in \Sigma^-$
are similarly defined.
\begin{lemma}\label{lem:CsymbdynA}
The triplets 
$(\ALM, \rho^+, \Sigma^+)$
and similarly 
$(\ALP, \rho^-, \Sigma^-)$
are both $C^*$-symbolic dynamical systems.
\end{lemma}
\begin{proof}
Since we have
\begin{align*}
\sum_{\alpha\in \Sigma^+}\rho^+_\alpha(1) 
= & \sum_{\alpha\in \Sigma^+} \sum_{i=1}^{m(l)} \sum_{\xi \in F(v_i^l)}
\rho^+_\alpha(E_i^{l-}(\xi)) \\
= & \sum_{\alpha\in \Sigma^+} \sum_{i=1}^{m(l)} \sum_{\xi \in F(v_i^l)}
\sum_{\beta \in \Sigma^-}
\sum_{j=1}^{m(l+1)} A_{l,l+1}^+(i,\alpha,j)E_j^{l+1-}(\xi\beta) 
\ge  1,
\end{align*}
the family
$\{ \rho_\alpha^+\}_{\alpha \in \Sigma^+}$
is essential.
It is easy to see that $\{ \rho_\alpha^+\}_{\alpha \in \Sigma^+}$
is faithful, so that 
the triplets 
$(\ALM, \rho^+, \Sigma^+)$
and similarly 
$(\ALP, \rho^-, \Sigma^-)$
are both $C^*$-symbolic dynamical systems.
\end{proof}
We will concentrate on the algebra $\OALMP$,
the other algebra $\OALPM$ has a symmetric structure.
We define a $C^*$-subalgebra 
$\DLMP$ of $\OALMP$ by
\begin{equation}
\DLMP
= C^*( S_\mu E_i^{l-}(\xi) S_\mu^* \mid i=1,2,\dots, m(l),\, \xi \in F(v_i^l), \, \mu \in B_*(\Lambda_{{\frak L}^+}) ). \label{eq:DLMP}
\end{equation} 
Let $\varphi_{\frak{L}^-}: X_{\frak{L}^-}^+\longrightarrow \Omega_{\frak{L}^-}$
be the continuous surjection defined by 
$\varphi_{\frak{L}^-}(x) = \omega^0(x), \, x \in X_{\frak{L}^-}^+.$
It induces an embedding
$C(\Omega_{\frak{L}^-}) \hookrightarrow C( X_{\frak{L}^-}^+)$
corresponding to a natural inclusion
$\ALM\subset\DLMP.$
Recall the condition (I) for $C^*$-symbolic dynamical systems introduced in
\cite{MaMZ2010}.
\begin{definition}[{\cite[Section 3, Definition]{MaMZ2010}}]
The $C^*$-symbolic dynamical system 
$(\A_{\frak{L}^-},\rho^+,\Sigma)$
satisfies {\it condition }(I) if
there exists a unital increasing sequence
$$
\A_1 \subset \A_2 \subset \cdots \subset \A_{\frak{L}^-}
$$
of $C^*$-subalgebras of $\A_{\frak{L}^-}$ 
such that 
$\rho^+_\alpha(\A_l) \subset \A_{l+1}$ 
for all $l \in \N, \alpha \in \Sigma^+$
 and 
the union $\bigcup_{l=1}^\infty \A_l $ is dense in $\A_{\frak{L}^-}$
and 
for $k,l\in \N$ with $k \le l$,
there exists a projection 
$
q_k^l \in \DLMP
$
commuting with all elements of 
$\A_l$
such that
\begin{enumerate}
\renewcommand{\theenumi}{\arabic{enumi}}
\renewcommand{\labelenumi}{\textup{(\theenumi)}}
\item $q_k^l a \ne 0 $ for all nonzero $a \in \A_l$,
\item $q_k^l \phi_{_{\frak{L}^+}}^m(q_k^l) = 0 $ for all $m= 1,2,\dots, k,$
\end{enumerate}
where 
$
 \phi_{_{\frak{L}^+}}^m(X) = 
 \sum_{\mu \in B_m(\Lambda_{{\frak L}^+})} S_\mu X S_\mu^{*}.
$

By Proposition \ref{prop:main6} (ii) and \eqref{eq:defofrhoalphaA},
we know that our partial isometries $S_\alpha,\alpha \in \Sigma^+$  
satisfy the relations
\begin{equation}
\sum_{\alpha\in \Sigma^+} S_\alpha S_\alpha^* = 1, \qquad 
S_\alpha S_\alpha^* E_i^{l-}(\xi) = E_i^{l-}(\xi)S_\alpha S_\alpha^*, \qquad
\rho^+_\alpha(E_i^{l-}(\xi)) = S_\alpha^* E_i^{l-}(\xi) S_\alpha \label{eq:csds}
\end{equation} 
for $\alpha \in \Sigma^+, \xi \in F(v_i^l)$.
Let $s_\alpha, \alpha\in \Sigma^+$ be another family of partial isometries satisfying the relations \eqref{eq:csds}.
We may consider the corresponding $C^*$-algebra
$\mathcal{D}_{{\frak L}^-}^+ $, that is generated by projections
of the form $s_\mu E_i^{l-}(\xi) s_\mu^*$,
 and the homomorphism 
$\phi_{_{\frak{L}^+}}^m$ on it by using $s_\alpha, \alpha\in \Sigma^+$,
instead of $S_\alpha, \alpha\in \Sigma^+$.
 Then \cite[Lemma 3.2]{MaMZ2010} tells us that the condition (I) 
does not depend on the choice of such partial isometries 
satisfying the relations \eqref{eq:csds}.
Hence 
 the condition (I) for $(\A_{\frak{L}^-},\rho^+,\Sigma^+)$
 is intrinsically determined by $(\A_{\frak{L}^-},\rho^+,\Sigma^+)$
 from \cite[Lemma 3.2]{MaMZ2010}.
\end{definition}


\begin{definition}[{\cite[p. 19]{Renault2}}]
The topological dynamical system
$(X_{{\frak L}^-}^+, \sigma_{{\frak L}^-})$ is said to be essentially free
if  the set 
$X_{m,n}(\sigma_{{\frak L}^-}) =\{ x \in X_{{\frak L}^-}^+ \mid 
\sigma_{{\frak L}^-}^m(x) = \sigma_{{\frak L}^-}^n(x) \}
$
for $m,n \in \Zp$ with $m\ne n$ does not have non empty interior.
\end{definition}
A point $x\in X_{{\frak L}^-}^+$ is said to be eventually periodic
if $\sigma_{{\frak L}^-}^m(x) = \sigma_{{\frak L}^-}^n(x)$ for some $m,n\in \Zp$ 
with $m\ne n.$
The set of  eventually periodic points in $X_{{\frak L}^-}^+$ 
is denoted by 
 $\Pev(\sigma_{{\frak L}^-}).$ 
Hence we have
\begin{equation}
\Pev(\sigma_{{\frak L}^-}) = 
\bigcup_{\underset{m\ne n}{m,n;}}X_{m,n}(\sigma_{{\frak L}^-}).\label{eq:PevX}
\end{equation}
The following lemma is known for more general dynamical system.
As the author has not been able to find a complete proof in literature,
the proof is given for the sake of completeness.
\begin{lemma}[{cf. \cite[Proposition 3.1]{Renault2}}]
The topological dynamical system
$(X_{{\frak L}^-}^+, \sigma_{{\frak L}^-})$ is essentially free
if and only if 
the set $\Pev(\sigma_{{\frak L}^-})^c$ of non-eventually periodic points is dense in 
$X_{{\frak L}^-}^+.$
\end{lemma}
\begin{proof}
Assume that 
$(X_{{\frak L}^-}^+, \sigma_{{\frak L}^-})$ is essentially free.
Suppose that 
the set $\Pev(\sigma_{{\frak L}^-})^c$ of non-eventually periodic points is not dense in 
$X_{{\frak L}^-}^+.$
Since $X_{{\frak L}^-}^+$ is compact Hausdorff and hence regular,
there exists a point $x \in X_{{\frak L}^-}^+$ and an open neighborhood
$U_x \subset X_{{\frak L}^-}^+$ of $x$
such that 
\begin{equation*}
\overline{U_x} \cap \overline{\Pev(\sigma_{{\frak L}^-})^c} =\emptyset.
\end{equation*} 
Hence   
we have
$\overline{U_x} \subset \Pev(\sigma_{{\frak L}^-})$
so that 
\begin{equation*}
\overline{U_x} = \bigcup_{\underset{m\ne n}{m,n;}} ( X_{m,n}(\sigma_{{\frak L}^-}) 
\cap \overline{U_x}). 
\end{equation*} 
By the Baire's category theorem, 
there exist $m,n \in \Zp$ with $m\ne n$ such that 
$ X_{m,n}(\sigma_{{\frak L}^-}) 
\cap \overline{U_x}$
contains an interior point in the set 
$ \overline{U_x}.$
Therefore we conclude that 
$ X_{m,n}(\sigma_{{\frak L}^-}) $ contains an interior, a contradiction to the hypothesis that  
$(X_{{\frak L}^-}^+, \sigma_{{\frak L}^-})$ is essentially free.

Assume next that the set  
$\Pev(\sigma_{{\frak L}^-})^c$ 
is dense in $X_{{\frak L}^-}^+.$
By \eqref{eq:PevX}, we know that 
if 
$X_{m,n}(\sigma_{{\frak L}^-})$ contains an open set $V$  for some 
$m,n$ with $m\ne n,$
then 
$\Pev(\sigma_{{\frak L}^-})$ contains $V$,
 a contradiction to the hypothesis that
 $\Pev(\sigma_{{\frak L}^-})^c$ 
is dense in $X_{{\frak L}^-}^+.$
\end{proof}

The essential freeness of the topological dynamical system
$(X_{{\frak L}^-}^+, \sigma_{{\frak L}^-})$
 is  equivalent to the condition that 
the \'etale groupoid
$\G_{{\frak L}^-}^+$ is essentially principal (see \cite[Proposition 3.1]{Renault2}),
so that the $C^*$-subalgebra
$\DLMP$ is maximal abelian in $\OALMP$
by \cite[Proposition 4.7]{Renault}.

Recall that a clopen set $U_{X_{{\frak L}^-}^+}(v_i^l;{\xi})$
for $v_i^l \in V_l,  \xi \in F(v_i^l),$
 in $X_{{\frak L}^-}^+$ 
is defined in \eqref{eq:uxvilxi}.
\begin{definition}
A  $\lambda$-graph bisystem $\LGBS$
is said to satisfy $\sigma_{{\frak L}^-}$-{\it condition}\/ (I)
if for any $l,k\in \N$ with $k\le l,$
there exist
$x_i^l(\xi) \in  U_{X^+_{{\frak{L}^{-}}}(v_i^l;\xi)}$
for each $i=1,2,\dots, m(l)$ and $\xi \in F(v_i^l)$
such that 
\begin{equation}
\sigma_{{\frak L}^-}^n(x_i^l(\xi)) \ne x_j^l(\eta)
\quad
\text{for }
\xi\in F(v_i^l), \eta \in F(v_j^l),\, i,j =1,2,\dots,m(l),\,  n=1,2,\dots,k \label{eq:7.6.1}
\end{equation}
(cf. \cite[Lemma 5.1]{MaJMSJ1999}).
\end{definition}

\begin{proposition}\label{prop:conditionI}
Let $\LGBS$ be a $\lambda$-graph bisystem.
Consider the following three conditions:
\hspace{4cm}
\begin{enumerate}
\renewcommand{\theenumi}{\roman{enumi}}
\renewcommand{\labelenumi}{\textup{(\theenumi)}}
\item
The  $\lambda$-graph bisystem $\LGBS$ satisfies $\sigma_{{\frak L}^-}$-condition (I).
\item
The topological dynamical system
$(X_{{\frak L}^-}^+, \sigma_{{\frak L}^-})$ is essentially free.
\item
The $C^*$-symbolic dynamical system
$(\ALM, \rho^+, \Sigma^+)$
satisfies condition (I).
\end{enumerate}
Then we have 
(i) $\Longrightarrow $ (ii) and 
(i) $\Longrightarrow $ (iii).
\end{proposition}
\begin{proof}
(i) $\Longrightarrow $ (ii) :
Assume  that the $\lambda$-graph bisystem
$\LGBS$ satisfies $\sigma_{{\frak L}^-}$-condition (I) 
and the topological dynamical system
$(X_{{\frak L}^-}^+, \sigma_{{\frak L}^-})$ is not essentially free.
There exist $m,n \in \Zp$ with $m> n$ 
such that  
$X_{m,n}(\sigma_{{\frak L}^-}) =\{ x \in X_{{\frak L}^-}^+ \mid 
\sigma_{{\frak L}^-}^m(x) = \sigma_{{\frak L}^-}^n(x) \}
$
has a nonempty interior.
 By taking $l\in \N$ large enough
such as 
$l >m $, we may assume that 
$U(\mu,v_i^l;\xi) \subset X_{m,n}(\sigma_{{\frak L}^-})$
for some $\mu =(\mu_1,\dots,\mu_p) \in B_p(\Lambda_{{\frak L}^+})$ 
and
$v_i^l \in V_l.$
Since the numbers $m$ and $p$ may be taken large enough,
we may assume that $p=m.$
Take $x \in U(\mu,v_i^l;\xi)$
and let
 $ \omega^p(x) = (u_l, \beta_{-l})_{l=1}^\infty \in  \Omega_{{\frak L}^-}.$
We put
$$
\xi_{l+1}:= \beta_{-p-1}, \, \,
 \xi_{l+2}:= \beta_{-p}, \, \, 
\dots, \,\,
\xi_{l+p}:= \beta_{-2}, \, \, 
\xi_{l+p+1}:= \beta_{-1}
\, \, \,
\text{ and }
\, \, \,
v_{i_0}^{l+p}: = u_{l+p}.
$$  
Let
$\bar{\xi} = (\xi_1, \xi_2,\dots, \xi_l, \xi_{l+1}, \dots,\xi_{l+p+1}) 
\in B_{l+p+1}(\Lambda_{{\frak L}^-}),$
so that 
we have
$$
U_{X^+_{\frak{L}^-}(v_{i_0}^{l+p}; \bar{\xi}) }
\subset
\sigma_{{\frak L}^-}^p(U(\mu,v_i^l;\xi) ).
$$
As
$U(\mu,v_i^l;\xi) \subset X_{m,n}(\sigma_{{\frak L}^-}),$
any point of 
$
U_{X^+_{\frak{L}^-}(v_{i_0}^{l+p}; \bar{\xi}) }
$ consists of periodic points with its period $m-n.$
Now take $k \in \N$ such as $k > m-n.$
Then there exists no points 
$x_{i_0}^{l+p}(\bar{\xi}) \in 
U_{X^+_{\frak{L}^-}(v_{i_0}^{l+p}; \bar{\xi}) }
$
such that 
$\sigma_{{\frak L}^-}^n(x_{i_0}^{l+p}(\bar{\xi})) \ne 
x_{i_0}^{l+p}(\bar{\xi})
$ for all
$n=1,2,\dots,k.$
It is a contradiction to $\sigma_{{\frak L}^-}$-condition (I).

(i) $\Longrightarrow$ (iii) : 
For a fixed $l \in \N$, let $\A_{\frak{L}^{-},l}$ be the $C^*$-subalgebra
of $\ALM$ generated by the projections
$E_i^{l-}(\xi), \xi \in F(v_i^l), i=1,2,\dots,m(l).$
It satisfies the condition
\begin{equation*}
\rho_\alpha^+(\A_{\frak{L}^{-},l}) \subset \A_{\frak{L}^{-},l+1}, 
\qquad l \in \N, \, \alpha \in \Sigma^+,
\end{equation*}
and the union $\bigcup_{l=1}^\infty \A_{\frak{L}^{-},l}$ 
is dense in $\ALM.$
Since 
the $\lambda$-graph bisystem  $\LGBS$ satisfies $\sigma_{{\frak L}^-}$-condition (I),
 for any $l,k\in \N$ with $k\le l,$
there exist
$x_i^l(\xi) \in  U_{X^+_{\frak{L}^{-}}(v_i^l;\xi)}$
for each $i=1,2,\dots, m(l)$ and $\xi \in F(v_i^l)$
satsifying \eqref{eq:7.6.1}. 
Under fixing 
$l,k\in \N$ with $k\le l,$
we set
$$
Y = \{ x_i^l(\xi) \mid i=1,2,\dots, m(l), \xi \in F(v_i^l)\}.
$$
By \eqref{eq:7.6.1},
we have 
$\sigma_{{\frak L}^-}^n(Y) \cap Y = \emptyset$
for all $n=1,2,\dots,k.$
We may find a clopen set $V\subset X_{\frak{L}^-}^+$
such that $Y \subset V$ and  
$$
\sigma_{{\frak L}^-}^n(V) \cap V = \emptyset
\quad
\text{ for all }
n=1,2,\dots,k.
$$
Let $q_k^l$ be the projection of the characteristic function 
of $V$ on $X_{\frak{L}^{-}}^{+}.$
Since 
$
\sigma_{{\frak L}^-}^n(V) \cap V = \emptyset
$
for all
$n=1,2,\dots,k,$
we know that 
$
q_k^l \phi_{{\frak L}^+}^n(q_k^l) =0
$
for all
$n=1,2,\dots,k.$
Any nonzero element
$
a \in \A_{{\frak{L}}^-,l}
$ 
is of the form
$
 a = \sum_{i=1}^{m(l)}\sum_{\xi \in F(v_i^l)}c_i^l(\xi)E_i^{l-}(\xi)
$
for some $c_i^l(\xi) \in \mathbb{C}.$
Since $a \ne 0,$
there exists $i_0, \xi_0$ such that 
$c_{i_0}^l(\xi_0) \ne 0.$
Take $x_{i_0}^l(\xi_0) \in Y \subset V$
so that we have 
\begin{equation*}
(a q_k^l)(x_{i_0}^l(\xi_0))
= a(\varphi_{{\frak{L}^-}}(x_{i_0}^l(\xi_0))) q_k^l(x_{i_0}^l(\xi_0))
 =c_{i_0}^l(\xi_0) \ne 0,
\end{equation*}
and hence 
$a q_k^l \ne 0.$
\end{proof}
\begin{theorem}\label{thm:main6}
Suppose that a $\lambda$-graph bisystem $\LGBS$ 
satisfies $\sigma_{{\frak L}^-}$-condition (I).
Then the $C^*$-algebra $\OALMP$
is  the universal unital unique $C^*$-algebra
generated by
partial isometries
$S_{\alpha}$ indexed by symbols $\alpha \in \Sigma^+$
and mutually commuting projections
$E_i^{l-}(\xi)$ indexed by vertices $v_i^l \in  V_l$ and 
admissible words $\xi =(\xi_1,\dots, \xi_l) \in F(v_i^l)$ 
 subject to the  following operator relations called $\LGBS$:
\begin{align}
\sum_{\alpha \in \Sigma^+} S_{\alpha}S_{\alpha}^*  
 & = \sum_{i=1}^{m(l)} \sum_{\xi \in F(v_i^l)}  E_i^{l-}(\xi) = 1, \label{eq:L2}\\ 
 S_\alpha S_\alpha^* & E_i^{l-}(\xi)   =   E_i^{l-}(\xi) S_\alpha S_\alpha^*, 
\label{eq:L3}\\
 E_i^{l-} (\xi)   & = \sum_{\beta\in \Sigma^-}  \sum_{j=1}^{m(l+1)}
A_{l,l+1}^-(i,\beta,j)E_j^{l+1-}(\beta\xi),  \label{eq:L4}\\
S_{\alpha}^*E_i^{l-}(\xi) S_{\alpha} &  =  
\sum_{\beta\in \Sigma^-}
\sum_{j=1}^{m(l+1)} A_{l,l+1}^+(i,\alpha,j)E_j^{l+1-}(\xi\beta), \label{eq:L6}
\end{align}
where the word
$\beta\xi $ in \eqref{eq:L4} 
and
$\xi\beta $ in \eqref{eq:L6} 
are defined by 
$\beta \xi = (\beta, \xi_1,\dots,\xi_l)$
and
$\xi \beta = (\xi_1,\dots,\xi_l,\beta)$
for $\beta \in \Sigma^-,
\xi = (\xi_1,\dots,\xi_l) \in F(v_i^l)$
and
$
i=1,2,\dots,m(l) 
$,
respectively.
\end{theorem}
\begin{proof}
The uniqueness of the $C^*$-algebra $\OALMP$
among the generators
$S_{\alpha}, \alpha \in \Sigma^+$
and 
$E_i^{l-}(\xi), v_i^l \in  V_l, \xi \in F(v_i^l)$
subject to the operator relations $\LGBS$
means that if there exist 
another family of 
nonzero partial isometries 
$\widehat{S}_{\alpha}, \alpha \in \Sigma^+$
and nonzero mutually commuting projections
$\widehat{E}_i^{l-}(\xi), v_i^l \in  V_l, \xi \in F(v_i^l)$
satisfying the above operator relations
$\LGBS,$
then the correspondence
\begin{equation*}
S_\alpha \longrightarrow \widehat{S}_{\alpha}, \qquad
E_i^{l-}(\xi) \longrightarrow \widehat{E}_i^{l-}(\xi)
\end{equation*}
yield an isomorphism from $\OALMP$
onto the $C^*$-algebra
$\widehat{\mathcal{O}}_{\frak{L}^-}^+ $
generated by
$\widehat{S}_{\alpha}, \alpha \in \Sigma^+$
and 
$\widehat{E}_i^{l-}(\xi), v_i^l \in  V_l, \xi \in F(v_i^l).$
 We will prove this property.
Let us  denote by 
$\widehat{\mathcal{A}}_{\frak{L}^-} $
the $C^*$-subalgebra of 
$\widehat{\mathcal{O}}_{\frak{L}^-}^+ $
generated by
the projections
$\widehat{E}_i^{l-}(\xi), v_i^l \in  V_l,
\xi =(\xi_1,\dots, \xi_l) \in F(v_i^l).$
By the relations below
\begin{equation*}
 \sum_{i=1}^{m(l)} \sum_{\xi \in F(v_i^l)}  \widehat{E}_i^{l-}(\xi) = 1, \qquad 
 \widehat{E}_i^{l-} (\xi)   = \sum_{\beta\in \Sigma^-}  \sum_{j=1}^{m(l+1)}
A_{l,l+1}^-(i,\beta,j) \widehat{E}_j^{l+1-}(\beta\xi)
\end{equation*}
and commutativity of the projections
$E_i^{l-}(\xi)$,
we know that 
the correspondence 
$$
 \widehat{E}_i^{l-} (\xi)  \in \widehat{\mathcal{A}}_{\frak{L}^-} \longrightarrow 
\chi_{U_{\Omega_{\frak{L}^-}(v_i^l;\xi)}} \in C(\Omega_{\frak{L}^-})
$$
gives rise to an isomorphism of $C^*$-algebras between
$
 \widehat{\mathcal{A}}_{\frak{L}^-} 
$
and
$C(\Omega_{\frak{L}^-}).
$
Hence by Lemma \ref{lem:abelian} (ii),
the $C^*$-algebras
$\mathcal{A}_{\frak{L}^-}$ 
and
$
 \widehat{\mathcal{A}}_{\frak{L}^-} 
$
are canonically isomorphic through the correspondence
$
{E}_i^{l-} (\xi)  \in \mathcal{A}_{\frak{L}^-} 
\longrightarrow 
\widehat{E}_i^{l-} (\xi)  \in \widehat{\mathcal{A}}_{\frak{L}^-}.
$
Now we are assuming that the $\lambda$-graph bisystem $\LGBS$
satisfies $\sigma_{{\frak L}^-}$-condition (I),
so that the $C^*$-symbolic dynamical system
$
(\mathcal{A}_{\frak{L}^-},\rho^+, \Sigma^+) 
$
satisfies condition (I).
By \cite[Theorem 3.9]{MaMZ2010}, we know that 
the correspondence
\begin{align*}
{E}_i^{l-} (\xi)  \in \mathcal{A}_{\frak{L}^-} 
&\longrightarrow 
 \widehat{E}_i^{l-} (\xi)  \in \widehat{\mathcal{A}}_{\frak{L}^-}, \\
S_\alpha   \in \OALMP 
&\longrightarrow 
 \widehat{S}_\alpha  \in \widehat{\mathcal{O}}_{\frak{L}^-}^+, 
\end{align*}
extends to an isomorphism of $C^*$-algebras 
between 
$\OALMP$ and
$\widehat{\mathcal{O}}_{\frak{L}^-}^+. 
$
Therefore the $C^*$-algebra $\OALMP$
is the universal $C^*$-algebra
 subject to the  operator relations $\LGBS$.
 The above discussion shows that $\OALMP$
is the unique $C^*$-algebra subject to the operator relations
$\LGBS$. 
\end{proof}
\medskip
For the other $C^*$-algebra 
$\OALPM,$ we may similarly define the topological dynamical system
$(X_{{\frak L}^+}^-, \sigma_{{\frak L}^+})$ to be essentially free
and
the $\lambda$-graph bisystem $\LGBS$ to satisfy $\sigma_{{\frak L}^+}$-condition (I).
We may show a symmetric statement to Proposition \ref{prop:conditionI} 
to lead the following theorem.
\begin{theorem}\label{thm:main7'}
Suppose that a $\lambda$-graph bisystem $\LGBS$ satisfies 
$\sigma_{{\frak L}^+}$-condition (I).
The $C^*$-algebra $\OALPM$
is realized as the universal unital unique $C^*$-algebra
generated by partial isometries
$T_{\beta}$ indexed by symbols $\beta \in \Sigma^-$ 
and mutually commuting projections
$E_i^{l+}(\mu)$ indexed by vertices $v_i^l \in  V_l$ and 
admissible words $\mu =(\mu_1,\dots, \mu_l) \in P(v_i^l)$ 
 subject to the following operator relations called $({\frak L}^+,{\frak L}^-)$:
\begin{align}
\sum_{\beta \in \Sigma^-} T_{\beta}T_{\beta}^*   
& =  \sum_{i=1}^{m(l)} \sum_{\mu \in P(v_i^l)}  E_i^{l+}(\mu) =  1, \label{eq:P2}\\ 
 T_\beta T_\beta^* & E_i^{l+}(\mu) =  E_i^{l+}(\mu) T_\beta T_\beta^*, 
\label{eq:P3}\\
 E_i^{l+} (\mu)   & = \sum_{\alpha\in \Sigma^+}  \sum_{j=1}^{m(l+1)}
A_{l,l+1}^+(i,\alpha,j)E_j^{l+1+}(\mu\alpha),  \label{eq:P4}\\
T_{\beta}^*E_i^{l+}(\mu) T_{\beta} &  =  
\sum_{\alpha\in \Sigma^+}
\sum_{j=1}^{m(l+1)} A_{l,l+1}^-(i,\beta,j)E_j^{l+1+}(\alpha\mu) \label{eq:P6}
\end{align}
for $\beta \in \Sigma^-, \mu \in P(v_i^l),$ where
$\mu \alpha = (\mu_1,\dots,\mu_l, \alpha)$ and
$\alpha\mu = (\alpha,\mu_1,\dots,\mu_l).$
\end{theorem}

\medskip
We will next present the operator relations 
$\LGBS$ in Theorem \ref{thm:main6} as well as Theorem \ref{thm:main7'} into 
a simpler form than the above relations in Theorem \ref{thm:main6} as well as  Theorem \ref{thm:main7'}. 
For $l \in \N,$ $v_i^l \in V_l$
and $\beta \in \Sigma^-,$
we put
\begin{gather*}
F_\beta(v_i^l) =
\{ (\xi_1, \dots, \xi_l) \in F(v_i^l) \mid \xi_1 = \beta \} 
\subset B_l(\Lambda_{{\frak L}^-}), \\
\Sigma_1^-(v_i^l) =
\{ \lambda^-(e^-) \in \Sigma^- \mid  e^- \in E_{l,l-1}^-, s(e^-) = v_i^l\}.
\end{gather*}
We then see
$
\Sigma_1^-(v_i^l) =
\{ \beta\in \Sigma^- \mid  F_\beta(v_i^l) \ne \emptyset \}.
$
We define a projection for $\beta \in \Sigma_1^-(v_i^l)$ in $\ALM$
by
\begin{equation*}
E_i^l(\beta) := \sum_{\xi \in F_\beta(v_i^l)} E_i^{l-}(\xi) \qquad 
\text{ for } \quad \beta \in \Sigma_1^-(v_i^l). 
\end{equation*}
In case of $F_\beta(v_i^l) =\emptyset,$ we define $E_i^l(\beta) =0.$
We have the following lemma.
\begin{lemma}\label{lem:relations}
\hspace{6cm}
\begin{enumerate}
\renewcommand{\theenumi}{\roman{enumi}}
\renewcommand{\labelenumi}{\textup{(\theenumi)}}
\item
${\displaystyle\sum_{i=1}^{m(l)}  \sum_{\beta \in \Sigma^-_1(v_i^l)}  E_i^{l}(\beta) = 1.} $ 
\item
$ S_\alpha S_\alpha^*  E_i^{l}(\beta)   =   E_i^{l}(\beta) S_\alpha S_\alpha^*.$
\item
${\displaystyle\sum_{\beta \in \Sigma^-_1(v_i^l)} E_i^{l} (\beta) 
 =  \sum_{j=1}^{m(l+1)}\sum_{\gamma \in \Sigma_1^-(v_j^{l+1})}
A_{l,l+1}^-(i,\gamma,j)E_j^{l+1}(\gamma).}$
\item
${\displaystyle S_{\alpha}^*E_i^{l}(\beta) S_{\alpha}   =  
\sum_{j=1}^{m(l+1)} A_{l,l+1}^+(i,\alpha,j)E_j^{l+1}(\beta).}
$
\end{enumerate}
\end{lemma}
\begin{proof}
(i)
Since $F(v_i^l)= 
\sqcup_{\beta\in \Sigma^-_1(v_i^l)} F_\beta(v_i^l),$
we have
\begin{equation*}
\sum_{\xi \in F(v_i^l)} E_i^{l-}(\xi) 
=
\sum_{\beta\in \Sigma^-_1(v_i^l)}\sum_{\xi \in F_\beta (v_i^l)} E_i^{l-}(\xi) 
=
\sum_{\beta\in \Sigma_1^-(v_i^l)} E_i^{l}(\beta), 
\end{equation*}
so that 
$1 = \sum_{i=1}^{m(l)}
\sum_{\xi \in F(v_i^l)} E_i^{l-}(\xi)
$ 
implies
$1 = \sum_{i=1}^{m(l)}\sum_{\beta\in \Sigma^-_1(v_i^l)} E_i^{l}(\beta).
$

(ii) The desired equality is direct from \eqref{eq:L3}.

(iii)
We have
\begin{align*}
  \sum_{\beta\in \Sigma^-_1(v_i^l)} E_i^{l}(\beta) 
= &  \sum_{\xi \in F(v_i^l)} E_i^{l-}(\xi) \\
= & \sum_{\xi \in F(v_i^l)} \sum_{\gamma \in \Sigma^-}
\sum_{j=1}^{m(l+1)} A^-_{l,l+1}(i,\gamma,j) E_j^{l+1-}(\gamma\xi) \\
= & \sum_{j=1}^{m(l+1)} \sum_{\gamma \in \Sigma^-_1(v_j^{l+1}) }
 \sum_{\xi \in F(v_i^l)} A^-_{l,l+1}(i,\gamma,j) E_j^{l+1-}(\gamma\xi). 
\end{align*}
Now 
\begin{equation*}
 \sum_{\xi \in F(v_i^l)} A^-_{l,l+1}(i,\gamma,j) E_j^{l+1-}(\gamma\xi)
= A^-_{l,l+1}(i,\gamma,j)
 \sum_{\zeta \in F_\gamma(v_j^{l+1})}  E_j^{l+1-}(\zeta)
= A^-_{l,l+1}(i,\gamma,j) E_j^{l+1}(\gamma),
\end{equation*}
so that we have 
$$ \sum_{\beta \in \Sigma^-_1(v_i^l)} E_i^{l} (\beta) 
 =  \sum_{j=1}^{m(l+1)}\sum_{\gamma \in \Sigma^-_1(v_j^{l+1})}
A_{l,l+1}^-(i,\gamma,j)E_j^{l+1}(\gamma).
$$
(iv)
By \eqref{eq:L6}, we have
\begin{align*}
S_{\alpha}^*E_i^{l}(\beta) S_{\alpha}   
=  &
\sum_{\xi\in F_\beta(v_i^l)} S_{\alpha}^*E_i^{l-}(\xi) S_{\alpha} \\
=  &
\sum_{\xi\in F_\beta(v_i^l)} 
\left(
\sum_{\gamma\in \Sigma^-} \sum_{j=1}^{m(l+1)}
A_{l,l+1}^+(i,\alpha,j)E_j^{l+1-}(\xi\gamma)
\right) \\
=  & \sum_{j=1}^{m(l+1)} A_{l,l+1}^+(i,\alpha,j)
\sum_{\gamma\in \Sigma^-}
\sum_{\xi\in F_\beta(v_i^l)} 
E_j^{l+1-}(\xi\gamma) \\
= &\sum_{j=1}^{m(l+1)} A_{l,l+1}^+(i,\alpha,j)E_j^{l+1}(\beta).
\end{align*}
\end{proof}
The following lemma is obvious.
\begin{lemma} \label{lem:E12.10}
For $l \in \N$ 
and $\xi = (\xi_1,\dots,\xi_l ) \in F(v_i^l),$
let $e^-_n \in E^-_{n,n-1}, n=1,2,\dots, l$
be a finite sequence of edges satisfying 
\begin{equation*}
\xi_1 = \lambda^-(e^-_l), \, \xi_2 = \lambda^-(e^-_{l-1}), \dots,
\xi_l = \lambda^-(e^-_1), \quad
s(e_l^-) = v_i^l,\, s(e^-_n) = t(e^-_{n+1})
\end{equation*}  
for $ n=1,2,\dots, l-1$
that are figured as 
\begin{equation*}
v_i^l  \overset{\xi_1}{\underset{e^-_l}{\longrightarrow}}
v_{i_{l-1}}^{l-1} \overset{\xi_2}{\underset{e^-_{l-1}}{\longrightarrow}} \cdots
\longrightarrow v_{i_1}^1
\overset{\xi_l}{\underset{e^-_1}{\longrightarrow}}.
\end{equation*}
Put
$v_{i_n}^n =  s(e^-_n) \in V_n$
for
$ n=1,2,\dots, l-1.
$
Then we have 
\begin{equation*}
E_i^{l-}(\xi) =E_i^{l}(\xi_1)E_{i_{l-1}}^{l-1}(\xi_2) \cdots E_{i_1}^{1}(\xi_l).
\end{equation*}
\end{lemma}
Now we reach the following theorem that is one of  the main results of the paper
\begin{theorem}\label{thm:themain6}
Suppose that a $\lambda$-graph bisystem $\LGBS$ 
satisfies $\sigma_{{\frak L}^-}$-condition (I).
Then the $C^*$-algebra $\OALMP$
is  the universal unital unique  $C^*$-algebra
generated by
partial isometries
$S_{\alpha}$ indexed by symbols $\alpha \in \Sigma^+$
and mutually commuting projections
$E_i^{l}(\beta)$ indexed by $\beta \in \Sigma^-_1(v_i^l)$ 
with vertices $v_i^l \in  V_l, l \in \N$  
 subject to the  following operator relations:
\begin{align}
\sum_{\alpha \in \Sigma^+} S_{\alpha}S_{\alpha}^*  
 & = \sum_{i=1}^{m(l)} \sum_{\beta \in \Sigma^-_1(v_i^l)}  E_i^{l}(\beta) = 1, \label{eq:RMP1}\\ 
 S_\alpha S_\alpha^* & E_i^{l}(\beta)   =   E_i^{l}(\beta) S_\alpha S_\alpha^*, 
\label{eq:RMP2}\\
 \sum_{\beta \in \Sigma^-_1(v_i^l)} E_i^{l} (\beta)   
& =  \sum_{j=1}^{m(l+1)}\sum_{\gamma \in \Sigma^-_1(v_j^{l+1})}
A_{l,l+1}^-(i,\gamma,j)E_j^{l+1}(\gamma),  \label{eq:RMP3}\\
S_{\alpha}^*E_i^{l}(\beta) S_{\alpha} &  =  
\sum_{j=1}^{m(l+1)} A_{l,l+1}^+(i,\alpha,j)E_j^{l+1}(\beta) \label{eq:RMP4}.
\end{align}
The above four operator relations are also called the relations $\LGBS.$
\end{theorem}
\begin{proof}
Let $S_\alpha, \alpha \in \Sigma^+$ be partial isometries
and $E_i^l(\beta),  \beta \in \Sigma^-_1(v_i^l)$
be mutually commuting projections satisfying
 the relations 
\eqref{eq:RMP1}, \eqref{eq:RMP2}, \eqref{eq:RMP3} and \eqref{eq:RMP4}.
For $\xi =(\xi_1,\dots,\xi_l) \in F(v_i^l)$, 
let $e_n^- \in E_{n,n-1}^-, n=1,2,\dots,l$
and
$v_{i_n}^n \in V_n, n=1,2,\dots, l-1$ be as in Lemma \ref{lem:E12.10},
so that $v_{i_n}^n  =s(e_n^-)$ for $n=1,2,\dots,l-1.$
Define
\begin{equation}
\widetilde{E}_i^{l-}(\xi) :=E_i^{l}(\xi_1)E_{i_{l-1}}^{l-1}(\xi_2) \cdots E_{i_1}^{1}(\xi_l). \label{eq:projEil}
\end{equation}
Since
$E_i^{l}(\xi_1), \, E_{i_{l-1}}^{l-1}(\xi_2),  \cdots,  E_{i_1}^{1}(\xi_l)
$
 mutually commute, one sees that 
$\widetilde{E}_i^{l-}(\xi)$ is a projection.
We will henceforth show that the projections
$\widetilde{E}_i^{l-}(\xi)$ satisfy 
the equalities \eqref{eq:L4} and \eqref{eq:L6}.
For $\beta \xi \in F(v_j^{l+1})$ with
$A^-_{l,l+1}(i,\beta,j)=1$, we know 
\begin{equation*}
\widetilde{E}_{j}^{l+1-}(\beta \xi) = E_{j}^{l+1}(\beta) \widetilde{E}_{i}^{l-}(\xi). 
\end{equation*}
It then follows that 
\begin{align*}
\sum_{\beta\in \Sigma^-}  \sum_{j=1}^{m(l+1)}
   A_{l,l+1}^-(i,\beta,j) \widetilde{E}_j^{l+1-}(\beta\xi) 
= &\sum_{j=1}^{m(l+1)} \sum_{\gamma\in \Sigma_1^-(v_j^{l+1})}  
A_{l,l+1}^-(i,\gamma,j) E_{j}^{l+1}(\gamma) \widetilde{E}_{i}^{l-}(\xi) \\
= & \sum_{\beta\in \Sigma_1^-(v_i^l)}  
    E_{i}^{l}(\beta)\cdot  \widetilde{E}_{i}^{l-}(\xi)  
= \widetilde{E}_i^{l-}(\xi),
\end{align*}
because of the equality \eqref{eq:RMP3}.
Hence we obtain the equality \eqref{eq:L4}.
By using the preceding lemma, we have
\begin{align*}
& S_\alpha^* \widetilde{E}_i^{l-}(\xi) S_\alpha \\
=
& S_\alpha^* E_i^{l}(\xi_1)E_{i_{l-1}}^{l-1}(\xi_2) \cdots E_{i_1}^{1}(\xi_l) S_\alpha \\
=
& S_\alpha^* E_i^{l}(\xi_1) S_\alpha\cdot 
  S_\alpha^* E_{i_{l-1}}^{l-1}(\xi_2) S_\alpha \cdots  
  S_\alpha^* E_{i_1}^{1}(\xi_l) S_\alpha \\
=
& \left(
\sum_{j=1}^{m(l+1)} A_{l,l+1}^+(i,\alpha,j)E_j^{l+1}(\xi_1)
\right)
\cdot  
\left(
\sum_{j_l=1}^{m(l)} A_{l-1,l}^+(i_{l-1},\alpha,j_l)E_{j_l}^{l}(\xi_2)
\right)
\cdots \\
& \qquad\qquad\qquad \qquad\qquad\qquad \quad \cdots \left(
\sum_{j_2=1}^{m(2)} A_{1,2}^+(i_1,\alpha,j_2)E_{j_2}^{2}(\xi_l)
\right).
\end{align*}
Let $e^-_l \in E^-_{l+1,l}$ be the unique edge such that 
$\xi_1 = \lambda^-(e^-_l ), s(e^-_l ) = v_j^{l+1}.$
By \eqref{eq:RMP1} and \eqref{eq:RMP2}, we know that 
$E_j^{l+1}(\xi_1)\cdot E_{j_l}^l(\xi_2) \ne 0$ if and only if 
$v_{j_l}^l = t(e^-_l )$, and in this case
$A^+_{l-1,l}(i_{l-1},\alpha,j_l) =1.$
Hence we have 
\begin{align*}
& \left(
\sum_{j=1}^{m(l+1)} A_{l,l+1}^+(i,\alpha,j)E_j^{l+1}(\xi_1)
\right)
\cdot  
\left(
\sum_{j_l=1}^{m(l)} A_{l-1,l}^+(i_{l-1},\alpha,j_l)E_{j_l}^{l}(\xi_2)
\right) \\
= & \sum_{j=1}^{m(l+1)} A_{l,l+1}^+(i,\alpha,j)E_j^{l+1}(\xi_1)\cdot E_{j_l}^l(\xi_2).
 \end{align*}
We inductively know that 
\begin{align*}
& \left(
\sum_{j=1}^{m(l+1)} A_{l,l+1}^+(i,\alpha,j)E_j^{l+1}(\xi_1)
\right)
\cdot  \left(
\sum_{j_l=1}^{m(l)} A_{l-1,l}^+(i_{l-1},\alpha,j_l)E_{j_l}^{l}(\xi_2)
\right)
\cdots \\
& \qquad\qquad\qquad \qquad\qquad\qquad \quad\cdots \left(
\sum_{j_2=1}^{m(2)} A_{1,2}^+(i_1,\alpha,j_2)E_{j_2}^{2}(\xi_l)
\right) \\
&
=
\sum_{j=1}^{m(l+1)} A_{l,l+1}^+(i,\alpha,j)E_j^{l+1}(\xi_1)\cdot E_{j_l}^l(\xi_2)
 \cdots  E_{j_2}^2(\xi_l).
\end{align*}
As 
$
E_j^{l+1}(\xi_1)\cdot E_{j_l}^l(\xi_2)
 \cdots  E_{j_2}^2(\xi_l)= \sum_{\beta\in \Sigma^-}\widetilde{E}_j^{l+1-}(\xi \beta),
$
by \eqref{eq:RMP1},
we conclude that
$$
S_\alpha^* \widetilde{E}_i^{l-}(\xi) S_\alpha 
=
\sum_{j=1}^{m(l+1)} A_{l,l+1}^+(i,\alpha,j)\widetilde{E}_j^{l+1-}(\xi \beta).
$$
Since it is direct to see that the equality
$$
S_\alpha S_\alpha^*  \widetilde{E}_i^{l}(\xi) = \widetilde{E}_i^{l}(\xi) S_\alpha S_\alpha^*
$$
holds,
the family 
$S_\alpha, \alpha\in \Sigma^+,
\widetilde{E}^{l-}_i(\xi), \xi \in F(v_i^l)
$
of operators satisfy the relations $\LGBS$ of Theorem \ref{thm:main6}.
By the universal property and the uniqueness of the relations $\LGBS$ in Theorem \ref{thm:main6},
the family 
$S_\alpha, \alpha\in \Sigma^+,
E^{l-}_i(\beta), \beta \in \Sigma_1^-(v_i^l)
$
of operators satisfying
 the relations 
\eqref{eq:RMP1}, \eqref{eq:RMP2}, \eqref{eq:RMP3} and \eqref{eq:RMP4}
completely determine the algebraic structure of the $C^*$-algebra
$\OALMP$.
\end{proof}
\begin{remark}
If a $\lambda$-graph bisystem
$\LGBS$ comes from a $\lambda$-graph system ${\frak L}$
as in Example \ref{ex:3.2}(i), then the set 
$\Sigma_1^-(v_i^l)$ is a singleton $\{\iota \}$
for every vertex $v_i^l \in V.$
Hence the operator relations $\LGBS$ of Theorem \ref{thm:themain6}
coincides with the operator relations of Theorem \ref{thm:lambdagraphC*}.
\end{remark}

Similarly we have 
\begin{theorem}\label{thm:themain7'}
Suppose that a $\lambda$-graph bisystem $\LGBS$ satisfies 
$\sigma_{{\frak L}^+}$-condition (I).
The $C^*$-algebra $\OALPM$
is realized as the universal unital unique  $C^*$-algebra
generated by partial isometries
$T_{\beta}$ indexed by symbols $\beta \in \Sigma^-$ 
and mutually commuting projections
$F_i^{l}(\alpha)$ indexed by vertices $v_i^l \in  V_l$ and 
symbols $\alpha \in \Sigma_1^+(v_i^l)$ 
 subject to the  following operator relations:
\begin{align}
\sum_{\beta \in \Sigma^-} T_{\beta}T_{\beta}^*   
& =  \sum_{i=1}^{m(l)} \sum_{\alpha \in \Sigma^+_1(v_i^l)}
        F_i^{l}(\alpha) =  1, \label{eq:RPM1}\\ 
 T_\beta T_\beta^* & F_i^{l}(\alpha)   =   F_i^{l}(\alpha) T_\beta T_\beta^*, 
\label{eq:RPM2}\\
\sum_{\alpha \in \Sigma^+_1(v_i^l)} F_i^{l} (\alpha)   
& =   \sum_{j=1}^{m(l+1)} \sum_{\delta\in \Sigma^+_1(v_j^{l+1})}
A_{l,l+1}^+(i,\delta,j)F_j^{l+1}(\delta),  \label{eq:RPM3}\\
            T_{\beta}^*F_i^{l}(\alpha) T_{\beta} 
&  =  \sum_{j=1}^{m(l+1)} A_{l,l+1}^-(i,\beta,j)F_j^{l+1}(\alpha), \label{eq:RPM4}
\end{align}
where $\Sigma^+_1(v_i^l) = \{
\lambda^+(e^+) \in \Sigma^+ \mid e^+ \in E_{l-1,l}^+, t(e^+) = v_i^l
\}.
$
\end{theorem}
The  above operator relations are  called $({\frak L}^+,{\frak L}^-).$

For a $\lambda$-graph bisystem $\LGBS$, 
denote by 
${{\frak L}^{-t}}$ (resp. ${{\frak L}^{+t}}$) the labeled Bratteli diagram obtained by reversing the directions of all edges in ${{\frak L}^-}$ (resp. ${{\frak L}^+}$). 
Then the pair
$({{\frak L}^{+t}}, {{\frak L}^{-t}})$ becomes a $\lambda$-graph bisystem.
Since the $C^*$-algebras 
$\OALMP$ and $\OALPM$
are both universal $C^*$-algebras subject to the operator relations 
$\LGBS$ and $({\frak L}^+,{\frak L}^-),$
respectively,
we have canonical isomorphisms of $C^*$-algebras:
\begin{equation*}
\OALMP \cong {{\mathcal{O}}_{{{\frak L}^{-t}}}^-}, \qquad
\OALPM \cong {{\mathcal{O}}_{{{\frak L}^{+t}}}^+}. 
\end{equation*}

\section{K-groups for $\OALMP$}

In this section we will describe $K$-theory formulas
for  our $C^*$-algebras 
$\OALMP$ as well as  $\OALPM$.
We will then prove that the K-groups are invariant under
strong shift equivalence of the associated symbolic matrix bisystems.

By Theorem \ref{thm:themain6} and  Theorem \ref{thm:themain7'} 
(or Theorem \ref{thm:main6} and Theorem \ref{thm:main7'}),
we know that the $C^*$-algebras $\OALMP$ and  $\OALPM$
are nothing but the $C^*$-symbolic crossed products
$\A_{{\frak L}^-}\rtimes_{\rho^-}\Lambda_{{\frak L}^+}$
and 
$\A_{{\frak L}^+}\rtimes_{\rho^+}\Lambda_{{\frak L}^-}$
defined in \cite{MaCrelle},
respectively.
K-theory formulas for the $C^*$-algebra
$\A\rtimes_\rho\Lambda$  
 constructed from 
a $C^*$-dynamical system $(\A,\rho,\Sigma)$
in general have been presented in 
\cite{MaCrelle}.
We may apply the formulas to our $C^*$-algebras
$\OALMP$ and $\OALPM$.
We will focus on the former algebra  $\OALMP$,
the latter one is symmetric.

The endomorphisms
$\rho_\alpha^+:  
\A_{{\frak L}^-}\longrightarrow \A_{{\frak L}^-}$ for $\alpha \in \Sigma^+$
defined in \eqref{eq:defofrhoalphaA}
yield endomorphisms 
$ {\rho_{\alpha}^+}_{*}:K_*(\A_{{\frak L}^-})\rightarrow K_*(\A_{{\frak L}^-})$
for $\alpha \in \Sigma^+$
on the K-theory groups of $\A_{{\frak L}^-}$.
Define an endomorphism
$$
\rho_{*}^+:K_*(\A_{{\frak L}^-})\longrightarrow 
K_*(\A_{{\frak L}^-}), \qquad * = 0,1
$$
by setting
$\rho_{*}^+(g) = \sum_{\alpha \in \Sigma^+} {\rho_{\alpha}^+}_{*}(g),
g \in K_*(\A_{{\frak L}^-}).
$
\begin{lemma}\label{lem:Ktheory1}
\begin{align*}
K_0(\OALMP)  \, 
& \cong \, K_0(\A_{{\frak L}^-}) / (\id - {\rho}_{*}^+) K_0(\A_{{\frak L}^-}), \\
K_1(\OALMP) \,  
& \cong \, \Ker(\id - {\rho}_{*}^+) \text{ in } K_0(\A_{{\frak L}^-}). \\
\end{align*}
\end{lemma}
\begin{proof}
Since our $C^*$-algebra $\OALMP$ 
is isomorphic to 
the $C^*$-symbolic crossed product
$\A_{{\frak L}^-}\rtimes_{\rho^+}\Lambda_{{\frak L}^+}$,
 one has the six term
exact sequence of K-theory (\cite{MaCrelle},  cf. \cite{Pim}, \cite{KPW}):
\begin{equation*}
\begin{CD}
 K_0(\A_{{\frak L}^-}) @>\id - {\rho}_{*}^+>> 
 K_0(\A_{{\frak L}^-}) @>\iota_{*}>> 
 K_0(\OALMP) \\
@AAA @. @VVV \\
 K_1(\OALMP)
@<<\iota_{*}<
 K_1(\A_{{\frak L}^-}) @<<\id - {\rho}_{*}^+< 
 K_1(\A_{{\frak L}^-}). \\ 
\end{CD}
\end{equation*}
As $\A_{{\frak L}^-}$ is an AF-algebra, 
one sees that $K_1(\A_{{\frak L}^-}) = 0$, 
so that we have the desired formulas.
\end{proof}

The $K_0$-group
$K_0(C(\Omega_{{\frak L}^-}))$
of the commutative $C^*$-algebra
$C(\Omega_{{\frak L}^-})$
is canonically isomorphic to
the  abelian group
$C(\Omega_{{\frak L}^-},\Z)$ 
of $\Z$-valued continuous functions on $\Omega_{{\frak L}^-}.$
The correspondence
$\varphi_*([E_i^{l-}(\xi)]) = \chi_{U_{\Omega_{{\frak L}^-}(v_i^l;\xi)}}
$ induced by
\eqref{eq:ALMC} yields 
a natural  isomorphism 
$$
\varphi_*: K_0(\A_{{\frak L}^-}) \longrightarrow  C(\Omega_{{\frak L}^-},\Z)
$$ 
between
$K_0(\A_{{\frak L}^-})$
and the  abelian group
$C(\Omega_{{\frak L}^-},\Z).$
For $\omega \in \Omega_{{\frak L}^-},$
put
$$
r(\omega)
=\{ \omega' \in \Omega_{{\frak L}^-} \mid (\omega, \alpha^+,\omega') \in E_{{\frak L}^-}^+ \text{ for some } \alpha^+ \in \Sigma^+\}.
$$
We then define an endomorphism on $C(\Omega_{{\frak L}^-},\Z)$
by setting
$$
\lambda_{{\frak L}^-*}^+(f)(\omega) 
= \sum_{\omega' \in r(\omega)}
f(\omega')
\quad
\text{ for }
f \in C(\Omega_{{\frak L}^-},\Z), \, 
\omega \in \Omega_{{\frak L}^-}.
$$
We thus have the following K-theory formulas for  $\OALMP$.
\begin{theorem}\label{thm:Ktheory}
\begin{align*}
   K_0(\OALMP) 
\, \cong & \, 
C(\Omega_{{\frak L}^-}, \Z) / 
(\id -\lambda_{{\frak L}^-*}^+) C(\Omega_{{\frak L}^-}, \Z), \\
  K_1(\OALMP) 
\, \cong & \,
\Ker(\id -\lambda_{{\frak L}^-*}^+) \text{ in } C(\Omega_{{\frak L}^-}, \Z).
\end{align*}
\end{theorem}
\begin{proof}
It is easy to see that the diagram
\begin{equation*}
\begin{CD}
 K_0(\A_{{\frak L}^-}) @>\rho_*^+>> 
 K_0(\A_{{\frak L}^-})  \\
@V{\varphi_*}VV  @V{\varphi_*}VV \\
 C(\Omega_{{\frak L}^-},\Z)
@>\lambda_{{\frak L}^-*}^+>> 
 C(\Omega_{{\frak L}^-},\Z) \\ 
\end{CD}
\end{equation*}
commutes.
Hence by Lemma \ref{lem:Ktheory1},
we get the desired K-theory formulas.
\end{proof}
The above K-theory formulas are generalizations of those of the 
$C^*$-algebras ${\mathcal{O}}_{\frak L}$ associated with $\lambda$-graph system
${\frak L}$ in \cite[Theorem 5.5]{MaDocMath2002} (cf. \cite{C2})
and the crossed products $C(\Lambda_A)\rtimes_{\sigma_A^*}\Z$
 of the commutative $C^*$-algebra $C(\Lambda_A)$
on the two-sided topological Markov shifts $\Lambda_A$
by the automorphisms $\sigma_A^*$ induced by the homeomorphism of the shift 
$\sigma_A^*$ (cf. \cite{Poon}). 

\medskip

We will next prove that the groups 
$K_i(\OALMP), i=0,1$ are 
invariant under properly strong shift equivalence of the associated symbolic matrix bisystems of the $\lambda$-graph bisystems $\LGBS$ satisfying FPCC.
To prove it we will actually show the following theorem, that was improved by the referee's kind suggestion.
Let $\mathcal{K}$ denote the $C^*$-algebra of compact operators
on a separable infinite dimensional Hilbert space,
and  $\mathcal{C}$ denote its diagonal $C^*$-subalgebra. 
\begin{theorem}\label{thm:MoritaKtheory}
Let $(\M^-,\M^+)$ and $(\SN^-,\SN^+)$
be symbolic matrix bisystems.
Let 
$({\frak L}_\M^-,{\frak L}_\M^+)$ and 
$({\frak L}_\SN^-,{\frak L}_\SN^+)$
be the associated $\lambda$-graph bisystems both of which satisfy FPCC.
Suppose that 
$(\M^-, \M^+)$ and 
$(\SN^-, \SN^+)$
are properly strong shift equivalent.
Then 
there exists an isomorphism
$\varPhi:
{{\mathcal{O}}_{{\frak L}_\M^-}^+} \otimes \mathcal{K}
\longrightarrow 
{{\mathcal{O}}_{{\frak L}_\SN^-}^+}\otimes \mathcal{K}$
of $C^*$-algebras such that 
$\varPhi({{\mathcal{D}}_{{\frak L}_\M^-}^+} \otimes \mathcal{C})
= 
{{\mathcal{D}}_{{\frak L}_\SN^-}^+}\otimes \mathcal{C}$.
In particular, the $C^*$-algebras
${{\mathcal{O}}_{{\frak L}_\M^-}^+}$ and 
${{\mathcal{O}}_{{\frak L}_\SN^-}^+}$
are Morita equivalent, so that their K-groups
$K_i({{\mathcal{O}}_{{\frak L}_\M^-}^+})$ and 
$K_i({{\mathcal{O}}_{{\frak L}_\SN^-}^+})$
are isomorphic for $i=0,1.$
\end{theorem}
\begin{proof}
We may assume that 
$({\frak L}_\M^-,{\frak L}_\M^+)$ and 
$({\frak L}_\SN^-,{\frak L}_\SN^+)$
are properly strong shift equivalent in $1$-step.
As in the discussion in Section 6,
there exist an alphabet 
$\widehat{\Sigma},$ disjoint subsets
$C, D \subset \widehat{\Sigma}$
and  a bipartite symbolic matrix bisystem
 $(\widehat{\M}^-,\widehat{\M}^+) 
$
over 
$\widehat{\Sigma}$
such that 
\begin{equation*}
(\widehat{\M}^{CD-}, \widehat{\M}^{CD+}) = (\M^-,\M^+), \qquad
(\widehat{\M}^{DC-}, \widehat{\M}^{DC+}) = (\SN^-,\SN^+).
\end{equation*} 
Let $(\widehat{\frak L}^-,\widehat{\frak L}^+)$
be the associated $\lambda$-graph bisystems to
$(\widehat{\M}^-,\widehat{\M}^+) .$
We also denote by
$(\widehat{\frak L}^{CD-},\widehat{\frak L}^{CD+})$
(resp.
$(\widehat{\frak L}^{DC-},\widehat{\frak L}^{DC+})$
)
 the associated $\lambda$-graph bisystems to
$(\widehat{\M}^{CD-},\widehat{\M}^{CD+}) $
(resp. 
$(\widehat{\M}^{DC-},\widehat{\M}^{DC+}) ). $
By a completely similar argument to \cite[Theorem 6.1]{MaCrelle},
we know that there exist full projections
$P_C, P_D$ in the $C^*$-algebra
${\mathcal{O}}_{\widehat{\frak L}^-}^+$  
such that both
$P_C, P_D$ belong to ${\mathcal{D}}_{\widehat{\frak L}^-}^+$
satisfying 
 $P_C + P_D =1$ and
\begin{gather*}
P_C {\mathcal{O}}_{\widehat{\frak L}^-}^+ P_C\cong 
{\mathcal{O}}_{\widehat{\frak L}^{CD-}}^+, \qquad
P_C {\mathcal{D}}_{\widehat{\frak L}^-}^+ P_C\cong 
{\mathcal{D}}_{\widehat{\frak L}^{CD-}}^+, \\
P_D {\mathcal{O}}_{\widehat{\frak L}^-}^+ P_D\cong 
{\mathcal{O}}_{\widehat{\frak L}^{DC-}}^+, \qquad
P_D {\mathcal{D}}_{\widehat{\frak L}^-}^+ P_D\cong 
{\mathcal{D}}_{\widehat{\frak L}^{DC-}}^+.
\end{gather*}
Hence  the pair  
$({\mathcal{O}}_{\widehat{\frak L}^{CD-}}^+, 
{\mathcal{D}}_{\widehat{\frak L}^{CD-}}^+)
$
and 
$({\mathcal{O}}_{\widehat{\frak L}^{DC-}}^+, 
{\mathcal{D}}_{\widehat{\frak L}^{DC-}}^+)
$
is relative Morita equivalent in the sense of 
\cite{MaTransAMS2018}.
By using \cite[Theorem 4.7]{MaTransAMS2018},
there exists an isomorphism
$\Phi:{\mathcal{O}}_{\widehat{\frak L}^{CD-}}^+\otimes\mathcal{K} 
\longrightarrow
{\mathcal{O}}_{\widehat{\frak L}^{DC-}}^+\otimes\mathcal{K}
$
such that 
$\Phi({\mathcal{D}}_{\widehat{\frak L}^{CD-}}^+\otimes\mathcal{C})
=
{\mathcal{D}}_{\widehat{\frak L}^{DC-}}^+\otimes\mathcal{C}.
$
Since
$({\mathcal{O}}_{\widehat{\frak L}^{CD-}}^+, {\mathcal{D}}_{\widehat{\frak L}^{CD-}}^+)
 = 
({\mathcal{O}}_{\widehat{\frak L}_\M^{-}}^+, {\mathcal{D}}_{\widehat{\frak L}_\M^{-}}^+)$
and
$({\mathcal{O}}_{\widehat{\frak L}^{DC-}}^+, {\mathcal{D}}_{\widehat{\frak L}^{DC-}}^+)
 = 
({\mathcal{O}}_{\widehat{\frak L}_\SN^{-}}^+,{\mathcal{D}}_{\widehat{\frak L}_\SN^{-}}^+),$
we conclude that
there exists an isomorphism
$\varPhi:
{{\mathcal{O}}_{{\frak L}_\M^-}^+} \otimes \mathcal{K}
\longrightarrow 
{{\mathcal{O}}_{{\frak L}_\SN^-}^+}\otimes \mathcal{K}$
 such that 
$\varPhi({{\mathcal{D}}_{{\frak L}_\M^-}^+} \otimes \mathcal{C})
= 
{{\mathcal{D}}_{{\frak L}_\SN^-}^+}\otimes \mathcal{C}$.
\end{proof}
\begin{corollary}\label{cor:13.4}
The K-groups $K_i({\mathcal{O}}_{{\frak L}_\Lambda^-}^+), i=0,1$
of the $C^*$-algebra
${\mathcal{O}}_{{\frak L}_\Lambda^-}^+$
of the canonical $\lambda$-graph bisystem
$({\frak L}_\Lambda^-, {\frak L}_\Lambda^+)$
of a subshift $\Lambda$ is invariant under topological conjugacy of subshifts. 
\end{corollary}


\section{A duality: $\lambda$-graph systems as $\lambda$-graph bisystems}

Let ${\frak L} =(V,E,\lambda,\iota)$ be a $\lambda$-graph system over $\Sigma.$
We will construct a $\lambda$-graph bisystem $\LGBS$ from ${\frak L}$ 
as in Example \ref{ex:3.2} (i).
Let us recognize the map $\iota: V\longrightarrow V$
as a new symbol written $\iota$, and
define a new alphabet
$\Sigma^- =  \{\iota\}.$
The original alphabet $\Sigma$ is written $\Sigma^+.$ 
Let 
$
E^+_{l,l+1}:= E_{l,l+1} 
$ for $l \in \Zp$ 
and
$ 
\lambda^+= \lambda:E^+ \longrightarrow \Sigma^+.
$  
We then have a labeled Bratteli diagram
${\frak L}^+ = (V, E^+, \lambda^+)$ over alphabet $\Sigma^+,$
that is 
the original labeled Bratteli diagram $\frak L$ without the map 
$\iota: V \longrightarrow V.$ 
The other 
labeled Bratteli diagram
${\frak L}^- = (V, E^-, \lambda^-)$ over alphabet $\Sigma^-$
is defined in the following way.
Define an edge $e^-\in E^-_{l+1,l}$
if
$\iota(v_j^{l+1}) = v_i^l$ 
so that 
$s(e^-) = v_j^{l+1}, t(e^-) = v_i^l$
and 
$\lambda^-(e^-) =\iota \in \Sigma^-.$ 
Then we have a labeled Bratteli diagram
${\frak L}^- = (V, E^-, \lambda^-)$ over alphabet $\Sigma^-.$
Then the local property of the $\lambda$-graph system  ${\frak L}$ makes the pair 
$\LGBS$ a $\lambda$-graph bisystem.
This $\lambda$-graph bisystem does not satisfy FPCC.
Let $(I_{l,l+1}, A_{l,l+1})_{l \in \Zp}$
be the transition matrix system for the $\lambda$-graph system
${\frak L}$ defined in \eqref{eq:tmsI} and \eqref{eq:tmsA}. 
The transition matrix bisystem $(A^-,A^+)$ for 
the $\lambda$-graph bisystem $\LGBS$ defined in
Section 8 satisfies 
\begin{equation}
A_{l,l+1}^-(i,\iota,j)  = I_{l,l+1}(i,j), 
\qquad
A_{l,l+1}^+(i,\alpha,j)  = A_{l,l+1}(i,\alpha,j)  \label{eq:15.lambda}
\end{equation}
for  $ \alpha \in \Sigma^+$ and $i=1,2,\dots,m(l), \, j=1,2,\dots,m(l+1).$

Let $S_\alpha, \alpha \in \Sigma^+$
and
$E_i^l(\beta), \beta \in \Sigma^-_1(v_i^l)$
be the canonical generating family of the $C^*$-algebra
$\OALMP$ satisfying the relations $\LGBS$ in Theorem \ref{thm:themain6}.
Since $\Sigma_1^-(v_i^l) = \{\iota\}$ for all vertices $v_i^l \in V_l$,
the projection 
$E_i^l(\beta), \beta \in \Sigma^-_1(v_i^l)$  
may be written $E_i^l$ without the symbol $\beta.$
The equalities \eqref{eq:15.lambda} tell us that the relations 
$\LGBS$ in Theorem \ref{thm:themain6} is exactly the same as the  relations
$({\frak L})$ in Theorem \ref{thm:lambdagraphC*}.
Hence the $C^*$-algebra
$\OALMP$ coincides with the $C^*$-algebra ${\mathcal{O}}_{\frak L}$
of the $\lambda$-graph system ${\frak L}$ by their universal properties.


Let us consider the other $C^*$-algebra $\OALPM.$
Then  by Theorem \ref{thm:main7'} together with \eqref{eq:15.lambda}, 
the $C^*$-algebra is generated by one coisometry
$T_\iota$ and a family of mutually commuting projections
$E_i^{l+}(\mu) , \mu \in P(v_i^l), i=1,2,\dots,m(l),\, l\in \Zp$  
satisfying the following relations:
\begin{align}
T_{\iota}T_{\iota}^*   
& =   \sum_{i=1}^{m(l)} \sum_{\mu \in P(v_i^l)} E_i^{l+}(\mu) =  1, \label{eq:Plambda2}\\ 
 E_i^{l+} (\mu)   & = \sum_{\alpha\in \Sigma}  \sum_{j=1}^{m(l+1)}
A_{l,l+1}(i,\alpha,j)E_j^{l+1+}(\mu\alpha),  \label{eq:Plambda4}\\
T_{\iota}^*E_i^{l+}(\mu) T_{\iota} &  =  
\sum_{\alpha\in \Sigma}
\sum_{j=1}^{m(l+1)} I_{l,l+1}(i,j)E_j^{l+1+}(\alpha\mu). \label{eq:Plambda6}
\end{align}
By \eqref{eq:Plambda6} with \eqref{eq:Plambda2}, we have
\begin{equation*}
T_\iota^* T_\iota = 
\sum_{i=1}^{m(l)}\sum_{\mu \in P(v_i^l)}T_{\iota}^*E_i^{l+}(\mu) T_{\iota}
 =  
\sum_{j=1}^{m(l+1)}
\sum_{i=1}^{m(l)}
I_{l,l+1}(i,j)
\sum_{\alpha\in \Sigma}
\sum_{\mu \in P(v_i^l)}
 E_j^{l+1+}(\alpha\mu) 
\end{equation*}
Under the condition 
$I_{l,l+1}(i,j) =1$, the local property of $\lambda$-graph bisystem
ensures us 
$$
\{\alpha \mu \in B_{l+1}(\Lambda_{{\frak L}^+}) 
\mid \mu \in P(v_i^l), \alpha \in \Sigma \}
= P(v_j^{l+1}),
$$
so that the equality
$$
I_{l,l+1}(i,j)
\sum_{\alpha\in \Sigma}
\sum_{\mu \in P(v_i^l)}
 E_j^{l+1+}(\alpha\mu) 
=
I_{l,l+1}(i,j)
\sum_{\nu \in P(v_j^{l+1})}
E_j^{l+1+}(\nu) 
$$
holds.
As for each $j=1, 2, \dots,m(l+1),$
there uniquely exists $i=1, 2, \dots, m(l)$ such that 
$I_{l,l+1}(i,j) =1,$ 
so that $\sum_{i=1}^{m(l)} I_{l,l+1}(i,j) =1.$
Hence we have
\begin{equation*}
\sum_{j=1}^{m(l+1)}
\sum_{i=1}^{m(l)} I_{l,l+1}(i,j)
\sum_{\alpha\in \Sigma}
\sum_{\mu \in P(v_i^l)}
E_j^{l+1+}(\alpha\mu)
=
\sum_{j=1}^{m(l+1)}
\sum_{\nu \in P(v_j^{l+1})}
E_j^{l+1+}(\nu) =1,
\end{equation*}
so that $T_\iota$ is a unitary.

Recall that 
the $C^*$-subalgebra
$\DLPM$
of $\OALPM$
is generated by the projections of the form
\begin{equation*}
T_\iota^n E_i^{l+}(\mu) T_\iota^{*n}, \qquad 
\mu \in P(v_i^l), \, n=0, 1,2,\dots, \, i=1,2,\dots, m(l), \, l\in \Zp.
\end{equation*}
It is canonically isomorphic to the commutative $C^*$-algebra
$C(X_{{\frak L}^+}^-)$ of continuous functions on $X_{{\frak L}^+}^-$
through the correspondence
\begin{equation*}
\varphi_{{\frak L}^+}: 
T_\iota^n E_i^{l+}(\mu) T_\iota^{*n}
\in \DLPM\longrightarrow 
\chi_{U_{X_{{\frak L}^+}^-}(\iota^n, v_i^l;\mu)}
\in C(X_{{\frak L}^+}^-)
\end{equation*}
as in Lemma \ref{lem:smeilsm} for 
$X_{{\frak L}^+}^-$ 
(Note that Lemma \ref{lem:smeilsm} treats 
$X_{{\frak L}^-}^+$).
Hence we know the following lemma.
\begin{lemma}
The isomorphism
$\varphi_{{\frak L}^+}: \DLPM\longrightarrow C(X_{{\frak L}^+}^-)$
satisfies
$\varphi_{{\frak L}^+}\circ \Ad(T_\iota^*) 
= \sigma_{{\frak L}^+}^* \circ\varphi_{{\frak L}^+}
$
where $\sigma_{{\frak L}^+}^*: C(X_{{\frak L}^+}^-) \longrightarrow C(X_{{\frak L}^+}^-)$
is defined by
$\sigma_{{\frak L}^+}^*(f) = f \circ \sigma_{{\frak L}^+}$ 
for $f \in C(X_{{\frak L}^+}^-).$
\end{lemma}
\begin{proof}
The $C^*$-subalgebra 
$\ALP$ of $\OALPM$ generated by mutually commuting 
projections $E_i^{l+}(\mu), \mu \in P(v_i^l)$
 is isomorphic to the commutative $C^*$-algebra
$C(\Omega_{{\frak L}^+})$ of continuous functions on the compact Hausdorff space
$\Omega_{{\frak L}^+}$
defined in Section 7.
Now the $\lambda$-graph bisystem 
$({\frak L}^-, {\frak L}^+)$
comes from a $\lambda$-graph system ${\frak L}.$
Hence the compact Hausdorff space
$X_{{\frak L}^+}^-$
is given by in this case
\begin{align*}
X_{{\frak L}^+}^- =
& \{ (\iota, \omega^i)_{i=1}^\infty \in 
\prod_{i=1}^\infty (\Sigma^- \times \Omega_{{\frak L}^+}) \mid
\omega^i = (\alpha_{-i+l}, u_l^i)_{l=1}^\infty \in \Omega_{{\frak L}^+}, i=1,2,\dots, \\
& \qquad (\omega^i, \iota, \omega^{i+1}) \in E_{{\frak L}^+} ^-, i=1,2,\dots,
(\omega^0, \iota, \omega^1) \in E_{{\frak L}^+} ^-\text{ for some } 
\omega^0 \in \Omega_{{\frak L}^+} \}.
\end{align*}
Let $\sigma_{{\frak L}^+}: X_{{\frak L}^+}^-\longrightarrow X_{{\frak L}^+}^-$
be the shift map defined by
$ \sigma_{{\frak L}^+}((\iota, \omega^i)_{i=1}^\infty) = (\iota, \omega^{i+1})_{i=1}^\infty.$
As
$\omega^i = (\alpha_{-i+l}, u_l^i)_{l=1}^\infty
\in \Omega_{{\frak L}^+}, i=0,1, 2,\dots
$ 
and   
$\iota(u^{i+1}_{l+1}) =u^{i}_{l}, l=0,1,\dots, \, i=1,2,\dots,$
we know that 
$ (\iota, \omega^{i+1})_{i=1}^\infty$ uniquely determines 
$(\iota, \omega^i)_{i=1}^\infty$ as in the diagram below,
so that 
the shift map
 $\sigma_{{\frak L}^+}: X_{{\frak L}^+}^-\longrightarrow X_{{\frak L}^+}^-$
is actually a homeomorphism on $X_{{\frak L}^+}^-$.
\begin{equation*}
\begin{CD}
@. @. @. u_0^0@>{\alpha_1}>> u_1^0 @>{\alpha_2}>> u_2^0 @>{\alpha_3}>> \cdots
 \omega^0 \\
@. @. @. @AA{\iota}A @AA{\iota}A @AA{\iota}A @.\\
@. @. u_0^1 @>{\alpha_0}>> u^1_1 @>{\alpha_1}>> u^1_2 @>{\alpha_2}>> u^1_3
@>{\alpha_3}>> \cdots \omega^1\ \\
@. @. @AA{\iota}A  @AA{\iota}A  @AA{\iota}A @AA{\iota}A @.\\
@.u_0^2 @>{\alpha_{-1}}>> u_1^2 @>{\alpha_0}>> u^2_2 @>{\alpha_1}>> u^2_3 
@>{\alpha_2}>> u^2_4
@>{\alpha_3}>> \cdots \omega^2\ \\
 @. @AA{\iota}A  @AA{\iota}A  @AA{\iota}A  @AA{\iota}A @AA{\iota}A @.\\
u_0^3 @>{\alpha_{-2}}>>u_1^3 @>{\alpha_{-1}}>> u_2^3 @>{\alpha_0}>> u^3_3 @>{\alpha_1}>> u^3_4 
@>{\alpha_2}>> u^3_5
@>{\alpha_3}>> \cdots \omega^3\ \\
 @AA{\iota}A @AA{\iota}A  @AA{\iota}A  @AA{\iota}A  @AA{\iota}A @AA{\iota}A @.\\
\end{CD}
\end{equation*}
As 
$$
\sigma_{{\frak L}^+}^*(\chi_{U_{X_{{\frak L}^+}^-}(\iota^n,v_i^l;\mu)})
=\chi_{U_{X_{{\frak L}^+}^-}(\iota^{n-1},v_i^l;\mu)}
\quad\text{ for }
n \ge 1
$$
and
$$
\sigma_{{\frak L}^+}^*(\chi_{U_{X_{{\frak L}^+}^-}(v_i^l;\mu)})
=\sum_{\alpha\in \Sigma}\sum_{j=1}^{m(l+1)}
 \chi_{U_{X_{{\frak L}^+}^-}(\iota,v_j^{l+1};\alpha\mu)},
$$
we conclude that 
$\varphi_{{\frak L}^+}\circ \Ad(T_\iota^*) 
= \sigma_{{\frak L}^+}^* \circ\varphi_{{\frak L}^+}
$ because of the relation \eqref{eq:Plambda6}.
\end{proof}
By their universal properties of  both  the algebra 
$\OALPM$ and the crossed product
$ C(X_{{\frak L}^+}^-)\rtimes_{\sigma_{{\frak L}^+}^*}\Z$
by the automorphism $\sigma_{{\frak L}^+}^*$
of $C(X_{{\frak L}^+}^-)$,
we know that 
$\OALPM$ is isomorphic to 
$ C(X_{{\frak L}^+}^-)\rtimes_{\sigma_{{\frak L}^+}^*}\Z,$
Thus we have 
the following proposition.
\begin{proposition}\label{prop:lambdaalgebras}
Let $\frak L$ be a left-resolving $\lambda$-graph system over $\Sigma.$
Let $\LGBS$ be the associated $\lambda$-graph bisystem.
Then we have 
\begin{enumerate}
\renewcommand{\theenumi}{\roman{enumi}}
\renewcommand{\labelenumi}{\textup{(\theenumi)}}
\item
The $C^*$-algebra $\OALMP$ is canonically isomorphic to the 
the $C^*$-algebra ${\mathcal{O}}_{\frak L}$ 
of the original $\lambda$-graph system  ${\frak L}.$ 
\item
The $C^*$-algebra $\OALPM$ is canonically isomorphic to  the crossed product
$ C(X_{{\frak L}^+}^-)\rtimes_{\sigma_{{\frak L}^+}^*}\Z.$ 
\end{enumerate}
\end{proposition}
Let $A$ be an $N\times N$ irreducible non-permutation matrix over $\{0,1\},$
  and $\Lambda_A$ denotes the shift space of the two-sided topological Markov shift  $(\Lambda_A,\sigma_A)$
defined by the matrix $A$ as in \eqref{eq:LambdaA}. 
Let $I_N$ be the $N\times N$ identity matrix.
Then the pair $(I_N, A)$ naturally yields a symbolic matrix system
and hence a symbolic matrix bisystem
whose $\lambda$-graph bisystem
is denoted by
$({\frak L}_A^-,{\frak L}_A^+).$
Then we have 
\begin{corollary}\label{cor:duality}
The $C^*$-algebra ${\mathcal{O}}^+_{{{\frak L}_A^-}}$
is isomorphic to the Cuntz--Krieger algebra $\OA$, whereas 
the other $C^*$-algebra
${\mathcal{O}}^-_{{{\frak L}_A^+}}$
is isomorphic to the $C^*$-algebra of  the crossed product 
$C(\Lambda_A) \rtimes_{\sigma_A^*}\Z$
of the commutative $C^*$-algebra  
$C(\Lambda_A) $ of complex valued continuous functions 
on the two-sided shift space $\Lambda_A$ by the automorphism induced by the homeomorphism $\sigma_A$
of the shift on $\Lambda_A.$
\end{corollary}
Let $\LGBS$ be a $\lambda$-graph bisystem.
As seen in the construction of the $C^*$-algebra
 $\OALMP$, the $C^*$-subalgebra 
$\ALM$ generated by the projections $E_i^{l-}(\xi), \xi \in F(v_i^l)$
is isomorphic to the commutative $C^*$-algebra 
$C(\Omega_{{\frak L}^-})$  whose character space  
$\Omega_{{\frak L}^-}$ consists of infinite labeled paths of the labeled 
Bratteli diagram  ${\frak L}^-$.
Since the matrix
$A_{l,l+1}^+(i,\alpha,j)$ for $v_i^l \in V_l, v_j^{l+1}\in V_{l+1}, \alpha \in \Sigma^+$
in the operator relation \eqref{eq:L6}
is the structure matrix of the other 
Bratteli diagram  ${\frak L}^+$,
the operator relation \eqref{eq:L6}
of  $\OALMP$ tells us that the Bratteli diagram
${\frak L}^+$ ``acts'' on $\Omega_{{\frak L}^-}$.
We actually see that $\rho_\alpha$ for $\alpha\in \Sigma^+$ defined in 
\eqref{eq:defofrhoalphaA}
gives rise to an endomorphism on 
$C(\Omega_{{\frak L}^-})$.
This means that the $C^*$-algebra 
$\OALMP$ may be regarded as the one constructed from an action
of ${\frak L}^+$ onto ${\frak L}^-$.
From this view point, the other 
$C^*$-algebra 
$\OALPM$ may be regarded as the one constructed from an action
of ${\frak L}^-$ onto ${\frak L}^+$.
This observation says that the two $C^*$-algebras 
$\OALMP$ and $\OALPM$
are obtained from the labeled Bratteli diagrams 
 ${\frak L}^-$ and ${\frak L}^+$
by exchanging its roles of the action and the space, respectively.
In this sense, one may consider that the 
two $C^*$-algebras 
$\OALMP$ and $\OALPM$ have a ``duality'' to each other. 
Therefore Corollary \ref{cor:duality} shows us that the pair of the Cuntz--Krieger algebra 
$\OA$ and the crossed product $C^*$-algebra $C(\Lambda_A) \rtimes_{\sigma_A^*}\Z$
is regarded as a ``duality'' pair.

More precisely, our definition of ``duality'' in this setting is the following.
\begin{definition}
Let $\mathcal{G}_1, \mathcal{G}_2$ be amenable and \'etale groupoids such that 
their unit spaces $\mathcal{G}_1^{(0)}, \mathcal{G}_2^{(0)}$
are totally disconnected compact Hausdorff spaces.
The pair 
$(C^*(\mathcal{G}_1), C(\mathcal{G}_1^{(0)}))$
and 
$(C^*(\mathcal{G}_2), C(\mathcal{G}_2^{(0)}))$
 of the $C^*$-algebras of the \'etale groupoids $\mathcal{G}_1, \mathcal{G}_2$
and their commutative $C^*$-subalgebras of its diagonals 
 $C(\mathcal{G}_1^{(0)})$ and 
$C(\mathcal{G}_2^{(0)})$
is said to be {\it a duality pair}\/
if there exists a $\lambda$-graph bisystem 
$\LGBS$ such that there exist isomorphisms
$\Phi_1: C^*(\mathcal{G}_1) \longrightarrow \OALMP$
and
$\Phi_2: C^*(\mathcal{G}_2) \longrightarrow \OALPM$
such that 
$\Phi_1( C(\mathcal{G}_1^{(0)}) ) = C(X_{{\frak L}^-}^+)$
and
$\Phi_2( C(\mathcal{G}_2^{(0)}) ) = C(X_{{\frak L}^+}^-)$,
that is
\begin{equation}
(C^*(\mathcal{G}_1), C(\mathcal{G}_1^{(0)}) )= (\OALMP, \DLMP),  \qquad
(C^*(\mathcal{G}_2), C(\mathcal{G}_2^{(0)}) )= (\OALPM, \DLPM). \label{eq:groupoidduality}
\end{equation}
In other words, the condition \eqref{eq:groupoidduality} is equivalent to
the condition that the groupoids  
 $\mathcal{G}_1, \mathcal{G}_2$ are isomorphic to 
 $\mathcal{G}_{{\frak L}^-}^+, \mathcal{G}_{{\frak L}^+}^-$
as \'etale groupoids, respectively
when the groupoids 
$\mathcal{G}_1, \mathcal{G}_2, \mathcal{G}_{{\frak L}^-}^+, \mathcal{G}_{{\frak L}^+}^-$
are all essentially principal (Renault \cite[Proposition 4.11]{Renault2}).
\end{definition}
Let $A$ be an $N\times N$ irreducible non-permutation matrix over $\{0,1\}.$
Let $\DA$ 
be the commutative $C^*$-subalgebra of diagonal elements of the canonical AF-algebra 
inside the Cuntz--Krieger algebra $\OA$.
The subalgebra $\DA$ 
is isomorphic to the commutative $C^*$-algebra $C(\Lambda_A^+)$
of the right one-sided shift space $\Lambda_A^+$ of $\Lambda_A$.
As a result,  we have the following corollary of Proposition \ref{prop:lambdaalgebras}
by the discussion of this section. 
\begin{corollary}
The pair $(\OA, \DA)$ and $(C(\Lambda_A) \rtimes_{\sigma_A^*}\Z, C(\Lambda_A))$
is a duality pair in the above sense. 
\end{corollary}

\medskip

{\it Acknowledgments:}
The author would like to deeply thank the referee
for many helpful comments and suggestions in the presentation of the paper. 
This work was  supported by JSPS KAKENHI Grant Numbers 15K04896, 19K03537.




\end{document}